\documentclass[hidelinks,onefignum,onetabnum]{siamart220329}

\usepackage{lipsum}
\usepackage{amsfonts}
\usepackage{graphicx}
\usepackage{epstopdf}
\usepackage{algorithmic}
\usepackage{cancel}
\usepackage{multirow}
\ifpdf
  \DeclareGraphicsExtensions{.eps,.pdf,.png,.jpg}
\else
  \DeclareGraphicsExtensions{.eps}
\fi

\newcommand{\R}{\mathbb{R}}
\def\bl{\boldsymbol \lambda}
\def\blk{\boldsymbol \lambda^{(k)}}
\def\bu{\mathbf u}
\def\bb{\mathbf b}

\def\bz{\mathbf z}
\def\b0{\mathbf 0}

\def\bi{\boldsymbol{i}}

\def\Veta{\boldsymbol{\lambda}}
\def\Vpsi{\boldsymbol{\psi}}

\def\FZcomment[#1]{\textcolor{black}{#1}}
\def\GLcomment[#1]{\textcolor{red}{#1}}
\def\HIcomment[#1]{\textcolor{black}{#1}}

\newcommand{\name}{{UPenMM}}  

\newsiamremark{remark}{Remark}
\newsiamremark{hypothesis}{Hypothesis}
\crefname{hypothesis}{Hypothesis}{Hypotheses}
\newsiamthm{claim}{Claim}

\title{Uniform multi-penalty regularization for linear ill-posed inverse problems\thanks{Submitted to the editors DATE.
\funding{This work was partially supported by the Istituto Nazionale di Alta Matematica, Gruppo
Nazionale per il Calcolo Scientifico (INdAM-GNCS).}}}

\author{Villiam Bortolotti\thanks{Department of Civil, Chemical, Environmental, and Materials Engineering, University of Bologna, Via Terracini, 28,  40131 Bologna, Italy (\texttt{villiam.bortolotti@unibo.it})}
 \and  Germana Landi\thanks{Department of Mathematics, University of Bologna,
  Piazza di Porta S.~Donato~5, 40126 Bologna, Italy
  (\texttt{germana.landi@unibo.it}).}  
\and Fabiana Zama
\thanks{Department of Mathematics, University of Bologna,
  Piazza di Porta S.~Donato~5, 40126 Bologna, Italy
  (\texttt{fabiana.zama@unibo.it}).}
}

\usepackage{amsopn}

\ifpdf
\hypersetup{
  pdftitle={point-wise multi-penalty},
  pdfauthor={G. Landi, V Bortolotti, and F. Zama}
}
\fi

\begin{document}

\maketitle

\begin{abstract}
This study examines, in the framework of variational regularization methods, a multi-penalty regularization approach which builds upon the Uniform PENalty (UPEN) method, previously proposed by the authors for Nuclear Magnetic Resonance  data processing. The paper introduces two iterative methods, UpenMM and GUpenMM, formulated within the Majorization-Minimization  framework. These methods are designed to identify appropriate regularization parameters and solutions for linear inverse problems utilizing multi-penalty regularization. The paper demonstrates the convergence of these methods and illustrates their potential through numerical examples, showing the practical utility of point-wise regularization terms in solving various inverse problems.
\end{abstract}

%
%
\section{Introduction\label{sec:intro}}
Variational regularization methods form the foundation for addressing a wide range of linear inverse ill-posed problems that can be expressed as:
\begin{equation*}
  A\bu + \mathbf{e} = \bb
\end{equation*}
where $\bb \in\mathbb{R}^M$ represents noisy data, $\mathbf{e}\in\mathbb{R}^M$ is a Gaussian noise vector, $A\in \mathbb{R}^{M \times N}$, $M\geq N$, is a linear operator and $\bu\in\mathbb{R}^N$ is the exact object to be recovered (see for instance  \cite{ito2014inverse} and \cite{engl1996regularization} for a survey). In the following, we assume $A$ to be full rank. Variational methods involve minimizing an objective functional that comprises a data fidelity term,
denoted as $\phi(\bu)$, and a regularization term, denoted as $\mathcal{R}_{\lambda}(\bu)$, i.e.:
$$\min\limits_{\bu\in \Omega} \;  \{ \phi(\bu) + \mathcal{R}_{\lambda}(\bu) \},$$
where $\Omega$ is a compact subset of $\mathbb{R}^N$ used to 
impose some physical constraints, such as, for example, non-negativity.
The data fidelity term reflects the problem formulation and accounts for the uncertainty in the measured data. On the other hand, the regularization term incorporates prior knowledge about the unknown solution, such as its smoothness or sparsity. The impact of the regularization term is governed by a scalar value $\lambda$ called the regularization parameter. 

To capture diverse and contrasting characteristics of solutions, we investigate  a multi-penalty regularization approach setting 
\begin{equation}\label{eq:generalR}
  \mathcal{R}_\lambda(\bu) = \sum_{i=i}^p \lambda_i \psi_i(\bu)
\end{equation}
where $\psi_i(\bu)$, $i=1, \ldots, p$ are regularization functions
and we propose a convergent numerical method to compute both $\bu$ and $\lambda_i$. In our experiments, we focus on  the following point-wise penalties 
\begin{equation}\label{eq:upen}
\psi_i(\bu)=(L\bu)_i^2+ \epsilon, \quad i=1,\ldots,N, \quad p=N
\end{equation}
where $L\in\mathbb{R}^{N \times N}$ represents the discretization of the second-order derivative operator \HIcomment[with periodic boundary conditions,] and $\epsilon$ is a positive constant. %
This approach was successfully proposed for the first time in 1998 in the context of one-dimen\-sio\-nal Nuclear Magnetic Resonance (NMR) data processing, together with a heuristic numerical procedure for its realization, named Uniform PENalty (UPEN) method (Borgia et al. \cite{borgia1998uniform, borgia2000uniform}). 
The high quality of the results obtained, especially in the study of porous media properties, has been widely recognized by the scientific community, leading to the conference series Magnetic Resonance in Porous Media,  started in Bologna in 1990 and yielding in 2012 the UPenWin software tool distributed by the University of Bologna (\url{https://site.unibo.it/softwaredicam/en/software/upenwin}). 
In 2016, an extension to the two-dimensional case was developed, beginning the formalization of UPEN within the mathematical framework of inverse problems \cite{bortolotti2016uniform}. The first formulation of the Uniform Penalty Principle was introduced, with the proof that solutions conforming to this principle are regularized solutions of the data fitting model. Subsequent extensions have considered the case of multi-penalty regularization including $L_1$-based terms and more software tools \cite{bortolotti2019upen2dtool,bortolotti2022mupen2dtool} were created to tackle problems in the field of NMR relaxometry.
What has been missing until now is a formal analysis of the numerical method to prove its convergence properties. This is the main contribution of the present work. To obtain our main result, initially, contributions previously developed in the mathematical field of inverse problems are leveraged, particularly from the results of Ito et al. that, starting in 2011, introduce and analyze the Balancing Principle (BP) for multi-penalty regularization. 
This principle, analyzed in \cite{ito2011multi} and \cite{ito2014multi} when $p=2$, is extended here to the case of point-wise penalties \eqref{eq:upen} and box constraints, and the solution and regularization parameters are characterized as critical points of an Augmented Tikhonov functional. In this way, the Uniform Penalty Principle is recognized as a particular case of the generalized Balancing Principle, and the properties of the resulting variational model are demonstrated. Then the heuristic numerical procedure UPEN is framed into the class of the Majorization-Minimization (MM) methods. This analysis allows us to establish the  convergence of UPEN to the solution of the variational model, satisfying the Uniform Penalty Principle.

Multi-penalty regularization has gained growing interest in recent literature, but most of the authors consider the case of two-penalty terms. We remark that the famous elastic regression used in Statistics \cite{zou_hastie_2005} is a multi-penalty regularization method combining the $L_1$ and $L_2$ penalties of the Lasso \cite{tibshirani_1996} and Ridge \cite{hoerl_kennard_1970} methods. Moreover, most of the literature is devoted to the problem of developing a suitable parameter choice rule.
Lu, Pereverzev et al. \cite{lu2011multi,lu2010discrepancy}  thoroughly investigate the case of two $L_2$-based terms and present a modified Discrepancy Principle for the computation of the two regularization parameters together with its numerical implementation.
The parameter choice issue is also discussed in \cite{belge2002efficient}, where a multi-parameter generalization of the  L-curve criterium is proposed, and in \cite{brezinski2003multi}, where an approach based on the Generalized Cross Validation method is suggested.
Reichel and Gazzola \cite{gazzola2016new} consider regularization terms of the form
\begin{equation}\label{}
  \psi_i(\bu)=\|D_i\bu\|^2, \qquad i=1, \ldots, p
\end{equation}
where \FZcomment[$D_i\in \mathbb{R}^{d_i\times N}$ ($d_i \leq N$)] are proper regularization matrices and describe an approach to determine the regularization parameters based on the Discrepancy Principle focusing on the case $p=2$.
Fornasier et al. \cite{fornasier2014parameter} propose a modified version of the discrepancy principle for multi-penalty regularization and provide a theoretical justification.
Multi-penalty regularization is also considered for unmixing problems with two penalty terms based on $L_q$ and $L_p$ norms, $0\leq q<2$ and $2\leq p<\infty$, in \cite{naumova2014minimization,naumova2013multi, kereta2019linear,kereta2021computational} and in \cite{grasmair2020adaptive}, where the penalty terms are based on the $L_1$ and $L_2$ norms.
Finally, in \cite{wang2013multi}, two-penalty regularization, with $L_0$ and $L_2$ penalty terms, is considered for nonlinear ill-posed problems and its regularizing properties are analyzed. 
Ito et. al. in \cite{ito2011multi} and subsequent papers   \cite{Ito2011_tik} and \cite{ito2014multi}, 
introduce the  Balancing Principle to handle multiple-penalty regularization and present fixed-point-like algorithms for the case of two penalties.

\paragraph{Contributions}\hfil \\
Our paper makes the following contributions:
\begin{itemize}
    \item Proposal of two iterative methods, named UpenMM and GUpenMM, developed within the framework of MM methods, to compute the regularization parameters and solutions of linear inverse problems using multi-penalty regularization.
    GUpenMM is a general scheme incorporating heuristic rules successfully applied in NMR data processing.
    \item Proof of the convergence of UpenMM and GUPenMM to the solution and regularization parameters satisfying the Uniform Penalty Principle.
    \item Presentation of numerical examples demonstrating the robustness and effectiveness of the point-wise regularization terms \eqref{eq:upen} in solving general inverse problems in various applications.
\end{itemize}
The remainder of this paper is organized as follows. Section~\ref{sec:model} presents the Balancing and Uniform Penalty principles and their properties.  Section~\ref{sec:method}
presents UpenMM and GUpenMM for the solution of the proposed  multi-penalty regularization problem and proves its convergence.
The results of a two-dimensional numerical experiment are presented in Section~\ref{sec:experiments}, while the one-dimensional tests are discussed in the supplementary materials. Finally, the conclusions are provided in Section~\ref{sec:conclusions}.
%
\section{Derivation of the uniform  multi-penalty regularization}\label{sec:model}
In this section, we present a generalization of the Balancing Principle, originally introduced in \cite{Ito2011_tik} and  initially formulated for a single penalty term. This principle has been further extended to accommodate multiple penalties in subsequent works \cite{ito2011multi,ito2014multi}, with a particular focus on the case of two penalty terms.
Here, we provide a detailed analysis of multipenalty regularization specifically for the case where there are more than two penalty terms, for example representing point-wise second-order smoothing terms as in \eqref{eq:upen}.

We first present the regularization problem in a Bayesian framework, providing a statistical interpretation to the choice of regularization parameters and to the structure of the Augmented Tikhonov problem \eqref{eq:Atikh}. %
Theorem \ref{teo:J} and Corollary \ref{cor:cor1} allow us to establish the existence of the minimum of the Augmented Tikhonov functional. Then,
introducing the value function \eqref{eq:val_f}, we show in Theorem \ref{teo1} 
that the regularization parameters are critical points of a function $\Phi_{\gamma}(\Veta)$. 
This theorem links theoretical insights to practical computation, showing how to compute the regularization parameters that balance data fidelity with regularization terms effectively.

To start, we introduce the vector of the penalty parameters $\Veta\in \mathbb{R}^p$ and the vector penalty function $\Vpsi: \mathbb{R}^N \rightarrow \mathbb{R}^p$ such that
\begin{gather*}
  \Veta = (\lambda_1, \lambda_2, \ldots, \lambda_p)^T, \quad \lambda_i >0, \; i=1,\ldots,p  \\
  \Vpsi(\bu) = (\psi_1(\bu), \ldots, \psi_p(\bu))^T, \quad \psi_i(\bu) \geq \epsilon, \; i=1,\ldots,p
\end{gather*}
where $\psi_i$ are positive continuous functions and $\epsilon$ is a small positive constant.
%
Assuming the experimental data $\bb$ to be corrupted by Gaussian white noise, we consider the following data-fit  function
\begin{equation*}
  \phi(\bu) = \frac{1}{2}\| A\bu-\bb \|^2 .
\end{equation*}
By employing a hierarchical Bayesian approach \cite{gelman2013bayesian},
the information about the unknown solution is contained in the Posterior Probability Density Function  $P(\bu, \Veta, \tau | \bb)$ where  $\tau=1/\sigma^2$ is the inverse variance of data noise. Extending the approach in \cite{jin2008augmented} by applying it to $p$ penalty parameters $\lambda_i$ and regularization functions $\psi_i$,
we consider the  critical points
of  the following Augmented-Tikhonov problem:
\begin{multline}\label{eq:Atikh}
\min_{(\bu, \Veta, \tau)} \; \left \{\frac{\tau}{2} \| A \bu - \bb \|^2 + \Veta^T \Vpsi(\bu) +
\left [\beta_0 \tau - \left (\frac{M}{2} + \alpha_0-1 \right )\ln(\tau) \right ]+ \right.
\\
\left. \sum_{i=1}^p \left [\beta_i \lambda_i - \left ( \frac{N}{2} + \alpha_i -1 \right ) \ln(\lambda_i) \right ]\right \}
\end{multline}
%
where $(\alpha_i, \beta_i)$, with $i\geq 1$, and $(\alpha_0, \beta_0)$ are parameter pairs representing the prior distributions. Following \cite{jin2008augmented}, these distributions are assumed to be Gamma distributions, and they correspond to the hyperparameters $\lambda_i$ and $\tau$ respectively.
%

In the following theorem, we prove that \eqref{eq:Atikh} is well-posed.
%
\begin{theorem}\label{teo:J}
There exists at least one minimizer to the Augmented-Tikhonov function $\mathcal{J}$:
\begin{equation}\label{eq:J}
\mathcal{J}(\bu, \Veta, \tau) \equiv \frac{\tau}{2} \| A \bu - \bb \|^2 + \Veta^T \Vpsi(\bu) + \left [\beta_0 \tau - \alpha_0'\ln(\tau) \right ]+
\sum_{i=1}^p \left [\beta_i \lambda_i - \alpha_i' \ln(\lambda_i) \right ]
\end{equation}
with $\alpha_0'=\left (\frac{M}{2} + \alpha_0-1 \right )$, 
$\alpha_i'=\left (\frac{p}{2} + \alpha_i-1 \right )$, $i=0, \ldots, p$. 
\end{theorem}
\begin{proof}
The functional $\mathcal{J}(\bu, \Veta, \tau) $ is continuous on $\mathbb{R}^{N} \times  \mathcal{V} \times (0, \infty)$
where
$\mathcal{V}= \left \{ \mathbf{z} \in \mathcal{R}^{p}: z_i > 0, \ i=1, \ldots, p\right \}$. Therefore it is sufficient to prove that it is bounded from below and it is coercive, i.e. it tends to infinity when approaching the boundary of $\mathbb{R}^{N} \times  \mathcal{V} \times (0, \infty)$.
To prove that  $\mathcal{J}(\bu, \Veta, \tau)$ is bounded from below, we observe that
$$
\mathcal{J}(\bu, \Veta, \tau) \geq \left [\beta_0 \tau - \alpha_0'\ln(\tau) \right ]+
\sum_{i=1}^p \left [\HIcomment[(\beta_i+\epsilon)] \lambda_i - \alpha_i' \ln(\lambda_i) \right ] \equiv \nu(\tau)+ \zeta(\Veta) .
$$
We prove  that $\zeta(\Veta)$ is convex on $\mathcal{V}$. Let $\mathcal{H}$ be  the  Hessian matrix of $\zeta(\Veta)$,
we observe that it is diagonal and positive definite
 i.e.
$$ \mathcal{H}_{k,j}=\frac{\partial^2}{\partial \lambda_k \partial \lambda_j} \zeta(\Veta) = \left \{
\begin{aligned}
& \frac{\alpha_k'}{\lambda_k^2} > 0, & k=j; \\
& 0, & k \neq j;
\end{aligned} \right .
$$
therefore  the critical point $\hat \Veta$ of $\zeta$:
$$\frac{\partial \zeta}{\partial \lambda_i}(\hat \Veta) =0, \quad \text{i.e.} \quad \hat \lambda_i = \frac{\alpha_i'}{\HIcomment[\beta_i+\epsilon]} >0, \ i=1, \ldots, p$$
defines its minimum, i.e. $\zeta(\Veta) \geq \zeta(\hat \Veta)$, $\forall \Veta \neq \hat \Veta$.
Similarly, $\nu(\tau)$ has a minimum $\hat \tau = \frac{\alpha_0'}{\beta_0} >0$  and
$$ \nu(\tau) \geq \alpha_0'- \alpha_0' \ln(\frac{\alpha_0'}{\beta_0}).$$
Therefore $\mathcal{J}(\bu, \Veta, \tau)$ is bounded from below.

Next, we show that $\mathcal{J}$ is coercive.
First we observe that:
$$
\lim_{\tau \rightarrow 0^+}\nu(\tau)  = +\infty, \quad  \lim_{\tau \rightarrow +\infty }\nu(\tau)= + \infty.
$$
Moreover, since
$$\|\Veta \|\to 0, \quad \text{iff} \quad \lambda_i  \to 0, \quad \forall i$$
and
$$\|\Veta \|\to \infty, \quad \text{iff} \quad \exists i, \; 1 \leq i \leq p  \; \text{s.t.} \; \lambda_i  \to \infty$$
we have
$$
 \lim_{\|\Veta \|\rightarrow 0^+}\zeta(\Veta) = + \infty \quad \text{and} \quad \lim_{\|\Veta\| \rightarrow +\infty} \zeta(\Veta) = + \infty.
$$
Let $c_0$ be $\nu(\hat \tau)+ \zeta(\hat \Veta) $, 
then
\HIcomment[%
$$
\mathcal{J}(\bu, \Veta, \tau) - c_0  \geq   \frac{\tau}{2} \| A \bu - \bb \|^2   . 
$$]
%
Since the quadratic form $\| A \bu - \bb \|^2$ is positive definite, it is coercive and consequently, $\mathcal{J}$ is coercive, i.e:
%
%
$$\lim_{\|\bu\| \rightarrow \infty} \mathcal{J}(\bu, \Veta, \tau)  = +\infty$$
and it has a finite minimizer on $\mathbb{R}^N \times  \mathcal{V} \times (0, \infty)$.
\end{proof}
\begin{corollary}\label{cor:cor1}
Let $ \Omega \subset \R^N$ be the compact subset of admissible solutions $\bu$, then there exists at least one minimizer of the functional $\mathcal{J}$ in \eqref{eq:J}, where $\bu \in \Omega$, $\tau >0$ and $\lambda_i>0$, $\forall i$.
\begin{proof}
Let  $ \bi_{\Omega}(\bu)$ be the   indicator function of 
$\Omega$, i.e.
$$ \bi_{\Omega}(\bu) = \left \{
\begin{array}{ll} 
0 & \bu \in \Omega; \\
\infty & \bu \notin \Omega;
\end{array} \right .
$$
through the functional $\hat{\mathcal{J}}$:
$\hat{\mathcal{J}} \equiv \mathcal{J}+ \bi_{\Omega}(\bu)$
we can view the  minimization of $\mathcal{J}$ where $\bu \in \Omega$ as the minimization of $\hat{\mathcal{J}}$ where $\bu \in \R^N$, i.e.:
$$ \min_{\bu \in \Omega} \mathcal{J} = \min_{\bu \in \R^N} \hat{\mathcal{J}} .$$
The functional $\hat{\mathcal{J}}(\bu, \Veta, \tau) $ is continuous on $\mathbb{R}^{N} \times  \mathcal{V} \times (0, \infty)$
where
$$\mathcal{V}= \left \{ \mathbf{z} \in \mathcal{R}^{p}: z_i > 0, \ i=1, \ldots, p\right \}.$$
Observing that 
$\hat{\mathcal{J}}  > \mathcal{J}$
we can proceed  as in Theorem \ref{teo:J} 
and prove that $\hat{\mathcal{J}}$ is continuous, bounded from below and coercive. Hence there is at least one minimizer in $\Omega \times  \mathcal{V} \times (0, \infty)$.
\end{proof}
\end{corollary}
%
In the hierarchical Bayesian approach, the hyperparameter $\tau$ serves as an estimator of the inverse of the noise variance $\sigma^2$  \cite{jin2008augmented}.
The first-order optimality conditions of the augmented Tikhonov functional $\mathcal{J}(\bu, \Veta, \tau)$ yield the following expression:
$$
\frac{\partial \mathcal{J}}{\partial \tau} (\bu, \Veta, \tau)   =   \frac{1}{2}
\| A \bu - \bb \|^2 + \beta_0 - \frac{\alpha'_0}{\tau} =0 .
$$
By substituting $\alpha_0' = \left(\frac{M}{2} + \alpha_0 - 1\right)$, we can derive the relationship for $\sigma^2$:
\begin{equation}\label{eq:sigma}
    \sigma^2= \frac{1}{2}\frac{\| A \bu - \bb \|^2 + 2 \beta_0}{\frac{M}{2} + \alpha_0-1}.
\end{equation}
Analogously to the formulation in \cite{ito2014multi} (Eq. 2.1), we express the critical points of $\mathcal{J}(\bu, \Veta, \tau)$ under the assumption of a noninformative prior, where \HIcomment[$\beta_0 \approx 0$ and $\beta_i=0$, which is commonly adopted in practice \cite{jin2008augmented,ito2014inverse}.] The critical points can be written as:
$$ \left \{
\begin{aligned}
\bu_{\Veta} = & \arg\min_{\bu} \left\{\frac{1}{2}\| A \bu - \bb \|^2 + \frac{1}{\hat \tau} \hat \Veta^T \Vpsi(\bu)   \right\}; \\
\hat \lambda_i = & \frac{\alpha_i}{\psi_i(\bu_{\Veta})}, \; i=1,\ldots, p; \\
\hat \tau = & \frac{\alpha_0}{\| A \bu_{\Veta} - \bb \|^2 + \HIcomment[\beta_0]} .
\end{aligned} \right . .
$$
Here, we assume $\alpha_i = \alpha$, and by setting $\gamma = \alpha_0/\alpha$ and $\Veta = \hat{\Veta} / \hat{\tau}$, \HIcomment[%
we obtain 
\begin{equation}\label{eq:BPp_aTikh}
\left \{
\begin{aligned}
\bu_{\Veta} = & \arg\min_{\bu} \left\{\frac{1}{2}\| A \bu - \bb \|^2 +  \Veta^T \Vpsi(\bu)   \right\} ; \\
\lambda_i = & \frac{\| A \bu_{\Veta} - \bb \|^2+\HIcomment[\beta_0]}{ \gamma \psi_i(\bu_{\Veta})}, \; \; i=1,\ldots, p . \\
\end{aligned} \right .
\end{equation}
Following \cite{Ito2011_tik} and \cite{ito2014multi}, the multi-parameter Balancing Principle is derived by setting $\beta_0=0$ in \eqref{eq:BPp_aTikh}:
]
\begin{equation}\label{eq:BPp}
\left \{
\begin{aligned}
\bu_{\Veta} = & \arg\min_{\bu} \left\{\frac{1}{2}\| A \bu - \bb \|^2 +  \Veta^T \Vpsi(\bu)   \right\} ; \\
\lambda_i = & \frac{\| A \bu_{\Veta} - \bb \|^2}{ \gamma \psi_i(\bu_{\Veta})}, \; \; i=1,\ldots, p . \\
\end{aligned} \right .
\end{equation}
In this formulation, the optimal regularization parameter $\Veta$ is obtained by balancing the least-squares fidelity term $\tfrac{1}{2}\| A \bu - \bb \|^2$ with the penalty functions $\psi_i$.
\begin{remark} \label{rmk:2}
It is worth noting that when the  parameter $\alpha_0$ is set to $1$ and \HIcomment[$\beta_0\approx0$], the expression \eqref{eq:sigma} simplifies further:
$$\sigma^2\HIcomment[\approx]\frac{\| A \bu - \bb \|^2 }{M}.$$
\HIcomment[This result indicates that the mean residual norm, obtained with the Balancing Principle, can be a reliable estimate of the noise variance.
The expression relates to the Morozov discrepancy principle \cite{Mor66}, where the residual norm is used to tune the regularization parameter to the noise level, thus aligning the model fidelity with the observed data's noise characteristics.]
\end{remark}
Let $\Omega$ be a compact subset of $\mathbb{R}^N$; following \cite{ito2011multi}, we
define the value function as
\begin{equation}
F(\Veta) = \min_{\bu \in \Omega} \; \{ \phi(\bu) + {\Veta}^T \Vpsi(\bu) \}.
\label{eq:val_f}
\end{equation}
In \cite{ito2011multi}, it is proved that $F(\Veta)$ is a concave and differentiable function whose derivatives, with respect to $\lambda_j$, are given by:
\begin{equation}\label{eq:dFdlam}
  \frac{\partial F}{\partial \lambda_j} = \psi_j (\bu_{\Veta}), \quad j=1,\ldots,p
\end{equation}
where
\begin{equation*}
  \bu_{\bl} = \arg\min_{\bu \in \Omega} \; \{ \phi(\bu) + {\Veta}^T \Vpsi(\bu) \}.
\end{equation*}
In the following theorem, 
we present the variational characterization of the Balancing Principle \eqref{eq:BPp}, 
by extending 
the result shown in \cite{ito2014multi}  for the case of $p=2$ to the case $p\gg 2$.
\def\us{\mathbf{u}_{{\boldsymbol{\lambda}}^*}}
\begin{theorem}\label{teo1}
    Let $\gamma$ be a positive constant, and let us define
    \begin{equation}\label{eq:phig}
    \Phi_\gamma(\Veta) \equiv \frac{F^{\gamma + p}(\Veta)}{\prod_{i=1}^p\lambda_i}
    \end{equation}
with $F(\Veta)$ given  in equation \eqref{eq:val_f}. Then, the function $\Phi_\gamma(\bl):\mathbb{R}^p \rightarrow \mathbb{R}$ is continuous, differentiable and
the vector $\Veta^* \in \mathbb{R}^p$ with components
\begin{equation}\label{eq:teo1}
\lambda_j^* = \frac{\phi(\us)}{\gamma \psi_j(\us)}, \quad j=1, \ldots, p
\end{equation}
where
\begin{equation*}
  \us = \arg\min_{\bu \in \Omega} \; \{ \phi(\bu) + {\Veta^*}^T \Vpsi(\bu) \}
\end{equation*}
is a critical point of  $\Phi_\gamma(\Veta)$.
%
%
\end{theorem}
\begin{proof}
Continuity and differentiability of $\Phi_\gamma(\bl)$ follow from continuity and differentiability of the value function $F(\bl)$ which have been proved in \cite{ito2011multi} (cfr. Lemma 2.1 and Theorem 2.2).
Since $F$ is differentiable, the critical point $\Veta^*$ is defined by
\begin{equation}\label{eq:crp}
 \frac{\partial \Phi_\gamma(\Veta^*)}{\partial \lambda_j}=0, \qquad j=1, \ldots, p.
 \end{equation}
Computing the derivatives:
$$
\begin{aligned}
\frac{\partial \Phi_\gamma(\Veta)}{\partial \lambda_j}  = & \frac{(\gamma + p) F^{\gamma + p - 1}(\Veta) \cdot \frac{\partial F(\Veta)}{\partial \lambda_j} - F^{\gamma + p}(\Veta) \cdot \frac{1}{\lambda_j}}{\left(\prod_{i=1}^p \lambda_i\right)} \\[4pt]
= & \frac{F^{\gamma + p - 1}(\Veta) }{\left(\prod_{i=1}^p \lambda_i\right)} \left [  (\gamma + p) \frac{\partial F(\Veta)}{\partial \lambda_j} - \frac{F(\Veta) }{\lambda_j} \right ]
\end{aligned}
$$
and substituting the expression \eqref{eq:dFdlam},
%
we obtain that the derivatives are null when the parameters $\bl_j$ satisfy the following linear equations:
\begin{equation}\label{eq:lineqs}
 (\gamma + p) \lambda_j \cdot \psi_j (\bu_{\Veta}) - F(\Veta)=0, \qquad j=1, \ldots, p.
 \end{equation}
%
Substituting the expression for $F(\Veta)$ into \eqref{eq:lineqs}, we obtain the following linear system:
%
\begin{equation}\label{eq:16}
(\gamma + p) \lambda_j \cdot \psi_j (\bu_{\Veta}) -\sum_{i=1}^p \lambda_i \psi_i (\bu_{\Veta}) = \phi(\bu_{\Veta}), \quad j=1, \ldots, p
\end{equation}
which can be written in matrix form as
\begin{equation}\label{eq:sys}
 \left (\mathbf{D} - \mathbf{1}\Vpsi(\bu_{\Veta})^T \right ) \Veta =\phi(\bu_{\Veta}) \mathbf{1}
 \end{equation}
where $\mathbf{D}$ is the diagonal matrix of order $p$ with non-zero elements $D_{i,i}=(\gamma+p) \psi_i (\bu_{\Veta})$,
and $\boldsymbol{1} \in \mathbb{R}^{p}$ is a column vector with elements equal to one.
We observe that the linear system has a matrix which is the sum of the  diagonal $\mathbf{D}$ and the rank-one update $-\mathbf{1}\Vpsi(\bu_{\bl})^T$. %
Since $\mathbf{D}$ is non singular because $\psi_i (\bu_{\Veta}) > 0$ for all $i$ and
\begin{equation*}
  1 + \mathbf{1}^T D^{-1}\Vpsi(\bu_{\Veta}) = 1 - \frac{1}{\gamma+p}\sum_{i=1}^p 1 = 1 - \frac{p}{p+\gamma}= \frac{\gamma}{p+\gamma} \neq 0 ,
\end{equation*}
we can apply the Shermann-Morrison formula \cite{sherman1950adjustment}  and obtain (see details in Appendix \ref{App_A}):
$$
\begin{aligned}
& \left (\mathbf{D} - \mathbf{1}\Vpsi(\bu_{\Veta})^T \right )^{-1} &= & 
\frac{1}{p+\gamma} \left [ \begin{pmatrix}
 \tfrac{1}{\psi_1(\bu_{\Veta})}  & \\
 0 & \tfrac{1}{\psi_2(\bu_{\Veta})}  & \\
 0 & 0 & \ddots & \\
  0 & \cdots & 0 & \tfrac{1}{\psi_p(\bu_{\Veta})}
  \end{pmatrix}  \right . \\
  &&& + \left . \frac{1}{\gamma}  \begin{pmatrix}
 \tfrac{1}{\psi_1(\bu_{\Veta})} & \tfrac{1}{\psi_1(\bu_{\Veta})}& \cdots & \tfrac{1}{\psi_1(\bu_{\Veta})} \\
\tfrac{1}{\psi_2(\bu_{\Veta})} & \tfrac{1}{\psi_2(\bu_{\Veta})}& \cdots & \tfrac{1}{\psi_2(\bu_{\Veta})} \\
\vdots & \vdots & \ddots & \vdots \\
\tfrac{1}{\psi_p(\bu_{\bl})} & \tfrac{1}{\psi_p(\bu_{\Veta})} & \cdots &  \tfrac{1}{\psi_p(\bu_{\Veta})}
\end{pmatrix}  \right ].
\end{aligned}
$$
Hence the solution of the linear system \eqref{eq:sys}  is:
\begin{equation*}
  \Veta^* = \left (\mathbf{D} - \mathbf{1}\Vpsi(\us)^T \right )^{-1} \phi(\us) \mathbf{1}
\end{equation*}
whose components are
\begin{equation*}
  \lambda^*_j = \phi(\us) \frac{1}{p+\gamma}\left( \frac{1}{\psi_j(\us)} + \frac{p}{\gamma \psi_j(\us)} \right) = \frac{\phi(\us)}{\gamma \psi_j(\us)},\quad j=1, \ldots, p 
\end{equation*}
which concludes the proof.
\end{proof}
%
%
This theorem shows that we can characterize the solution to the BP system as critical points of the function $\Phi_\gamma$ \eqref{eq:phig}.
The solution depends on the choice of the parameter $\gamma$; in \cite{ito2011multi} an iterative procedure has been proposed for its selection in the case $p=2$ but a general criterion for $p\gg 2$ is still missing.

The Uniform Penalty Principle introduced in \cite{borgia1998uniform, bortolotti2016uniform} suggests choosing the regularization parameters so that they satisfy the following two conditions:
\begin{description}
    \item[C1] All the penalty terms $\lambda_i \psi_i(\bu)$, $i=1, \ldots, p$, are uniform and equal to a constant value $c$, i.e.
    $$\lambda_i\psi_i(\bu) = c,\quad i=1, \ldots, p$$
    which gives  $\Veta^T \Vpsi(\bu) =  c p$.
    \item[C2] The multi-penalty term $\Veta^T \Vpsi(\bu)$ is equal to the data fidelity term:
    $$\phi(\bu)=\Veta^T \Vpsi(\bu).$$
\end{description}
Conditions \textbf{C1} and \textbf{C2} give 
$$\lambda_i = \frac{ \phi(\bu)}{p\psi_i(\bu)}, \quad i=1, \ldots, p $$
which corresponds to $\gamma=p$ in the Balancing Principle.

Considering the observed properties of the noise estimation (cfr. Remark \ref{rmk:2}), when $p=M$, the penalty terms are uniform and equal to the $M$th fraction of the noise variance. For this reason, we will refer to the nonlinear system
\begin{equation}\label{eq:UPS}
\left \{
\begin{aligned}
\bu_{\Veta} = & \arg\min_{\bu\in\Omega} \left\{\frac{1}{2}\| A \bu - \bb \|^2 +  \Veta^T \Vpsi(\bu)   \right\}; \\
\lambda_i = & \frac{\| A \bu - \bb \|^2}{ p \psi_i(\bu_{\Veta})}, \; i=1,\ldots,p;
\end{aligned} \right .
\end{equation}
as ``Uniform Penalty System'' (UPS). 
It is evident that the UPS is a particular case of the BP obtained for $\gamma=p$.
\section{The numerical method}\label{sec:method}
In order to develop a suitable iterative method for computing the solution $(\bu^*,\bl^*)$ to the UPS \eqref{eq:UPS}, we use a majorization-mi\-ni\-mi\-za\-tion  approach \cite{hunter2004tutorial} to the optimization problem
\begin{equation}\label{eq:MMpb}
  \min\limits_{\bl} \; \Phi_p(\bl) .
\end{equation}
The following theorem shows that problem \eqref{eq:MMpb} is well defined.
\begin{theorem}\label{th:coercivity}
    The function $\Phi_p(\bl):\mathbb{R}^p \rightarrow \mathbb{R}$ is coercive.
\end{theorem}
\begin{proof}
Let $\{\blk\}_{k\in\mathbb{N}}$ be a sequence such that $\|\blk\| \rightarrow \infty$.
Since $\psi_j(\bu)\geq \epsilon$ for all $j$ and $\phi(\bu)\geq 0$, we get
\begin{align*}
  \Phi_p(\blk) & =  \displaystyle{ \frac{\left[ \min\limits_{\bu\in\Omega} \; \phi(\bu)+\sum_{j=1}^p \lambda_j^{(k)}\psi_j(\bu)\right]^{2p}}{\prod_{j=1}^p \lambda_j^{(k)}} }\\
   & \geq \displaystyle{ \frac{\left[  \epsilon\sum_{j=1}^p\lambda_j^{(k)}\right]^{2p}}{\prod_{j=1}^p \lambda_j^{(k)}} }
        \geq \displaystyle{ \frac{\left[ \epsilon \max_j \lambda_j^{(k)}\right]^{2p}}{\left[ \max_j  \lambda_j^{(k)} \right]^p} }
\end{align*}
    where we have used $\sum_{j=1}^p\lambda_j^{(k)}\geq \max_j  \lambda_j^{(k)}$. Summarizing, we have
    \begin{equation*}
        \Phi_p(\blk) \geq \epsilon^{2p}\left[ \max_j  \lambda_j^{(k)}  \right]^p
    \end{equation*}
    which gives the thesis.
\end{proof}
In the following, we first revise the MM framework and then we derive our iterative scheme to address \eqref{eq:MMpb}.
\subsection{Majorization-minimization framework}
Let $f(\bl)$ be a function to be minimized over $\mathbb{R}^p$. Given an initial guess $\bl^{(0)}$, at each iteration $k\in \mathbb{N}$ of a MM method, the following two steps are performed. \\
  \textbf{i) Majorization Step.} A surrogate function $Q(\bl,\bl^{(k)})$ is constructed which satisfies the majorization conditions
  \begin{equation}\label{eq:S}
  \begin{array}{l}
    Q(\bl^{(k)},\bl^{(k)}) = f(\bl^{(k)}) ,\\
    Q(\bl,\bl^{(k)}) \geq f(\bl), \quad \text{for each} \; \bl  .
    \end{array}
  \end{equation}
  \textbf{ii) Minimization step.} The successive iterate $\bl^{(k+1)}$ is obtained by minimizing the surrogate function $Q(\bl,\bl^{(k)})$ with respect to $\bl$, i.e:
  \begin{equation*}
    \bl^{(k+1)} =\arg\min\limits_{\bl} \; Q(\bl,\bl^{(k)}) .
  \end{equation*}
As a consequence of the majorization step, a MM method gives a sequence of iterates $\{\bl^{(k)}\}_{k\in\mathbb{N}}$ for which  the objective function decreases monotonically:
$$ f(\bl^{(k+1)})\leq Q(\bl^{(k+1)},\bl^{(k)}) \leq Q(\bl^{(k)},\bl^{(k)})=f(\blk).$$
Under mild conditions, convergence of a MM sequence $\{\blk\}_{k\in\mathbb{N}}$ to stationary points can be proved by using the convergence results given in \cite{Wu1983}
for EM algorithms. For the sake of completeness, we report here a classical convergence result \cite{Sun2017,zhang2007surrogate}.
\begin{theorem}\label{th:conv}
Consider a MM method with surrogate function $Q(\bl,\bl')$ applied to the optimization problem
  \begin{equation*}
    \min\limits_{\bl} \; f(\bl)
  \end{equation*}
  and suppose that
  \begin{description}
    \item[i)] $f: \R^p \rightarrow \R$ is continuous and differentiable;
    \item[ii)] the level set $\mathcal{L}=\{\bl \; | \; f(\bl)\leq f(\bl^{(0)})\}$ is compact;
    \item[iii)] $Q(\bl,\bl')$ is continuous in both $\bl$ and $\bl'$ and differentiable in $\bl$.
  \end{description}
Then, all the limit points of the generated sequence $\{\blk\}_{k\in\mathbb{N}}$ are stationary points of $f(\bl)$ and the sequence $\{f(\blk)\}_{k\in\mathbb{N}}$ converges monotonically to $f(\bl^*)$ for some stationary point $\bl^*$.
\end{theorem}
\subsection{The algorithm}
In the MM framework, we derive an iterative method for \eqref{eq:MMpb} by defining the following surrogate function for  $\Phi_p(\bl)$:
\begin{equation*}
  Q(\bl,\bl')= \frac{\left( \phi(\bu_{\bl'})+\bl^T\Vpsi(\bu_{\bl'})\right)^{2p}}{\prod_{i=1}^p\lambda_i}
\end{equation*}
where
\begin{equation}\label{eq:ul}
  \bu_{\bl'} = \arg\min\limits_{\bu\in\Omega} \; \{\phi(\bu)+(\bl')^T\Vpsi(\bu) \}.
\end{equation}
It immediately follows that
\begin{equation*}
  Q(\bl,\bl)=\Phi_p(\bl) .
\end{equation*}
Moreover, we have
\begin{equation*}
  Q(\bl,\bl')
  \geq \frac{\left( \min\limits_{\bu} \; \left\{\phi(\bu)+\bl^T\Vpsi(\bu)\right\}\right)^{2p}}{\prod_{i=1}^p\lambda_i} = \Phi_p(\bl) .
\end{equation*}
Thus, the function $Q(\bl,\bl')$ satisfies the surrogate conditions \eqref{eq:S}.
Observe that the surrogate function is continuous in both its arguments and differentiable with respect the first one since $\Phi_p(\bl)$ is continuous and differentiable (cfr. Theorem \ref{th:coercivity}).
Then, convergence of our MM scheme immediately follows from Theorem \ref{th:conv}.
In Theorem \ref{teo:lambda} we prove that a local minimum to the surrogate function $Q(\bl,\blk)$ can be found in a closed form. To this purpose we need  the following proposition.
\begin{proposition}\label{prop}
Let $\bz \in \mathbb{R}^p$ be defined as follows:
$$z_1= -\frac{1}{\sigma_1}, \quad z_2 = -\frac{2p}{2p-1} \frac{1}{\sigma_2}, $$
\begin{equation}\label{eq:z1}
z_k =-\frac{2p}{((2p-1)-(k-2))} \frac{1}{\sigma_k}, \qquad k \geq 3,
\end{equation}
where
$\sigma_1=\sqrt{2p-1}$, $\sigma_2=\sqrt{\frac{(2p-1)^2-1}{2p-1}}$
and
\begin{equation}\label{eq:z2}
\sigma_k = \left ( \frac{(2p-1)^2-(k-2)(2p-1)-(k-1)}{(2p-1)-(k-2)} \right)^{1/2}, \qquad k =3, \ldots, p.
\end{equation}
Then
$$\sigma_i^2+\sum_{\ell=1}^{i-1} z_\ell^2 = 2p-1, \qquad  i=1,\ldots,p , $$
and
$$
\sum_{\ell=1}^{i-1} z_\ell^2+z_i \sigma_i = -1 .
$$
\end{proposition}
\begin{proof}
The result will be proven by induction. Let us define the induction statement $P(i)$ composed by two parts:
$$P(i), \qquad \sigma_i^2 + \sum_{\ell=1}^{i-1} z_\ell^2 = 2p-1, \qquad \sum_{\ell=1}^{i-1} z_\ell^2 + z_i \sigma_i = -1.$$ 
For $i = 1$, we have $$\sigma_1^2 + \sum_{\ell=1}^{1-1} z_\ell^2 = \sigma_1^2 = \left(\sqrt{2p-1}\right)^2 = 2p-1, $$ which satisfies the first part of $P(1)$. Similarly, 
$$\sum_{\ell=1}^{1-1} z_\ell^2 + z_1 \sigma_1 = z_1 \sigma_1 = -\frac{1}{\sigma_1} \sqrt{2p-1} = -1,$$ which  satisfies the second part of $P(1)$.

Assume $P(k)$ is true for some $k \geq 1$, i.e., $$\sigma_k^2 + \sum_{\ell=1}^{k-1} z_\ell^2 = 2p-1, \qquad \sum_{\ell=1}^{k-1} z_\ell^2 + z_k \sigma_k = -1.$$
Now, we will prove $P(k+1)$. Concerning the first part we have:
$$
\sigma_{k+1}^2 + \sum_{\ell=1}^{k} z_\ell^2 = \sigma_{k+1}^2 + z_k^2 + \sum_{\ell=1}^{k-1} z_\ell^2 ;
$$
adding the term  $(+ \sigma_k^2 -\sigma_k^2)$ and using $P(k)$ we obtain
$$ \sigma_{k+1}^2 + \sum_{\ell=1}^{k} z_\ell^2 = \sigma_{k+1}^2 + z_k^2 - \sigma_k^2 + (2p-1) . $$
Substituting  the expressions in \eqref{eq:z1}, \eqref{eq:z2} (see Appendix \ref{App_B}) we can verify that:
$$ \sigma_{k+1}^2 + z_k^2 - \sigma_k^2 =0$$
hence
$$\sigma_{k+1}^2 + \sum_{\ell=1}^{k} z_\ell^2 =(2p-1).$$
Proceeding in a similar way, we can prove the second part of $P(k+1)$:
$$\sum_{\ell=1}^{k} z_\ell^2 + z_{k +1}\sigma_{k+1} = \sum_{\ell=1}^{k-1} z_\ell^2 + z_{k}^2+z_{k +1}\sigma_{k+1} + \sigma_k z_k - \sigma_k z_k$$
and using $P(k)$ we obtain
$$
\sum_{\ell=1}^{k} z_\ell^2 + z_{k +1}\sigma_{k+1} = -1-\sigma_k z_k +  z_{k}^2 + z_{k +1}\sigma_{k+1}
$$
where, substituting   \eqref{eq:z1}, \eqref{eq:z2}
(see Appendix \ref{App_B}), we obtain:
$$ -\sigma_k z_k +  z_{k}^2 + z_{k +1}\sigma_{k+1} = 0.$$ 
Hence, we have proven $P(k+1)$, which completes the induction step.
\end{proof}
We can now prove our main result, showing that the  regularization parameters determined by the UPS \eqref{eq:UPS} constitute a local minimum to the surrogate function.
\begin{theorem}\label{teo:lambda}
The vector $\bl^*\in\mathbb{R}^p$ with components
    \begin{equation*}
  \lambda_i^* = \frac{\phi(\bu_{\bl'})}{p \psi_i(\bu_{\bl'})}, \; i=1,\ldots,p
\end{equation*}
is a local minimum for the surrogate function $Q(\bl,\bl')$.
\end{theorem}
\begin{proof}
    We prove this results by showing that $\bl^*$  satisfies the second order sufficient condition.
    We have
    \begin{equation}\label{eq:dim1}
        \frac{\partial Q(\bl,\bl')}{\partial \lambda_j} = \frac{\left( \phi(\bu')+\bl^T \Vpsi(\bu')\right)^{2p-1}}{\prod_{i=1}^p\lambda_i} \cdot
        \left( 2p \psi_j(\bu')-\frac{1}{\lambda_j}\left( \phi(\bu')+\bl^T \Vpsi(\bu')\right)\right)
    \end{equation}
    where, for easier notation, $\bu'$ denotes $\bu_{\bl'}$. Since $\phi(\bu')+(\bl^*)^T \Vpsi(\bu') = 2\phi(\bu')$, by evaluating the second term at $\bl^*$, we obtain
    \begin{equation}\label{eq:dim2}
      2p \psi_j(\bu')-\frac{1}{\lambda_j^*}\left( \phi(\bu')+(\bl^*)^T \Vpsi(\bu')\right) =
      2p \psi_j(\bu')-\frac{p \psi_j(\bu')}{\phi(\bu')} \cdot 2\phi(\bu') = 0.
    \end{equation}
    Thus $\bl^*$ is a critical point and the result is proved if we show that the Hessian matrix of $Q$ at $(\bl^*,\bl')$ is positive definite. By combining \eqref{eq:dim1} and \eqref{eq:dim2}, we obtain
    \begin{multline*}
      \frac{\partial^2}{\partial \lambda_\ell \partial \lambda_j}Q(\bl,\bl')_{ \big| \bl = \bl^*} = \\
      \frac{\left( 2 \phi(\bu')\right)^{2p-1}}{\prod_{i=1}^p\lambda_i^*} \cdot \frac{\partial }{ \partial \lambda_\ell}\left( 2p \psi_j(\bu')-\frac{1}{\lambda_j}\left( \phi(\bu')+\bl^T \Vpsi(\bu')\right)\right)_{ \big| \bl = \bl^*}
    \end{multline*}
     and
     \begin{multline*}
       \frac{\partial }{ \partial \lambda_\ell}\left( 2p \psi_j(\bu')-\frac{1}{\lambda_j}\left( \phi(\bu')+\bl^T \Vpsi(\bu')\right)\right) = \\
       \left\{
         \begin{array}{ll}
           \displaystyle{-\frac{1}{\lambda_j}\psi_\ell(\bu')}, & \hbox{if $j\neq \ell$;} \\
           \displaystyle{-\frac{1}{\lambda_\ell}\psi_\ell(\bu')+\frac{1}{\lambda_\ell^2}\left( \phi(\bu')+\bl^T \Vpsi(\bu')\right)}, & \hbox{if $j=\ell$.}
         \end{array}
       \right.
     \end{multline*}
     Therefore
 \begin{equation*}
      \frac{\partial^2}{\partial \lambda_\ell \partial \lambda_j}Q(\bl,\bl')_{ \big| \bl = \bl^*} =
      \frac{\left( 2 \phi(\bu')\right)^{2p-1}}{\prod_{i=1}^p\lambda_i^*} \cdot
       \left\{
         \begin{array}{ll}
           \displaystyle{-\frac{p}{\phi(\bu')}\psi_j(\bu')\psi_\ell(\bu')}, & \hbox{if $j\neq \ell$;} \\
           \displaystyle{\frac{p}{\phi(\bu')}(2p-1)\psi_\ell(\bu')^2}, & \hbox{if $j=\ell$;}
         \end{array}
       \right.
    \end{equation*}
where we have used
\begin{equation*}
  -\frac{1}{\lambda^*_\ell}\psi_\ell(\bu')+\frac{1}{{\lambda^*_\ell}^2}\left( \phi(\bu')+(\bl^*)^T \Vpsi(\bu')\right) =
-\frac{p}{\phi(\bu')}\psi_\ell(\bu')^2 + 2\phi(\bu') \frac{p^2}{\phi(\bu')^2}\psi_\ell(\bu')^2 .
\end{equation*}
Hence,
\begin{equation*}
  \nabla^2 Q(\bl^*,\bl') =
       \frac{p \left( 2 \phi(\bu')\right)^{2p-1}}{\phi(\bu') \prod_{i=1}^p\lambda_i^*} H
\end{equation*}
where $H$ is the symmetric matrix defined as
\begin{equation*}
  H_{j\ell}= \left\{
         \begin{array}{ll}
           -\psi_j(\bu')\psi_\ell(\bu'), & \hbox{if $j\neq \ell$} \\
           (2p-1)\psi_\ell(\bu')^2, & \hbox{if $j=\ell$}
         \end{array}
       \right. , \quad j,\ell=1,\ldots,p .
\end{equation*}
We prove that $H$ is positive definite by showing that it possesses a unique Cholesky factorization \cite{GolubVanLoan1996}.
Indeed, applying the Cholesky procedure to $H$  
we obtain a lower triangular matrix $L$  defined as follows:
\begin{equation}\label{eq:chol}
\begin{aligned}
&    L_{1,1} = \psi_1(\bu')\sqrt{2p-1}, \\
&    L_{i,1}= - \psi_i(\bu') / \sqrt{2p-1}, \qquad i=2, \ldots , p , \\
&    L_{2,2} = \psi_2(\bu') \sigma_2, \ \sigma_2=\sqrt{\frac{(2p-1)^2-1}{2p-1}} , \\
&    L_{i,2} = - \psi_i(\bu') \frac{2p}{2p-1}/ \sigma_2, \qquad i=3, \ldots , p , \\
&    k=3, \ldots, p  \\
&    \hspace{3mm} L_{k,k} = \psi_k(\bu') \sigma_k, \ \sigma_k = \left ( \frac{(2p-1)^2-(k-2)(2p-1)-(k-1)}{(2p-1)-(k-2)} \right)^{1/2},  \\
&    \hspace{3mm} L_{i,k} = - \psi_i(\bu') \frac{2p}{\sigma_k((2p-1)-(k-2)}, \ i=k+1,\ldots, p.
\end{aligned}
\end{equation}
Now we prove that $H=LL^T$ and $L$ has positive diagonal elements. \\
Defining $\sigma_1=\sqrt{2p-1}$ we can write each row of $L$ in \eqref{eq:chol} as follows:
$$L_{i,\cdot}= \psi_i(\bu')\left ( z_1, z_2, \ldots, z_{i-1}, \sigma_i, 0, \ldots, 0 \right )^T, \quad i>1 , $$
where
$$z_1= -\frac{1}{\sigma_1}, \quad z_2 = -\frac{2p}{2p-1} \frac{1}{\sigma_2}, \qquad
z_k =-\frac{2p}{((2p-1)-(k-2))} \frac{1}{\sigma_k}, \qquad k \geq 3.$$
Using Proposition \ref{prop} with $N=p$, we can easily check the Cholesky factorization of $H$. In the case $i=j$ we have
$$
H_{i,i} = \sum_{\ell=1}^{p} L_{i,\ell}L_{i,\ell} = \left(\sigma_{i}^2+\sum_{\ell=1}^{i-1} z_\ell^2\right )\psi_i(\bu')^2  =(2p-1)\psi_i(\bu')^2 , \qquad i=1, \ldots, p . $$
When $i\neq j$, for symmetry we consider  $1<j<i\leq p$:
$$
H_{i,j} = \sum_{\ell=1}^{p} L_{i,\ell}L_{j,\ell} = \sum_{\ell=1}^{j} L_{i,\ell}L_{j,\ell} =\psi_i(\bu')\psi_j(\bu')\left ( \sum_{\ell=1}^{j-1} z_\ell^2+z_j \sigma_j\right ) = -\psi_i(\bu')\psi_j(\bu').
$$
We can check that the diagonal entries of $L$, given by: 
$$L_{k,k}=\psi_k(\bu') \sigma_k > 0, \qquad  \psi_k(\bu') > 0 \ k=1, \ldots, p$$
are all positive.  Indeed it is straighforward for $k=1, 2$:
$$\sigma_1 = \sqrt{2p-1} > 0, \quad p>0$$
and   $\sigma_2$ in \eqref{eq:chol} is positive for $p>1$.
For the terms $k=3, \ldots, p$, we consider $f(k) \equiv \sigma_k^2$, i.e.
$$ f(k) = \frac{(2p-1)^2-(k-2)(2p-1)-(k-1)}{(2p-1)-(k-2)}, $$
and  observe that $p+1 \leq (2p-1)-(k-2)\leq 2p$, hence
$$f(k) \geq \frac{(2p-1)^2-(k-2)(2p-1)-(k-1)}{2p}\equiv \mathcal{G}(k) . $$
Since $\mathcal{G}'(k)=-1$ then the minimum is reached at $k=p$, i.e.
$$ f(k) \geq  \frac{2p^2}{2p}=p, \quad  p \geq 1.$$
Therefore $\sigma_k \geq \sqrt{p}>0$, for $k \geq 3$.  
Computing  the Cholesky factorization of $H$ we have proved that it is positive definite for all $p$. This concludes the proof.
\end{proof}
At the $k$th iteration of our MM scheme, the construction of the surrogate function $Q(\bl,\blk)$ in the majorization step requires the solution of a constrained quadratic optimization problem \eqref{eq:ul}. 
When the feasible set $\Omega$ is the nonnegative orthant, this task can be efficiently performed by using a gradient projection-type method \cite{ber99}. By Theorem \ref{teo:lambda}, the
minimization step, leads to 
 \begin{equation}\label{eq:update_lambda}
  \lambda_i^{(k+1)} = \frac{\phi(\bu_{\bl^{(k)}})}{p\psi_i(\bu_{\bl^{(k)}})}, \; i=1,\ldots,p .
\end{equation}
We remark that the regularization parameter vector $\bl^{(k+1)}$ obtained via \eqref{eq:update_lambda} is nonnegative.
Moreover, assumption $\psi_i(\bu)>\epsilon$, $i=1,\ldots,p$, prevents division by zero. 
We refer to our iterative method for solving model \eqref{eq:UPS} as UpenMM and we sketch it in Algorithm \ref{alg:MM} where, for easier notation, $\bu^{(k)}$ denotes $\bu_{\blk}$.
\begin{algorithm}[h]
    \caption{\name{} \label{alg:MM}}
    {\small
    \begin{algorithmic}[1]
        \STATE Choose $\bl^{(0)}\in\mathbb{R}^p$, and set $k=0$.
        \REPEAT
        \STATE {$\displaystyle \bu^{(k)} = \arg\min\limits_{\bu\in\Omega} \; \left \{ \frac{1}{2}\|A\bu-\bb\|^2+ \sum_{i=1}^p \lambda_i^{(k)} \psi_i (\bu) \right \}$ \label{alg:step1}}
        \STATE $ \lambda_i^{(k+1)} = \displaystyle{\frac{\|A\bu^{(k)}-\bb\|^2}{p \psi_i(\bu^{(k)})}}, \; i=1,\ldots,p$
        \STATE $ k = k+1 $
        \UNTIL{a {\tt stopping criterion} is satisfied}
    \end{algorithmic}
    }
\end{algorithm}
In our implementation, we stop Algorithm \ref{alg:MM} when the relative change in the
computed parameter vector $\bl$ goes below a certain threshold $Tol_{\lambda} \in (0, 1)$ or a maximum number of iterations $k_\text{max}$ is reached.
\begin{remark}
    Algorithm \ref{alg:MM} was introduced in \cite{ito2011multi} as a fixed-point like method for the numerical realization of the BP with two penalty terms. However, its convergence was not proven for the general case $p>1$.
\end{remark}

\subsection{Generalized \name{}}
UPenMM is a general framework which embeds different regularization functions $\psi_i:\Omega \to \R^+$.
We notice, however, that a converging procedure can still be obtained even if, instead of using the exact minimum of the surrogate function, we compute a vector $\bl^{(k+1)}$ satisfying the descent condition \cite{Dempster,Sun2017}
\begin{equation}\label{eq:decrease}
    Q(\bl^{(k+1)},\bl^{(k)}) \leq Q(\bl^{(k)},\bl^{(k)}).
\end{equation}
Let us denote by $\widehat{\bl}$ the exact minimizer of the surrogate function  defined in \eqref{eq:update_lambda}, and let $\widetilde{\bl}$ be an approximation obtained by a different relation
(for easier notation we have omitted the index $k$). We propose a generalized \name{} (G\name{}) algorithm, that uses a convex combination of $\widehat{\bl}$ and $\widetilde{\bl}$. More precisely, given $\varepsilon\in (0,1)$, the next iterate $\bl^{(k+1)}$ is the first element of the sequence 
\begin{equation}\label{eq:decr}
    \{ \; \varepsilon^j \widetilde{\bl} + (1-\varepsilon^j) \widehat{\bl}, \; j=0,1,2,\ldots \}
\end{equation}
which satisfies \eqref{eq:decrease}. We observe that \eqref{eq:decr} converges to $ \widehat{\bl}$ as $j \to \infty$, that obviously satisfies \eqref{eq:decrease}. The generalized algorithm, through appropriate strategies for computing $\widetilde{\bl}$, allows for decisive improvements in the case of real data affected by noise.
G\name{} is outlined in Algorithm \ref{alg:GMM} using  the same {\tt stopping criterion} as in Algorithm \ref{alg:MM}.
\begin{algorithm}[h]
    \caption{G\name{} \label{alg:GMM}}
    {\small
    \begin{algorithmic}[1]
        \STATE Choose $\bl^{(0)}\in\mathbb{R}^p$, $\varepsilon\in(0,1)$, and set $k=0$.
        \REPEAT
        \STATE {$\displaystyle \bu^{(k)} = \arg\min\limits_{\bu\in\Omega} \;  \left \{ \frac{1}{2} \|A\bu-\bb\|^2+ \sum_{i=1}^p \lambda_i^{(k)} \psi_i (\bu) \right \}$ }
        \STATE Compute $\widehat{\bl}^{(k+1)}$ and $\widetilde{\bl}^{(k+1)}$;  set $\bl^{(k+1)}=\widetilde{\bl}^{(k+1)}$ and $j=0$
        \WHILE{%
        $Q(\bl^{(k+1)},\bl^{(k)}) > Q(\bl^{(k)},\bl^{(k)})$ %
        }
        \STATE $j=j+1$
        \STATE $\bl^{(k+1)}= \varepsilon^j \widetilde{\bl}^{(k+1)} + (1-\varepsilon^j) \widehat{\bl}^{(k+1)} $
        \ENDWHILE
        \STATE $ k = k+1 $
        \UNTIL{a {\tt stopping criterion} is satisfied}
    \end{algorithmic}
    }
\end{algorithm}
\section{Numerical Results\label{sec:experiments}}
This section presents results derived from the application of \name{} and G\name{} to two-dimensional NMR relaxometry problems. We focus on Gaussian noise as it aligns with the data fitting function embedded in our model.
We analyze the convergence of \name{} and G\name{} in terms of the decrease of the surrogate function and evaluate their performance in efficiently reducing relative errors while recovering various complex features.
Analogous results are observed in one-dimensional tests, which are provided in the supplementary materials. These materials also include the convergence history and comparisons between G\name{} and other methods, such as Tikhonov regularization, confirming the effectiveness of \name{} and G\name{} in recovering different signal features, including peaks, smooth rounded areas, and ramps, at varying noise intensities.
All experiments were conducted using Matlab R2023b on an Apple M1 computer with 16 GB of RAM. The codes used for the current experiments can be made available upon request to the authors.

We consider a test problem derived from Nuclear Magnetic Resonance relaxometry applications \cite{bortolotti20212dnmr}.
The problem can be expressed as:
\begin{equation}
\bb = (\mathbf{K}_2 \otimes \mathbf{K}_1) \bu  + \mathbf{e}
\end{equation}
where $\mathbf{K}_1\in \mathbb{R}^{M_1 \times N_1}$ and $\mathbf{K}_2\in \mathbb{R}^{M_2 \times N_2}$ denote the discretized kernels of a separable Fredholm integral equation.
The vector $\mathbf{b} \in \mathbb{R}^{M}$, with $M=M_1 \cdot M_2$, represents the measured relaxation data from an Inversion Recovery (IR) Carr-Purcell-Meiboom-Gill (CPMG) pulse sequence \cite{bortolotti20212dnmr}. The unknown vector $\bu \in \mathbb{R}^{N}$ corresponds to the vector reordered two-dimensional relaxation distribution of size  $N_1 \times N_2$. 
The term $\mathbf{e} \in \mathbb{R}^{M}$ represents the additive Gaussian noise with a magnitude of $0.01$.
In this test, we have $N_1=N_2=80$, which implies $N=6400$, and $M_1=128$, $M_2=2048$, resulting in $M=262144$.
The \name{} and G\name{} iterations are stopped according  to  the following {\tt stopping criterion}:
\begin{equation}\label{eq:SC}
\| \Veta^{(k+1)}-\Veta^{(k)} \| \leq \| \Veta^{(k)} \| Tol_{\lambda} 
\end{equation}
with $Tol_{\lambda}=10^{-2}.$

The regularization model employed here has $N$ local $L_2$ penalty functions and one global $L_1$ penalty, i.e., $\psi_i(\bu)=(L\bu)_i^2+\epsilon$  for $i=1,\ldots,N$, with $\epsilon = 10^{-6}$,  and $\psi_{N+1}=\|\bu \|_1$. Using \name, this can be expressed as:
\begin{equation}
\lambda_i^{(k+1)} = \frac{\|A\bu^{(k)}-\bb\|^2}{(N+1) \left ((L\bu^{(k)})_i^2+\epsilon \right )}, \; i=1,\ldots,N, \; \lambda_{N+1}^{(k+1)} = \frac{\|A\bu^{(k)}-\bb\|^2}{(N+1) \left(\|\bu^{(k)} \|_1+\epsilon \right )}.
\end{equation}
However, in practice, \name{}  provides suboptimal solutions. A better approach involves the generalized formula:
\begin{equation}
\widetilde{\lambda}_i^{(k+1)} = \frac{\|A\bu^{(k)}-\bb\|^2}{N \widetilde{\psi}_i(\bu^{(k)})}, \quad i=1,\ldots,N, \quad \widetilde{\lambda}_{N+1}^{(k+1)} = \frac{\|A\bu^{(k)}-\bb\|^2}{\|\bu^{(k)} \|_1}.
\end{equation}
The modified penalty functions are given by:
\begin{equation}
\widetilde{\psi}_i(\bu^{(k)}) = \max_{i \in \mathcal{I}_i} (L\bu^{(k)})_i^2+ \max_{i\in \mathcal{I}_i} (P\bu^{(k)})_i^2 + \epsilon, \quad i=1,\ldots,N
\end{equation}
where $\mathcal{I}_i$ denotes the index subsets of the neighborhood around the $i$th point, for $i=1,\ldots,N$. Additionally, $P$ represents the matrix of the central finite difference approximation of the first-order differential operator. The subproblem at step 3 of Algorithms \ref{alg:MM} and \ref{alg:GMM} is solved by using the FISTA method \cite{BeckTeboulle2009}.
\FZcomment[The contour plots in Figure \ref{fig:2D_sol} highlight the superior performance of G\name{}. In particular, panel (c) shows more smoothing of the peak values compared to panel (b). Despite this, panel (c) provides a much better fit to the tail and background regions, which are not as well-represented in panel (b). This trade-off between peak smoothing and improved tail and background fitting makes panel (c) a more balanced and accurate representation overall. This observation is further supported by the error history in Figure \ref{fig:2D_alg}(a).
]
The descent properties of the surrogate functions can be seen in Figure \ref{fig:2D_alg}(b).
From a computational cost perspective, G\name{} is notably more efficient than \name{}, requiring only $6551$ inner FISTA iterations compared to $83675$ iterations.
One critical aspect in implementing condition \eqref{eq:decrease} for 2D data is ensuring overflow errors are avoided, especially considering the large value of $p$. By defining:
\begin{align}
Q(\bl^{(k+1)},\bl^{(k)}) &= \frac{\left( \phi(\bu_{\bl^{(k)}})+\left (\bl^{(k+1)}\right)^T\Vpsi(\bu_{\bl^{(k)}})\right)^{2p}}{\prod_{i=1}^p\lambda_i^{(k+1)}} \equiv \left (\frac{f_{k+1}^2}{\pi_{k+1}^{1/p}} \right )^p, \\
Q(\bl^{(k)},\bl^{(k)}) &= \frac{\left( \phi(\bu_{\bl^{(k)}})+ \left (\bl^{(k)}\right )^T\Vpsi(\bu_{\bl^{(k)}})\right)^{2p}}{\prod_{i=1}^p\lambda_i^{(k)}} \equiv \left(\frac{f_{k}^2}{\pi_{k}^{1/p}}\right )^p,
\end{align}
we can validate the decrease condition as follows:
\begin{equation}
f_{k+1}^2 \pi_{k}^{1/p} < f_{k}^2 \pi_{k+1}^{1/p},
\end{equation}
scaling  $\pi_{k}$ and $\pi_{k+1}$ to avoid zero values caused by underflow errors.
\begin{figure}[htbp]
\begin{center}
(a) \hspace{3cm} (b) \hspace{3cm} (c) \\
\includegraphics[width=4.5cm]{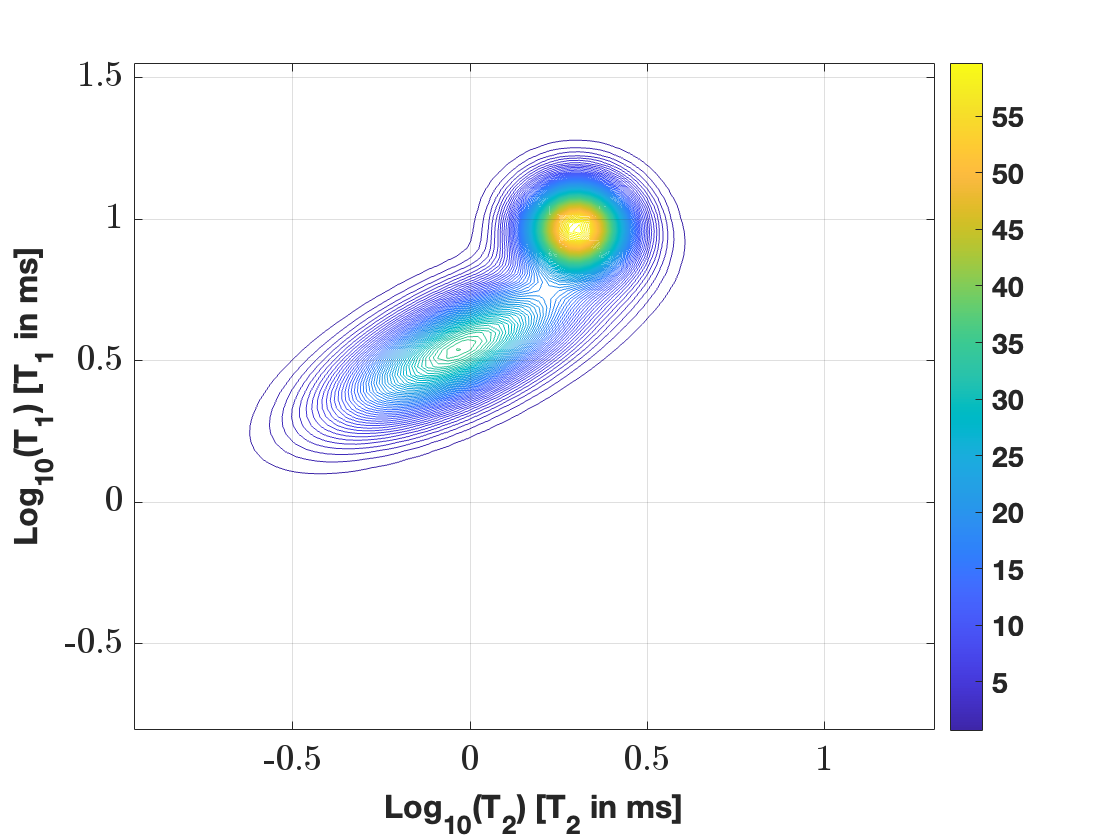}\hspace{-.5cm}
\includegraphics[width=4.5cm]{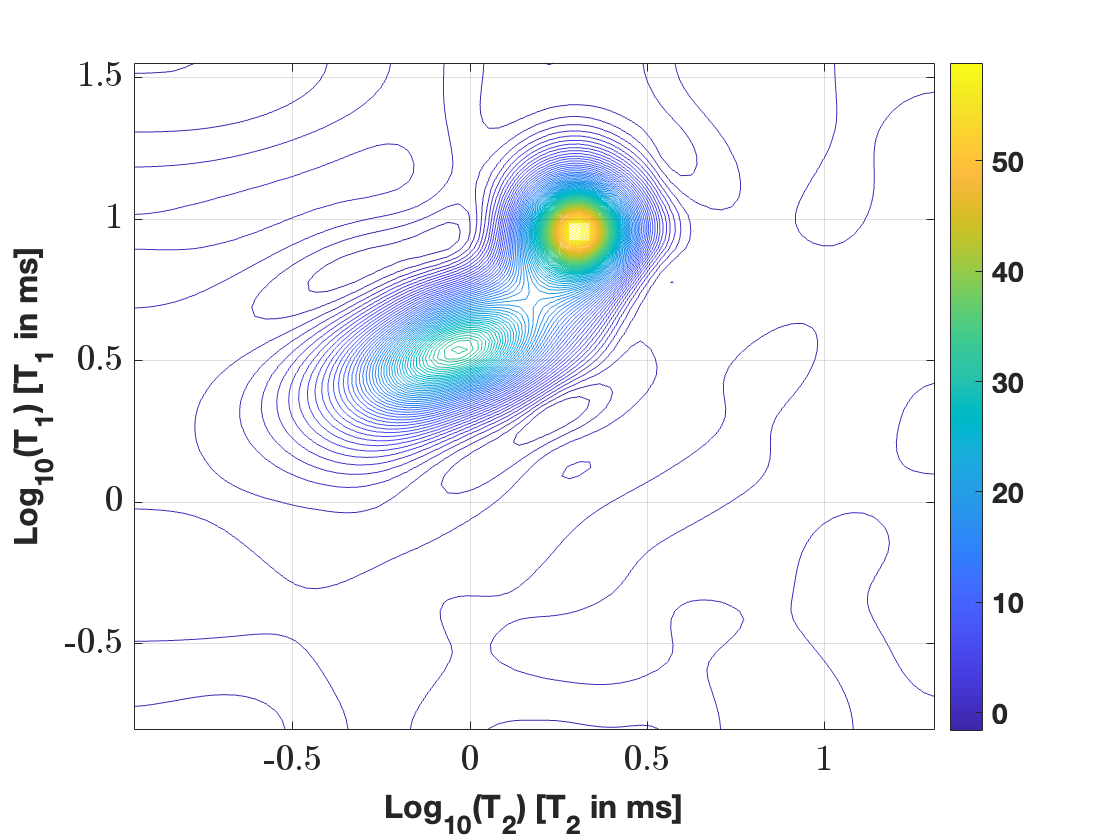}\hspace{-.5cm}
\includegraphics[width=4.5cm]{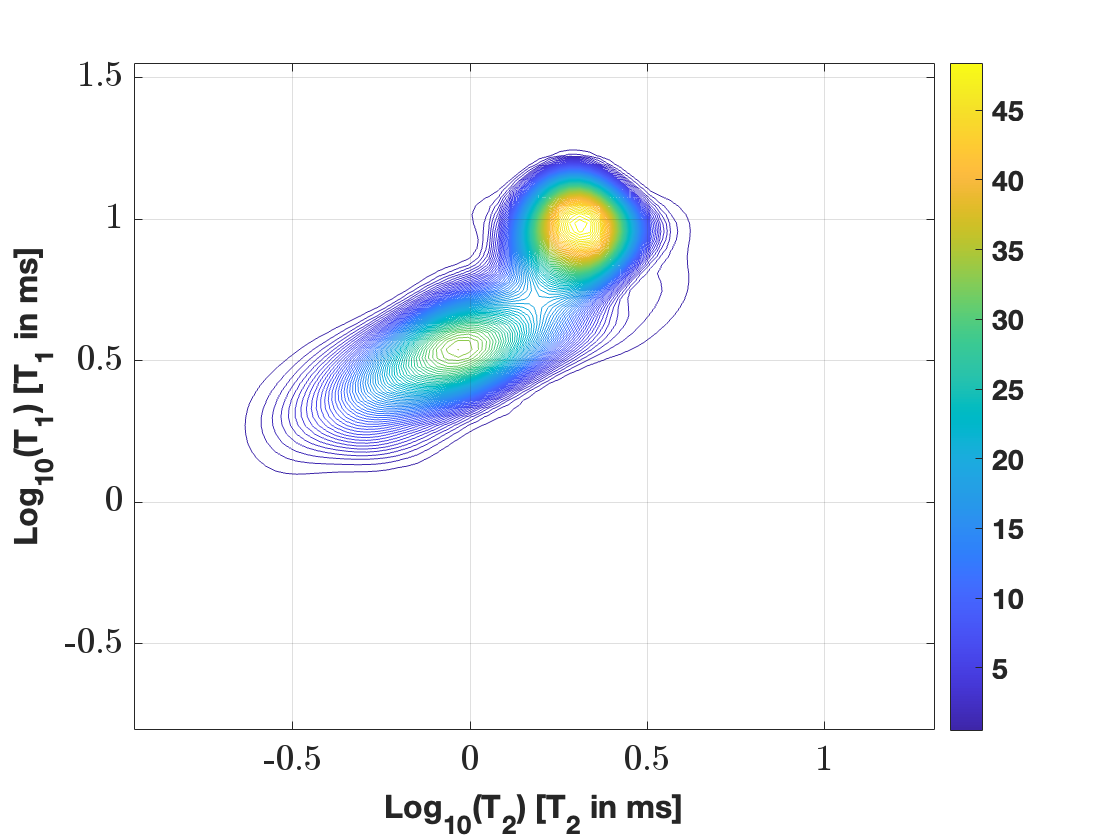}
\end{center}
\caption{2D  test problem contour plots of the computed solution $\bu$. (a)  true solution. (b) \name{}
Relative error: $1.397 \ 10^{-1}$. (c) G\name{}, Relative error: $1.2245 \ 10^{-1}$.}
\label{fig:2D_sol}
\end{figure}
\begin{figure}[htbp]
\begin{center}
(a) \hspace{5cm} (b) \\
\includegraphics[width=6cm]{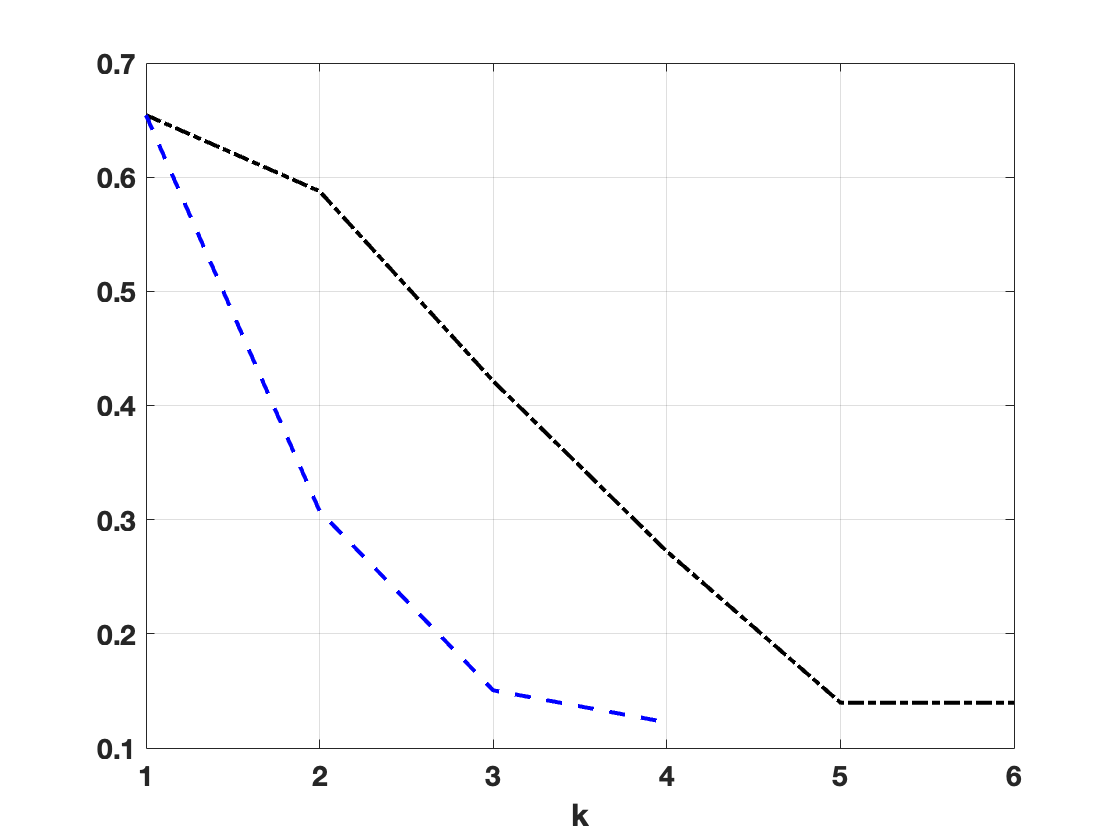}\hspace{-.6cm}
\includegraphics[width=6cm]{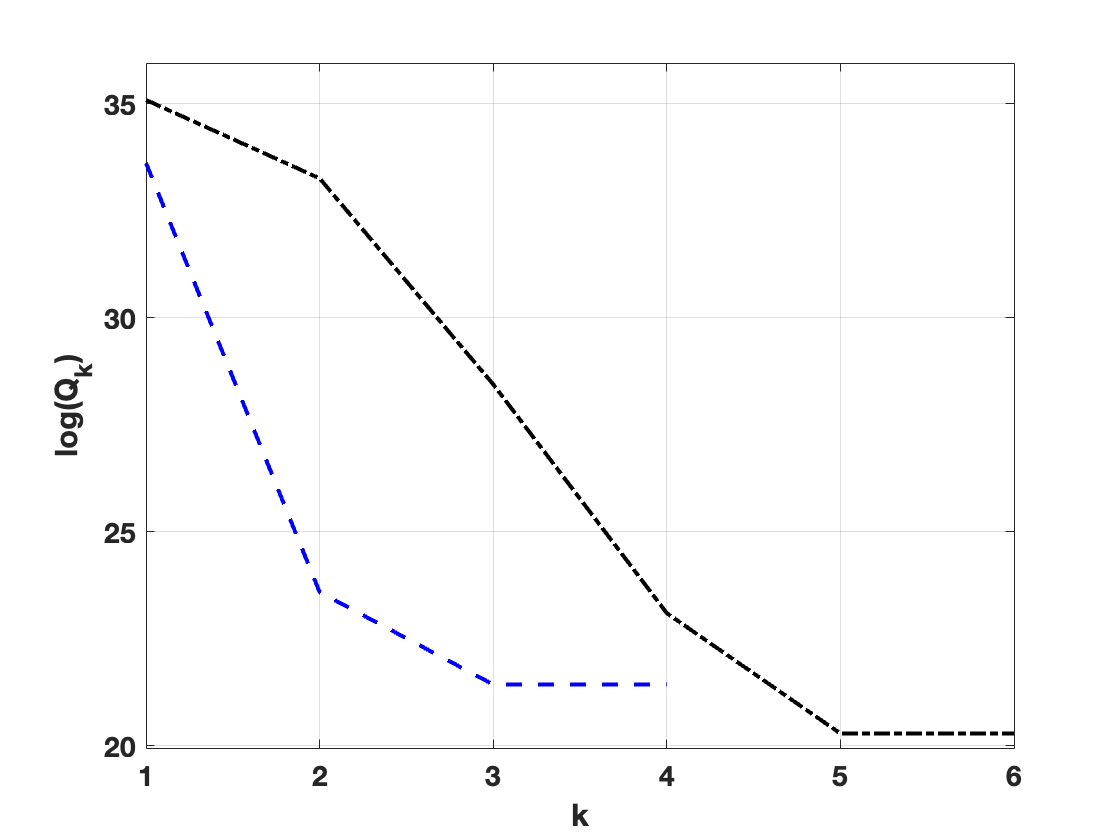}
\end{center}
\caption{2D  test problem. (a) Relative errors. (b) Surrogate function. 
Black dash-dot line are obtained by Algorithm \ref{alg:MM} and  blue dashed line are given by  Algorithm \ref{alg:GMM}.
}
\label{fig:2D_alg}
\end{figure}
\begin{remark}
The functions $\widetilde{\psi}_i(\bu)$ 
are obtained by applying a nonlinear maximum filter. 
The practice of introducing a nonlinear maximum filter in one or two dimensions was known from the earliest works \cite{borgia1998uniform,bortolotti2016uniform}. It allows for better control of instabilities caused by noise and improves the quality of approximations. In this work, we demonstrate that it fits within the context of the generalized \name{} method.   By controlling the decrease of the surrogate function, 
we obtain G\name{}, an efficient  algorithm   where  in general less than two convex combinations in \eqref{eq:decr} are requested (i.e. $j \leq 2$) for each step.
With this modification, the sequence of approximate solutions and regularization parameters rapidly converge, confirming what has been known heuristically for several years.
\end{remark}
%
\section{Conclusions}\label{sec:conclusions}
This paper provides the theoretical foundation of point-wise and multi-penalty regularization through the principle of uniform penalty. The numerical realization of such a principle   has been in use since the '90s and has been successfully applied to multiple application problems for its efficiency and relative simplicity. However, until now, there was a lack of a theoretical analysis that characterized the solutions and proved convergence in the most general case. With this work, we have provided the characterization of the solutions through the hierarchical Bayesian approach. We have extended to the case of point-wise and multi-penalty with bound constraints, what was available in the  literature for unconstrained solutions. Moreover, using the fundamental contributions of Ito et al. \cite{ito2014multi}, we have formalized the proof of convergence of the UPenMM method for very general regularization functionals. To achieve this, it was essential to view this method in the context of majorization-minimization algorithms. This is another original contribution of this work: through the identification of an appropriate surrogate function, it was possible to verify the convergence of UPenMM. Finally, the introduction of a generalized MM approach allowed us to include, in this theoretical framework, heuristic techniques used to address distortions due to noise in real data.

In conclusion, this study demonstrates the formal validity of the original UPEN algorithm, developed heuristically, and lays the foundation for its extension to inverse problems of a different nature
compared to the application field (NMR relaxometry) for which UPEN was initially conceived.

Future work includes investigating such a strategy on data affected by non-Gaussian noise, which could lead to nonlinear inverse problems.
\appendix
\section{Proof of Theorem \ref{teo1} (details)} \label{App_A} \\
When $\gamma > 0$ and $\psi_i (\Veta) \neq 0 
 \ \forall i$, we can apply the Shermann Morrison formula:
$$\left (\mathbf{D} - \mathbf{1}\Vpsi(\Veta)^T \right )^{-1} = \mathbf{D}^{-1} - \frac{\mathbf{D}^{-1}w v^T\mathbf{D}^{-1}}{ 1 + v^T\mathbf{D}^{-1}w} , $$
where 
\HIcomment[
$$w=\mathbf{1}, v=-(\psi_1(\Veta), \ldots, \psi_p(\Veta))^T, \text{and } \mathbf{D}^{-1} =  \frac{1}{p+\gamma} \begin{pmatrix}
 \tfrac{1}{\psi_1(\Veta)}  & \\
 0 & \tfrac{1}{\psi_2(\Veta)}  & \\
 0 & 0 & \ddots & \\
  0 & \cdots & 0 & \tfrac{1}{\psi_p(\Veta)} 
  \end{pmatrix}. $$]
\HIcomment[ Then we have: $$
\begin{aligned}
& \mathbf{D}^{-1}w = \frac{1}{p+\gamma} \mathbf{1}, \quad
 \mathbf{D}^{-1}v = -\frac{1}{p+\gamma} \begin{pmatrix}
\tfrac{1}{\psi_1(\Veta)}\\
\tfrac{1}{\psi_2(\Veta)}\\
\vdots \\
\tfrac{1}{\psi_p(\Veta)}
\end{pmatrix},\\
 & \text{ and }\quad \mathbf{D}^{-1}w v^T\mathbf{D}^{-1}  =  -\frac{1}{(p+\gamma)^2}  \begin{pmatrix}
 \tfrac{1}{\psi_1(\Veta)} & \tfrac{1}{\psi_1(\Veta)}& \cdots & \tfrac{1}{\psi_1(\Veta)} \\
\tfrac{1}{\psi_2(\Veta)} & \tfrac{1}{\psi_2(\Veta)}& \cdots & \tfrac{1}{\psi_2(\Veta)} \\
\vdots & \vdots & \ddots & \vdots \\
\tfrac{1}{\psi_p(\Veta)} & \tfrac{1}{\psi_p(\Veta)} & \cdots &  \tfrac{1}{\psi_p(\Veta)}
\end{pmatrix}.  
\end{aligned}
$$]
Hence
$$
\begin{aligned}
& \left (\mathbf{D} - \mathbf{1}\Vpsi(\Veta)^T \right )^{-1} = \\
& \frac{1}{p+\gamma} \left [ \begin{pmatrix}
 \tfrac{1}{\psi_1(\Veta)}  & \\
 0 & \tfrac{1}{\psi_2(\Veta)}  & \\
 0 & 0 & \ddots & \\
  0 & \cdots & 0 & \tfrac{1}{\psi_p(\Veta)} 
  \end{pmatrix} + \frac{1}{\gamma}  \begin{pmatrix}
 \tfrac{1}{\psi_1(\Veta)} & \tfrac{1}{\psi_1(\Veta)}& \cdots & \tfrac{1}{\psi_1(\Veta)} \\
\tfrac{1}{\psi_2(\Veta)} & \tfrac{1}{\psi_2(\Veta)}& \cdots & \tfrac{1}{\psi_2(\Veta)} \\
\vdots & \vdots & \ddots & \vdots \\
\tfrac{1}{\psi_p(\Veta)} & \tfrac{1}{\psi_p(\Veta)} & \cdots &  \tfrac{1}{\psi_p(\Veta)}
\end{pmatrix}  \right ] .
\end{aligned}
$$
%
 \section{Proof of Proposition \ref{prop} (details)} \label{App_B}
 The cases $k=1,2$ are trivial and we show the relations when $k\geq 3$.
 \begin{itemize}
 \item Prove that $ \sigma_{k+1}^2 + z_k^2 - \sigma_k^2 =0.$
 Substituting the expressions of $\sigma_k$ and $z_k$ we have:
$$ 
\begin{aligned}
& \sigma_{k+1}^2-\sigma_{k}^2 = \\
& = \frac{k+{\left(2p-1\right)}\,{\left(k-2\right)}-{{\left(2p-1\right)}}^2 -1}{2p-k+1}-\frac{k+{\left(2p-1\right)}\,{\left(k-1\right)}-{{\left(2p-1\right)}}^2 }{2p-k} = \\
& = -\frac{2p}{{\left(2p-k\right)}\,{\left(2p-k+1\right)}} ;
\end{aligned}
$$
and
$$
\begin{aligned}
z_k^2 
 & =  -\frac{4p^2 }{{\left(1+2p-k\right)}\,{\left(k+{\left(-1+2p\right)}\,{\left(-2+k\right)}-{{\left(-1+2p\right)}}^2 -1\right)}} \\
& = \frac{2p}{(2p-k+1)}\frac{\cancel{2p}}{\cancel{2p}(2p-k)} .
\end{aligned}
$$
Hence $\sigma_{k+1}^2 + z_k^2 - \sigma_k^2 =0$.
\item Prove that $ -\sigma_k z_k +  z_{k}^2 + z_{k +1}\sigma_{k+1} = 0.$
Substituting the expressions of $\sigma_k$ and $z_k$ we have: 
$$z_{k +1}\sigma_{k+1}=-\frac{2p}{2p-k} \quad \text{and} \quad
\sigma_k z_k = -\frac{2p}{2p-k+1}$$
hence
$$
-\sigma_k z_k + z_{k +1}\sigma_{k+1} = 
-\frac{2p}{(2p-k+1)(2p-k)}
$$
therefore $-\sigma_k z_k +  z_{k}^2 + z_{k +1}\sigma_{k+1} = 0$. 
 \end{itemize}
\section*{Acknowledgments}
We would like to thank Prof. Paola Fantazzini and remember  Prof. Bob Brown for their work in introducing and applying the Uniform Penalty principle to real-world problems in NMR. Additionally, we wish to honour the memory of Prof. Daniela di Serafino for her inspiration and encouragement. Finally, we thank the anonymous reviewers for their valuable comments and suggestions.
%
%

\end{document}


\maketitle

\section*{Introduction}

This document presents the supplementary materials for the article titled "Uniform Multi-Penalty Regularization for Linear Ill-Posed Inverse Problems" by Villiam Bortolotti, Germana Landi, and Fabiana Zama. The supplementary materials provide  extended discussions on one-dimensional test problems. After the description of the test problems and algorithmic settings in section \ref{sec:test}, low noise and high noise results are extensively discussed in sections \ref{sec:LN} and \ref{sec:HN} respectively. 
%
\section{One-dimensional test problems}\label{sec:test}
Starting from a ground-truth signal $\bu^* \in \mathbb{R}^N$,  we generate the test problem  by convolving it with a blurring operator $A \in \mathbb{R}^{M\times N}$ and then corrupting the blurred signal with additive Gaussian noise of level $\delta$, i.e.
$$\bb = \by + \text{noise}$$ where
$\text{noise} = \delta \boldsymbol{\eta} \|\by\|$,
with  $\by = A \bu^*$ and $\boldsymbol{\eta} \in \mathbb{R}^M$ a random normal vector with norm one.
We consider two scenarios: a low-noise scenario corresponding to $\delta=0.01$, and a high-noise scenario where $\delta=0.1$.
We define three test problems where each signal has different and progressively more complex features representing a challenge for the regularization functions.
\begin{itemize}
\item[{\tt T1}] The signal $\bu^*\in \R^{100}$ and the linear operator $A\in \R^{100\times 100}$ are obtained by the  {\tt heat} function from Hansen's Regularization Tools \cite{hansen2007regularization}. The ground truth $\bu^*$, represented in Figure \ref{fig:tests} (left), is a sparse signal with mostly zero values except for a single peak quite close to the axes origin.
\item[{\tt T2}] The signal $\bu^*\in \R^{404}$ (Figure \ref{fig:tests}, center) has two narrow peaks over a flat area and a smooth rounded area. The operator $A$ represents a Gaussian blur with a standard deviation of five. 
\item[{\tt T3}] The signal $\bu^*\in \R^{504}$ (Figure \ref{fig:tests}, right) presents a smooth rounded feature, a narrow peak and a ramp. The same Gaussian blur of {\tt T2} is used.
\end{itemize}
%
\begin{figure}[htbp]
\begin{center}
\includegraphics[width=4cm]{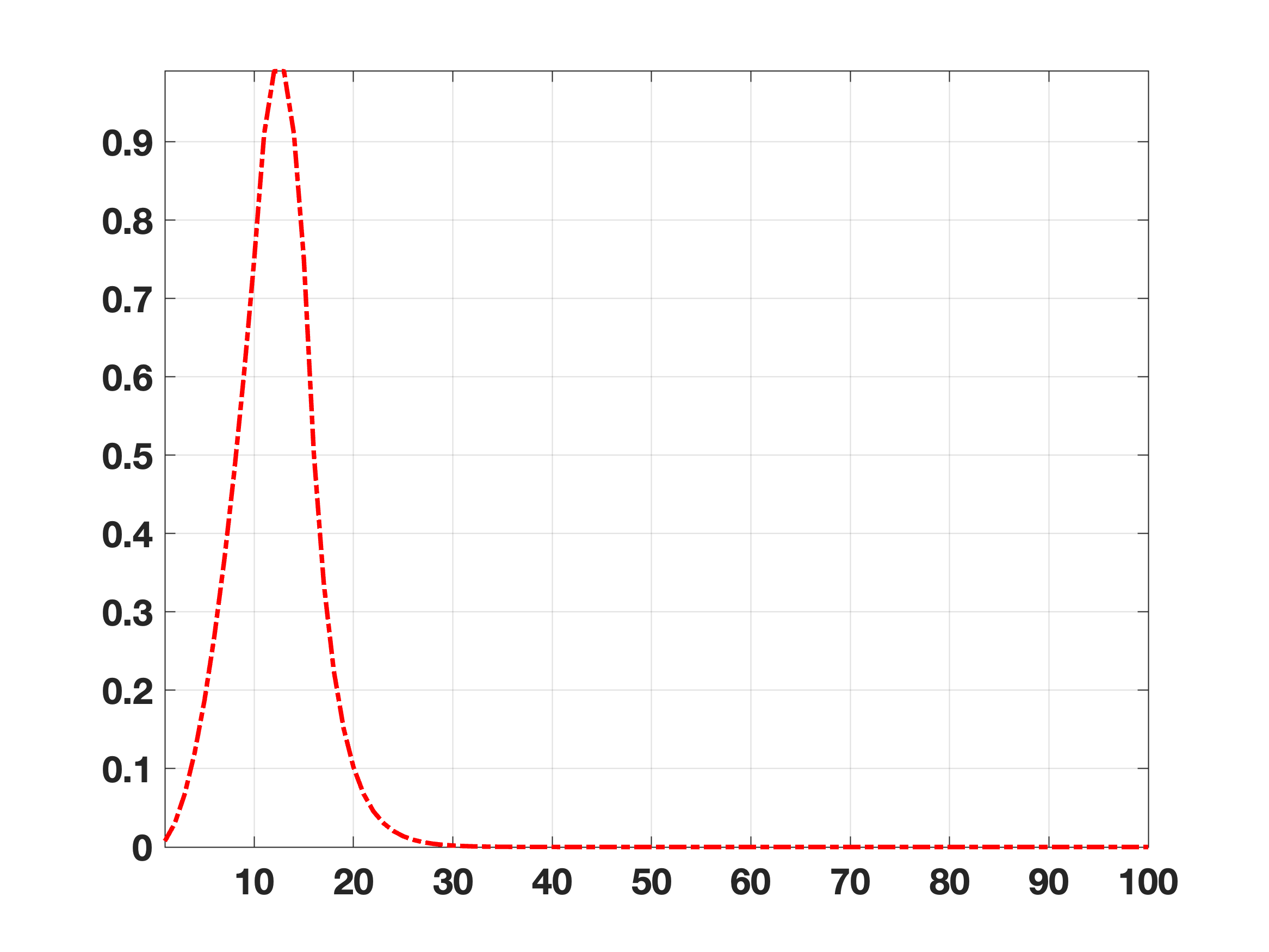} 
\includegraphics[width=4cm]{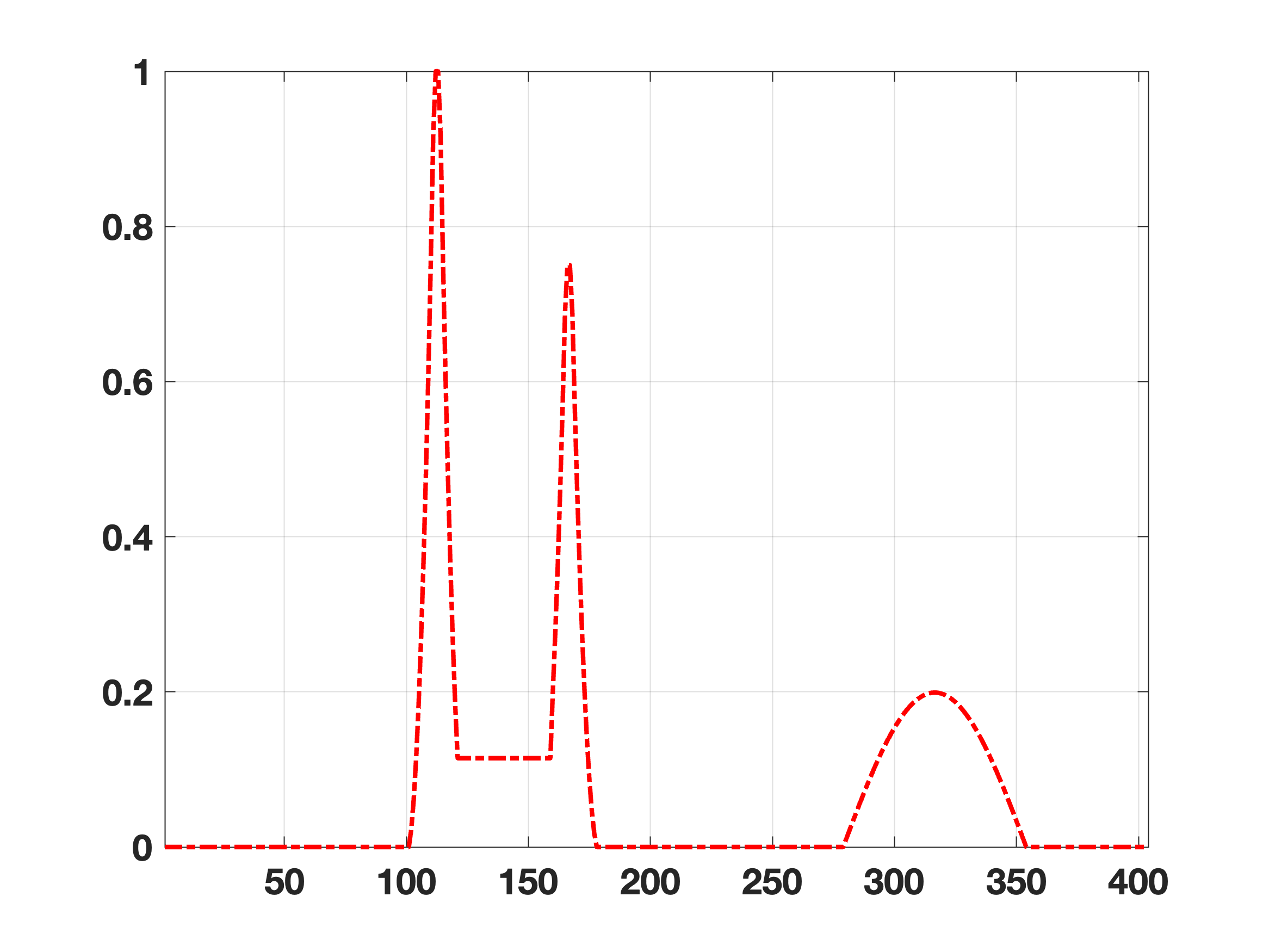}
\includegraphics[width=4cm]{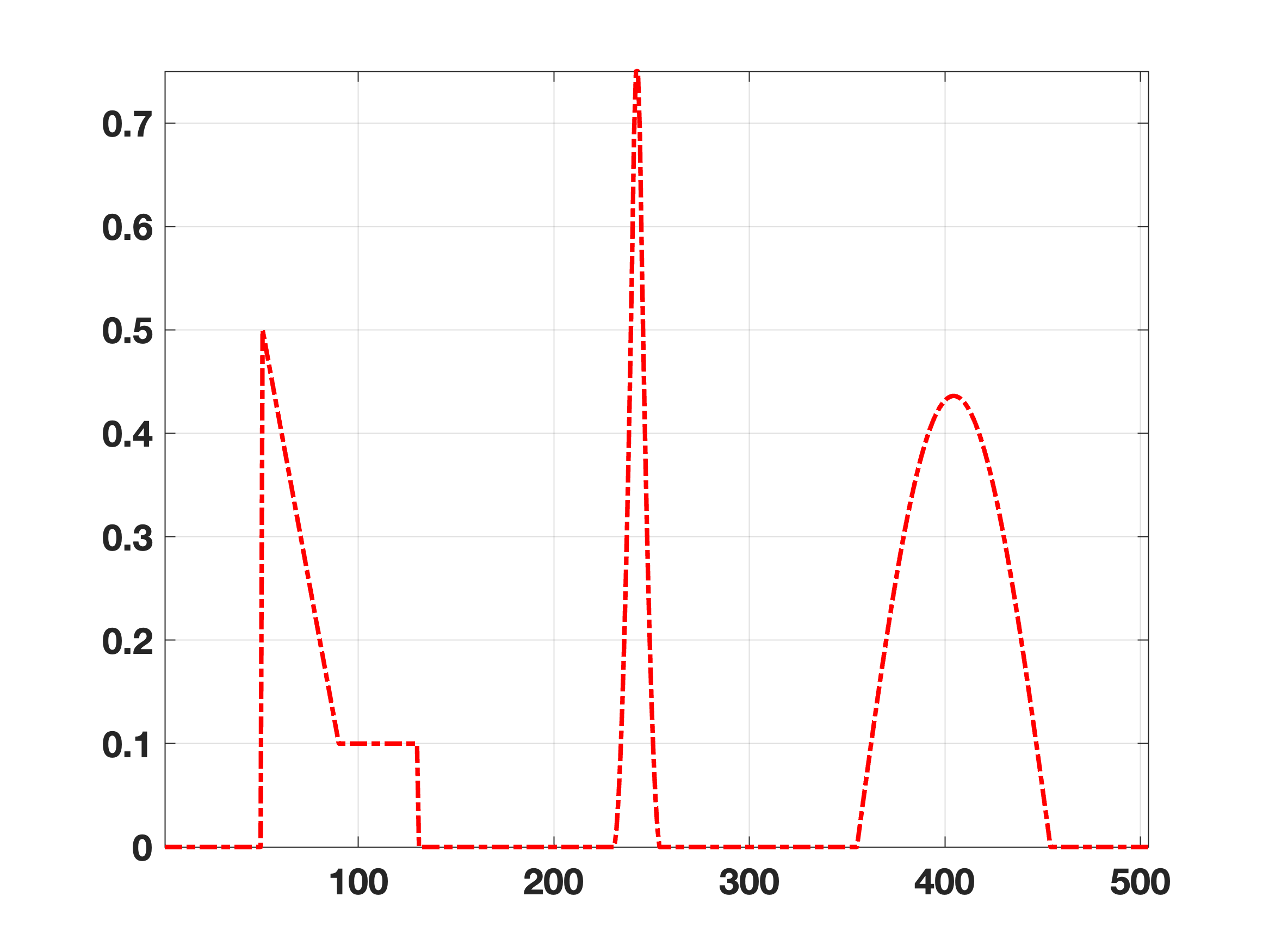}
\end{center}
\caption{1D test problems: ground truth $\bu^*$. Left {\tt T1}, center {\tt T2}, right {\tt T3}.}
 \label{fig:tests}
\end{figure}
%
%
%
For these three test problems, we observe the behaviour of Algorithm UPenMM with $p=N$  penalty functions as follows: 
\begin{equation}\label{eq:psifun}
 \psi_i(\bu)=(L\bu)_i^2+\epsilon, \quad i=1,\ldots,N   
\end{equation}
where $L\in\mathbb{R}^{N \times N}$ is the discretization of the second order derivative operator and $\epsilon$ is a small positive constant in the range $[10^{-7}, 10^{-5}] $. 
Moreover, we analyse the generalized approach proposed in Algorithm GUPenMM 
with $\varepsilon = 0.9$, $\psi_i$ as in \eqref{eq:psifun} and the following  generalized regularization parameter $\widetilde{\bl}$:  
\begin{equation}
    \widetilde{\lambda}_i = \frac{\|A\bu_{\bl}-\bb\|^2}{N \widetilde{\psi}_i(\bu_{\bl})},  \quad \widetilde{\psi}_i(\bu_{\bl}) = \max \; \{ {\psi}_{i-1}(\bu_{\bl}),{\psi}_i(\bu_{\bl}), {\psi}_{i+1}(\bu_{\bl})\}, \; i=1,\ldots,N.\label{eq:newpsi}
\end{equation}
Both algorithms are tested in the unconstrained ($\Omega=\R^N$) and non-negatively constrained ($\Omega=\R^N_+$ where $\R^N_+$ is the positive orthant) cases. 
A solution of the subproblem at step 3 of Algorithms UPenMM and GUPenMM has been computed, in the unconstrained case, by solving its first order conditions, while the Newton Projection method \cite{ber99} has been used in the constrained case.

The initial values $\bl_i^{(0)}$, $i=1,\ldots,N$, have been chosen such that
\begin{align*}
    \bl^{(0)}_i &= \displaystyle{\frac{\|A\bb-\bb\|^2}{N \psi_i(\bb)}} \quad \text{for Algorithm UPenMM}; \\
    \bl^{(0)}_i &=  \displaystyle{\frac{\|A\bb-\bb\|^2}{N \widetilde{\psi}_i(\bb)}} \quad \text{for Algorithm GUPenMM} .
\end{align*}
\section{Low-noise results}\label{sec:LN}
We begin by analyzing the behaviour of the two algorithms in the case of low noise, both in the unconstrained ($\Omega=\R^N$) and in the non-negatively constrained ($\Omega=\R^N_+$) scenarios. To show the amount of noise on the blurred data we represent in Figure \ref{fig:tests_noise} the blurred noisy data $\bb$ (blue line), and the blurred data $\by$.
\begin{figure}[htbp]
\begin{center}
\includegraphics[width=4cm]{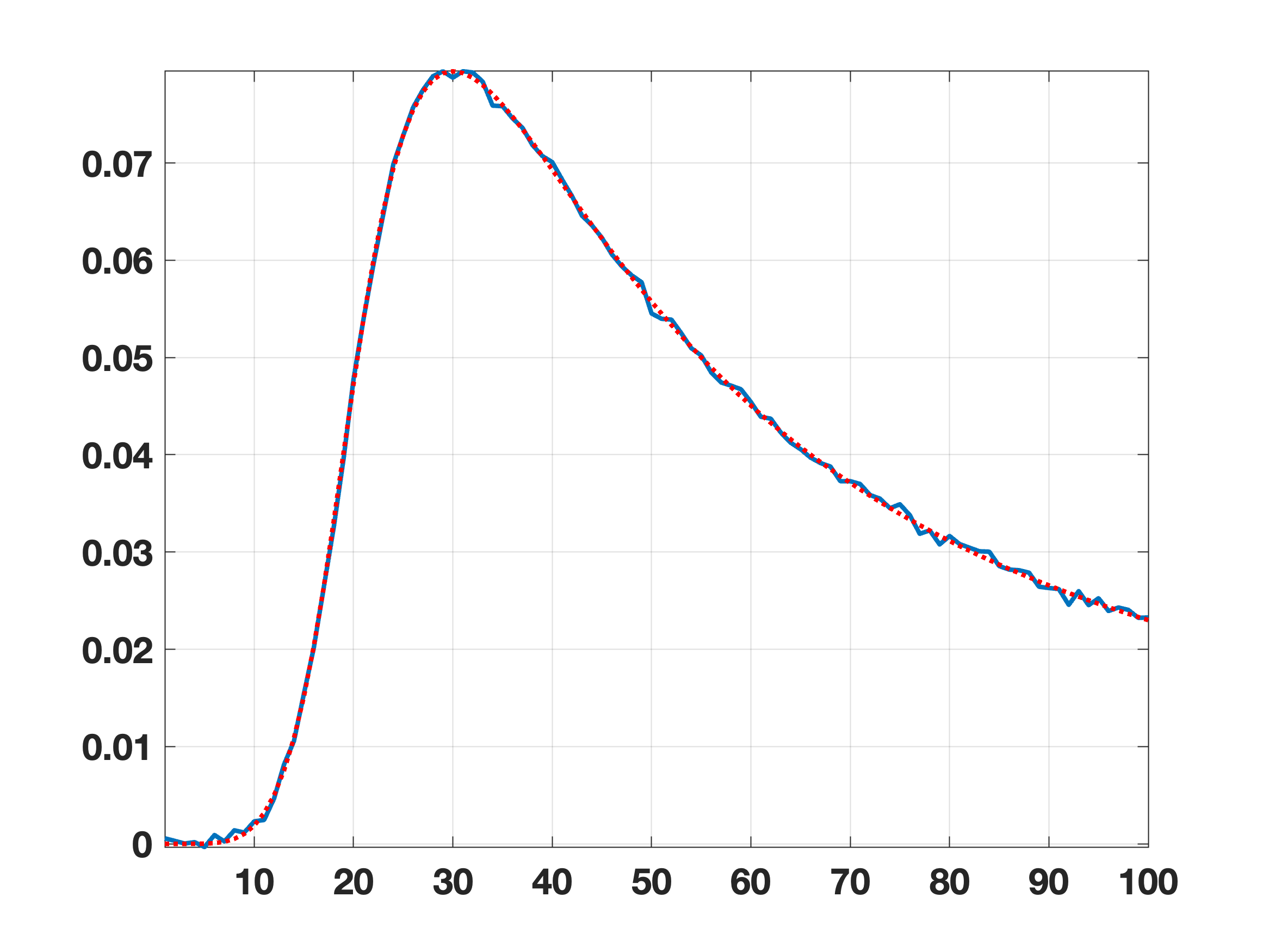} 
\includegraphics[width=4cm]{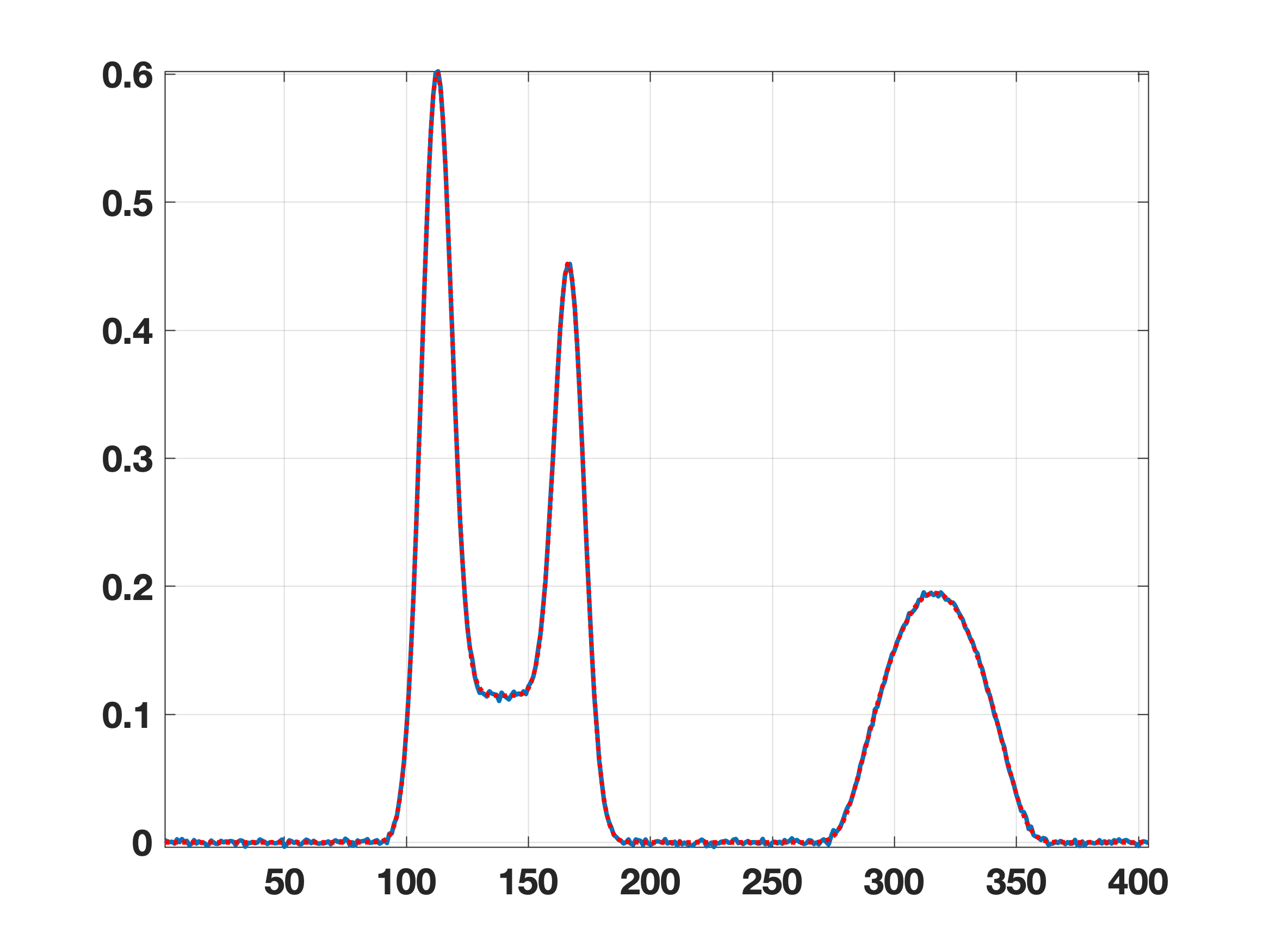} 
\includegraphics[width=4cm]{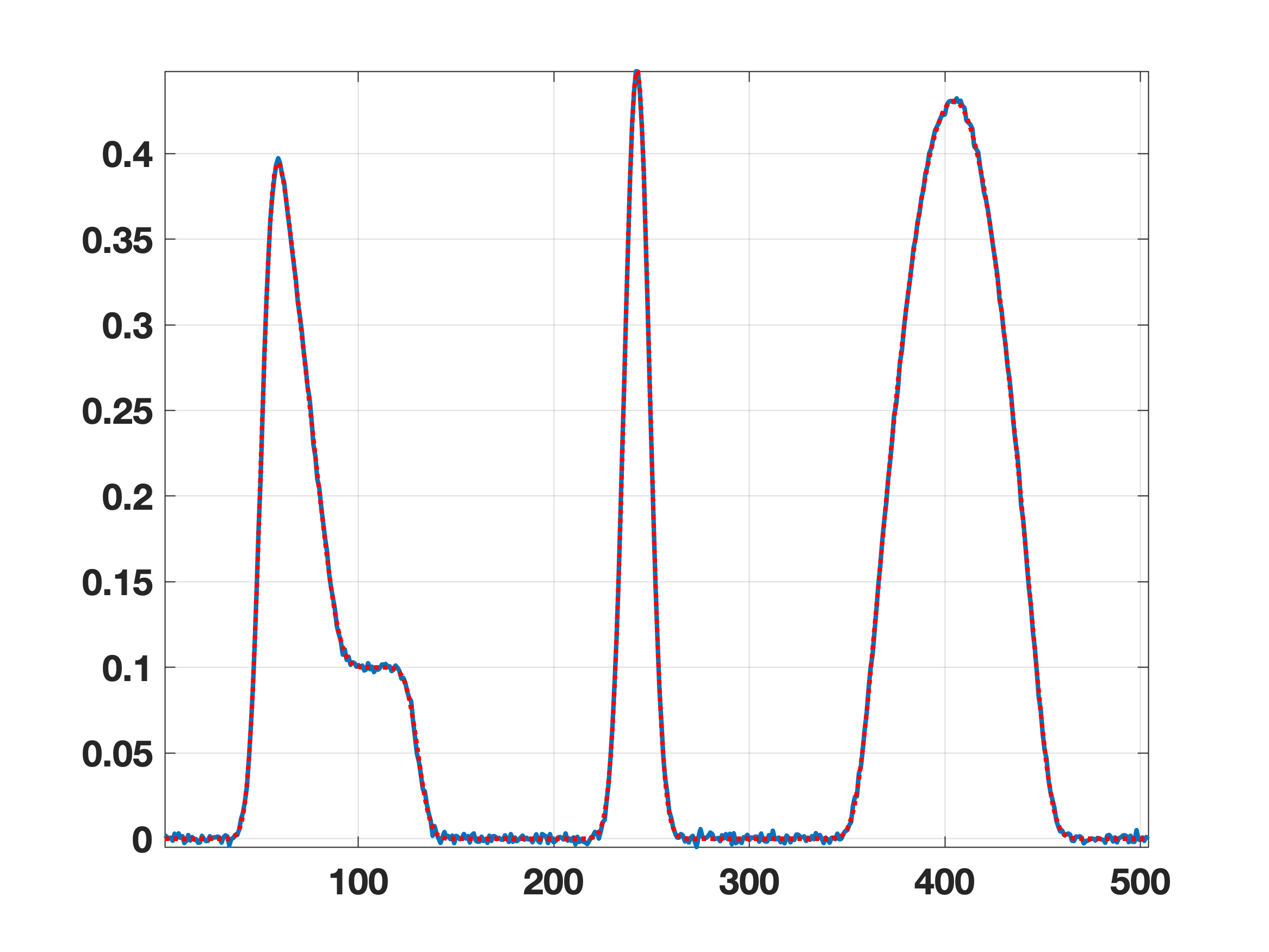} 
\end{center}
\caption{1D test problems ($\delta = 0.01$). Blue line: blurred noisy data $\bb$, dotted red line: blurred data $\by$. Left {\tt T1}, center {\tt T2}, right {\tt T3}.}
 \label{fig:tests_noise}
\end{figure}

A first evaluation of the methods is made by inspecting the error history represented in Figure \ref{fig:T1_errs}(a) when $\Omega=\R^N$ and in Figure \ref{fig:T1_errs}(b) for $\Omega=\R^N_+$. Each picture shows the relative error curves of Algorithm UPenMM (black dash-dot line)   and Algorithm GUPenMM (blue dashed line). The red dots represent the relative error values corresponding to  the following {\tt stopping criterion}:
\begin{equation}\label{eq:SC}
\| \Veta^{(k+1)}-\Veta^{(k)} \| \leq \| \Veta^{(k)} \| Tol_{\lambda} 
\end{equation}
with $Tol_{\lambda}=10^{-2}.$
%
\begin{figure}[htbp]
\begin{center}
(a) \hspace{5cm} (b) \\
   \includegraphics[width=6cm]{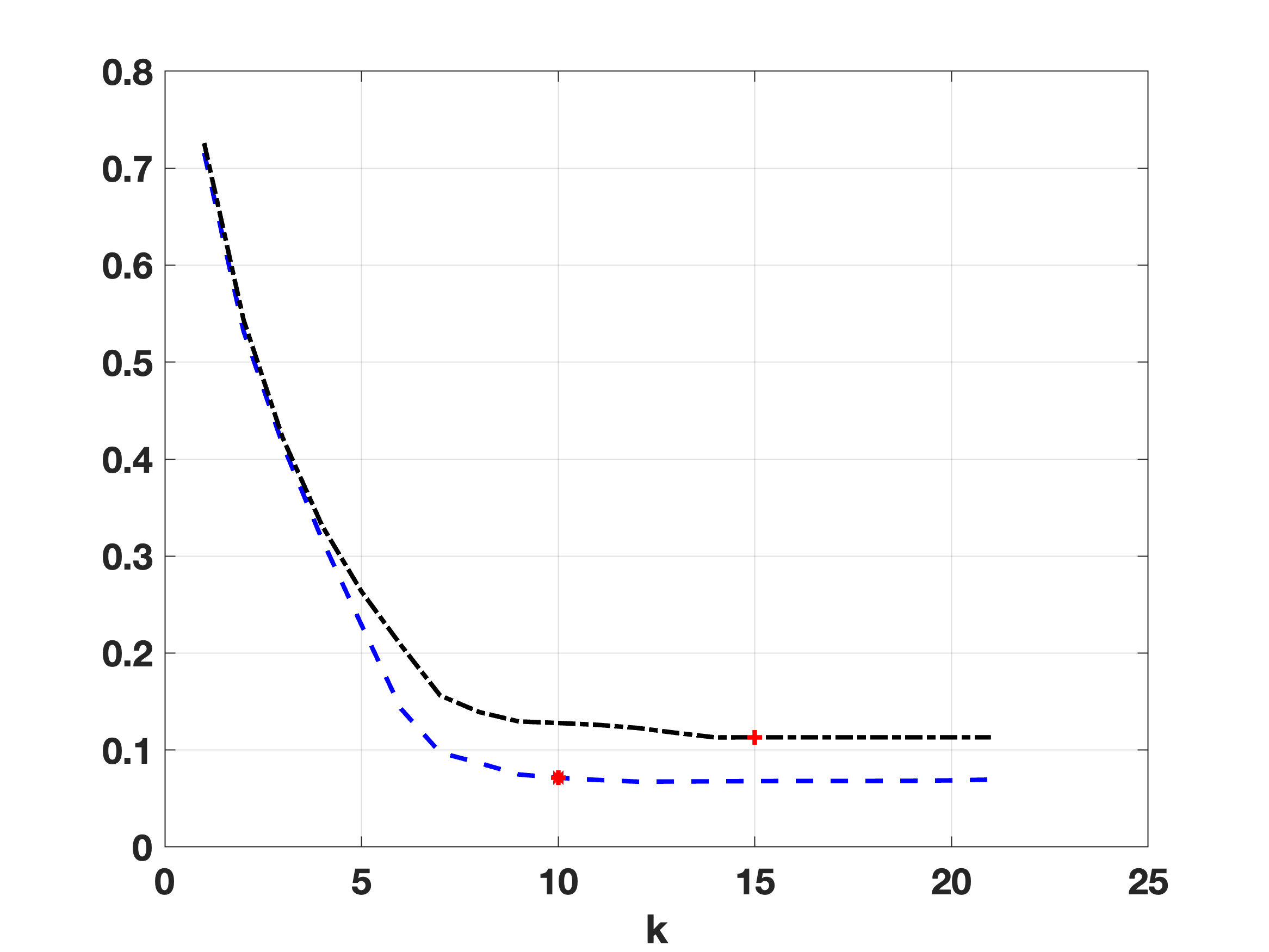}\hspace{-.6cm}
   \includegraphics[width=6cm]{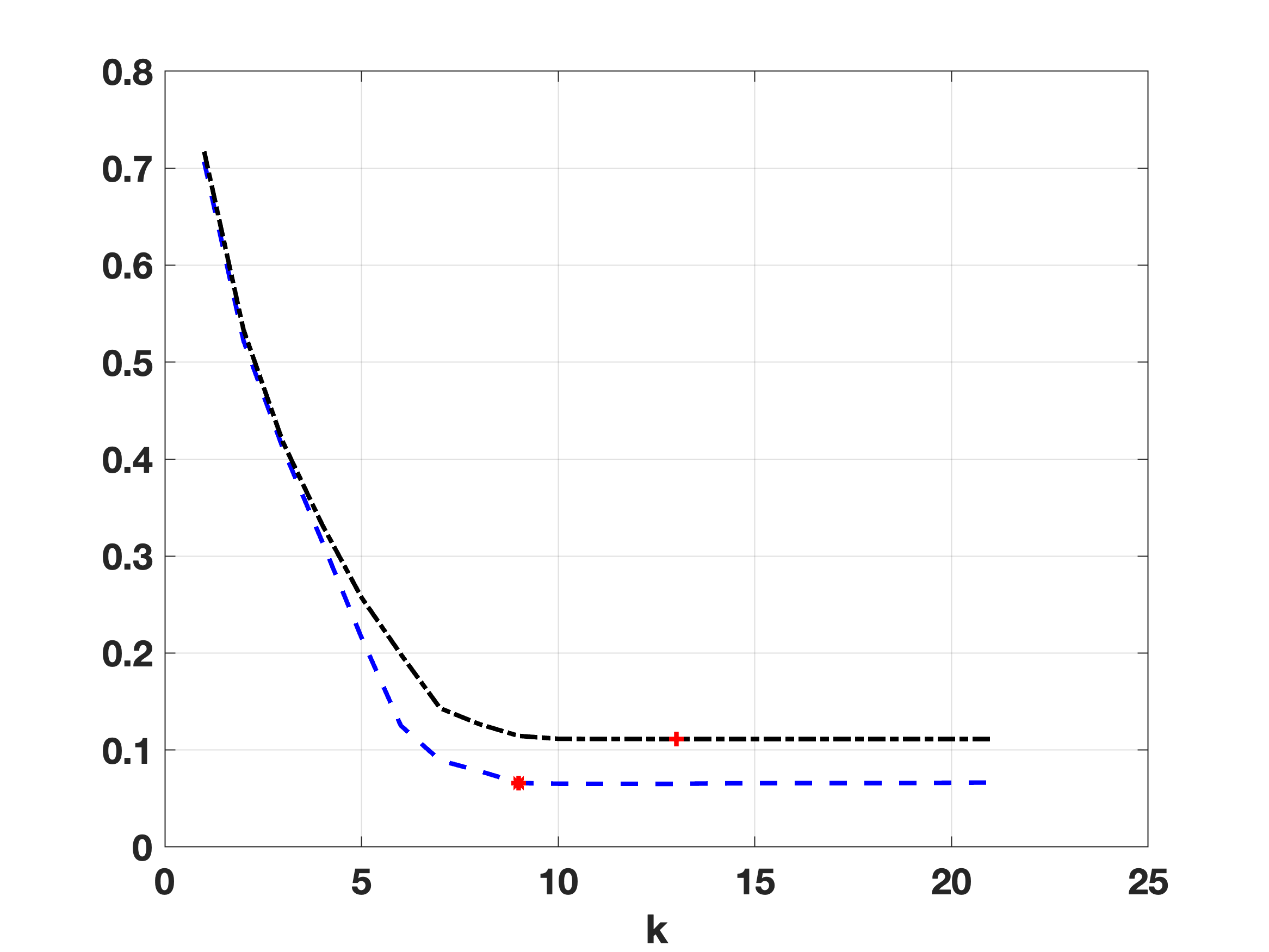}
\end{center}
\caption{Test problem {\tt T1} ($\delta = 0.01$). Relative error curves vs. the iteration index $k$.  (a) Case $\Omega=\R^N$. (b) Case $\Omega=\R^N_+$. The black dash-dot line represents the relative error values of Algorithm UPenMM and the blue dashed line represents the relative error values of Algorithm GUPenMM. The red dots indicate the relative error values corresponding to stopping criterion \eqref{eq:SC}.}
 \label{fig:T1_errs}
\end{figure}
We can appreciate the positive influence of the modified evaluation of the regularization parameters introduced in Algorithm GUPenMM producing smaller relative errors and faster convergence.
The plot of the values of the surrogate function is shown in Figure \ref{fig:T1_Q}, demonstrating that the convergence criterion is met.
\begin{figure}[htbp]
\begin{center}
(a) \hspace{5cm} (b) \\
   \includegraphics[width=6cm]{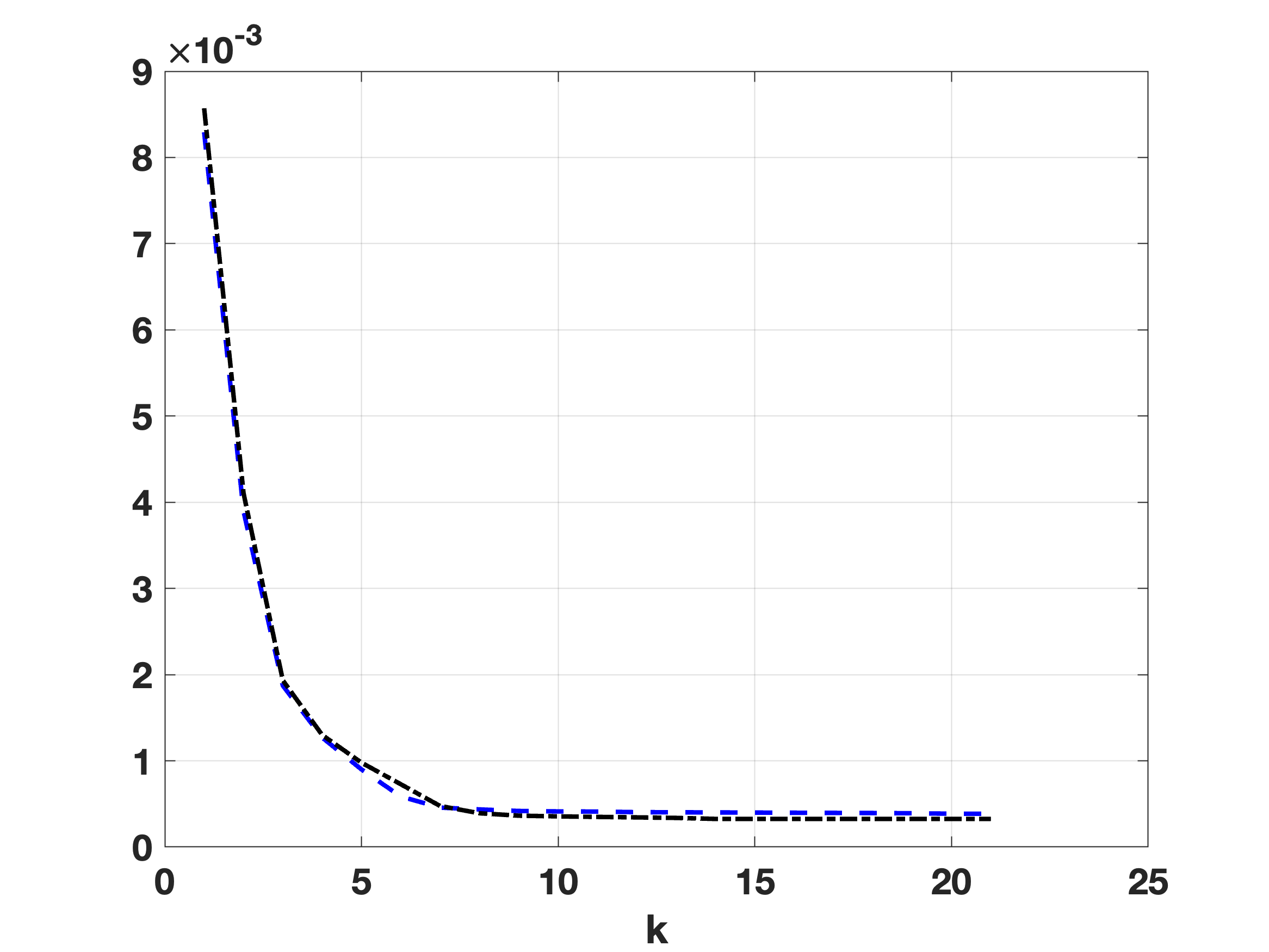}\hspace{-.6cm}
   \includegraphics[width=6cm]{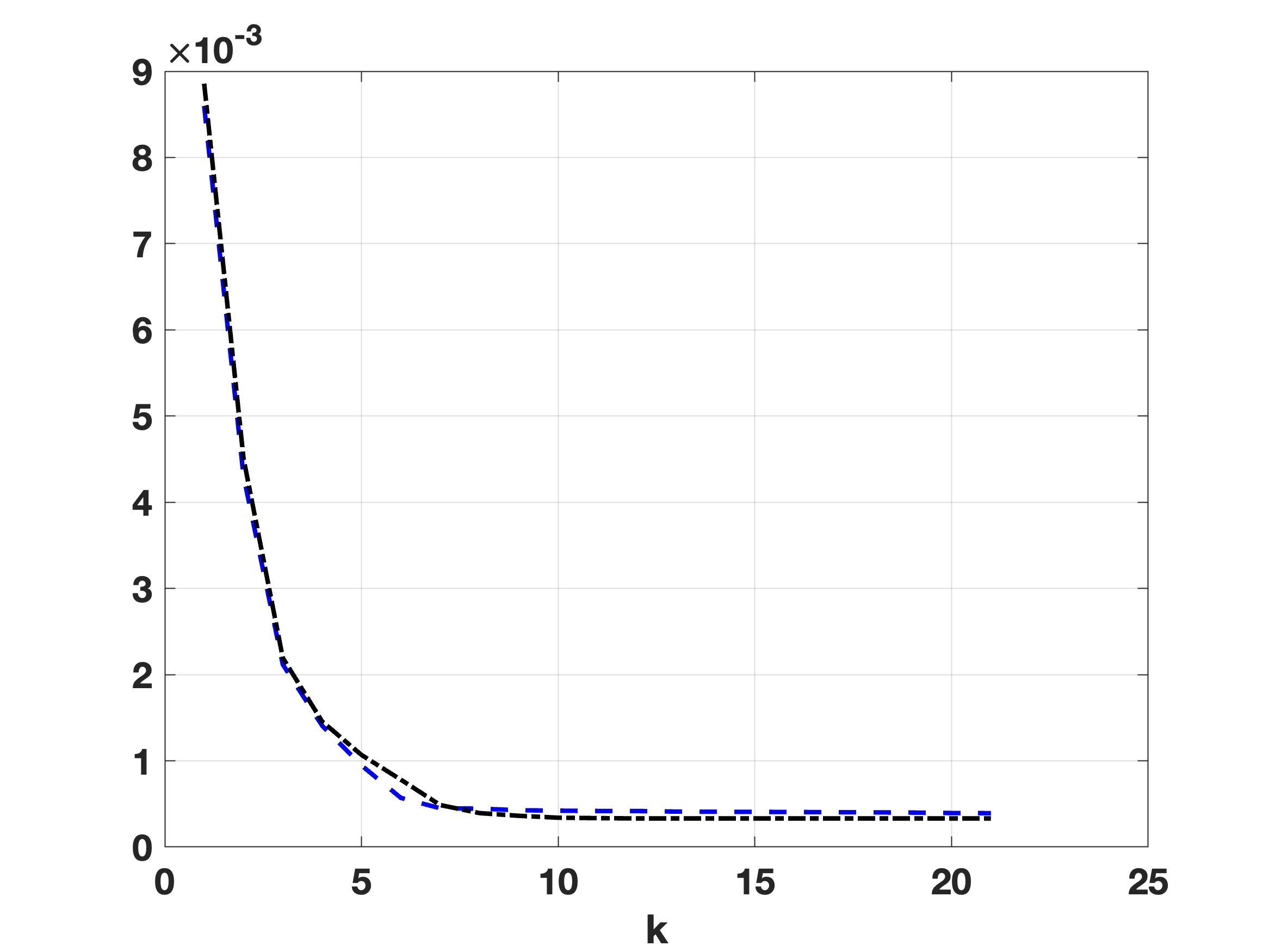}
\end{center}
\caption{Test problem {\tt T1} ($\delta = 0.01$). Value of the surrogate function. (a) $\Omega=\R^N$ (b) $\Omega=\R^N_+$. The black dash-dot line represents the $Q(\widehat{\bl}^{(k+1)},\bl^{(k)})$ values in  Algorithm UPenMM and the blue dashed line represents the $Q(\widetilde{\bl}^{(k+1)},\bl^{(k)})$ values in Algorithm GUPenMM.}
 \label{fig:T1_Q}
\end{figure}
The good agreement of the residual curves with the noise norm is observed in Figure \ref{fig:T1_res} where the black dash-dot line represents the residual norm values in Algorithm UPenMM and the blue dashed line indicates the residual norm values in Algorithm GUPenMM. The red line represents the noise norm.
\begin{figure}[htbp]
\begin{center}
(a) \hspace{5cm} (b) \\
\includegraphics[width=6cm]{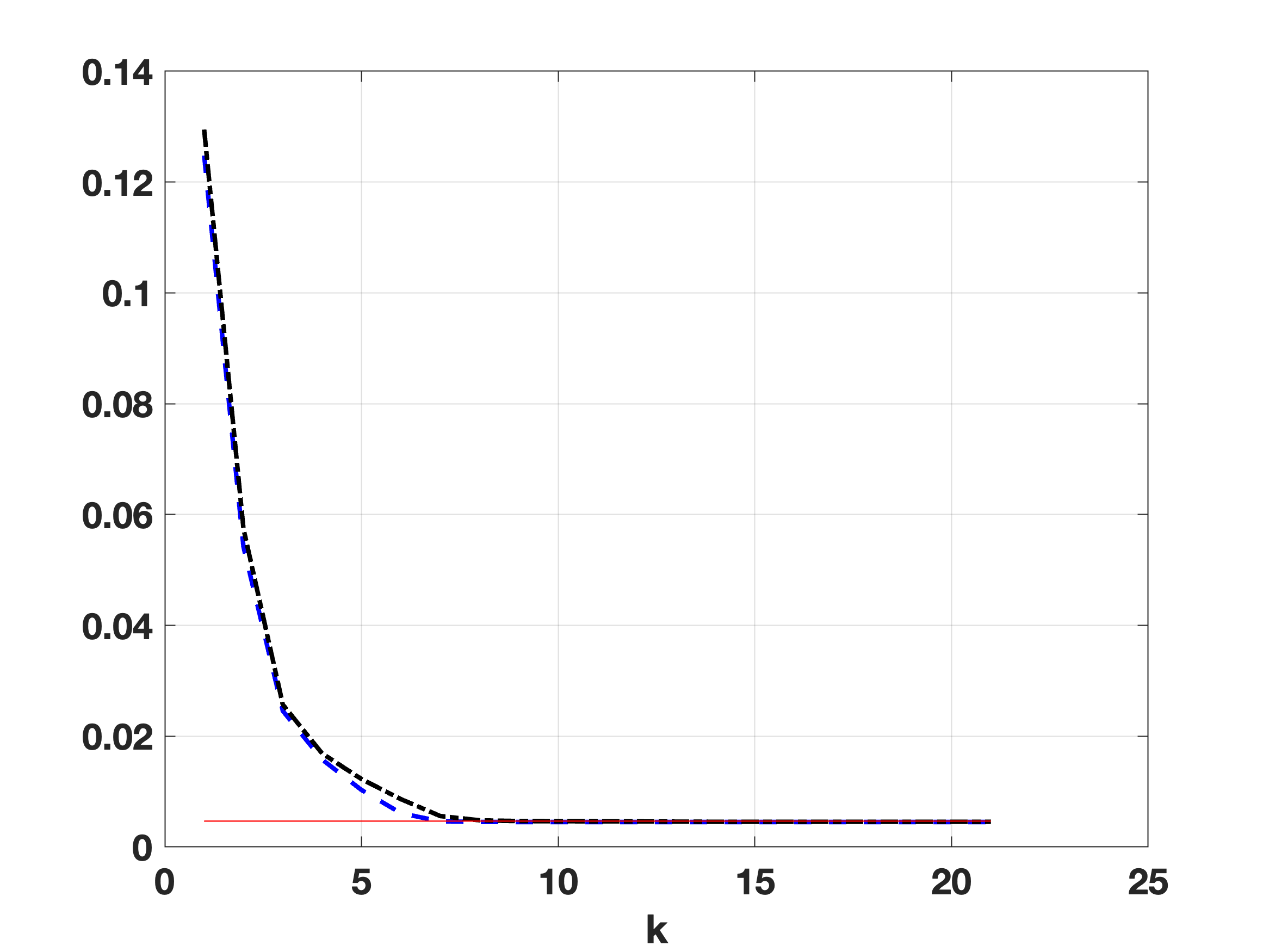}\hspace{-.6cm}
   \includegraphics[width=6cm]{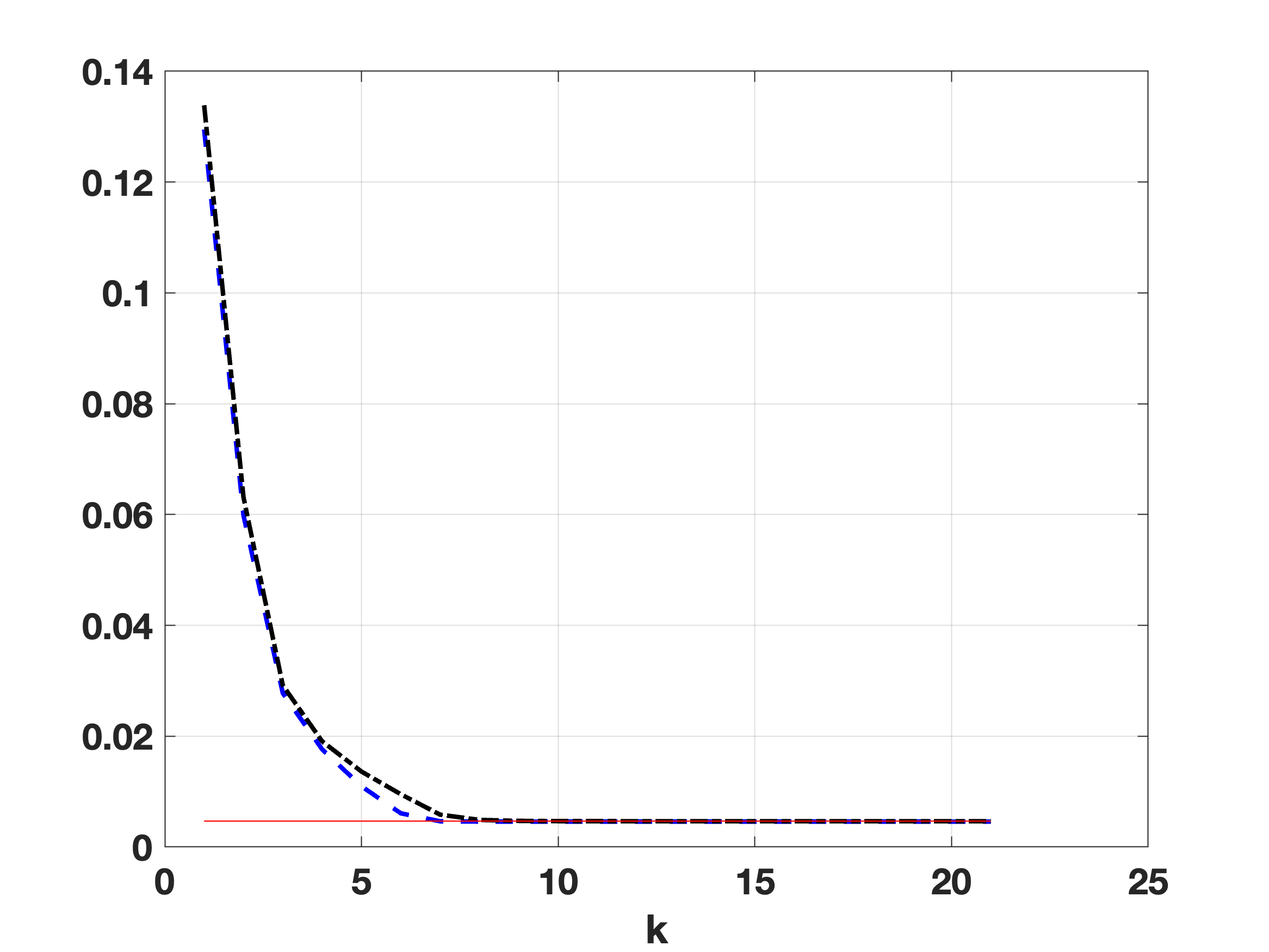}
\end{center}
\caption{Test problem {\tt T1} ($\delta = 0.01$). Behaviour of the residual norm. (a) $\Omega=\R^N$ (b) $\Omega=\R^N_+$. The black dash-dot line represents the residual norms in Algorithm UPenMM and the blue dashed line is the residual norms in Algorithm GUPenMM. The red line
represents the noise norm.}
\label{fig:T1_res}
\end{figure}
In addition to the error curves, the values in Table \ref{tab:compar35_minerr} and the plots of the reconstructed signals in Figure \ref{fig:T1_sol} confirm the superior quality obtained by Algorithm GUPenMM. 
\begin{figure}[htbp]
\begin{center}
(a) \hspace{5cm} (b) \\
\includegraphics[width=6cm]{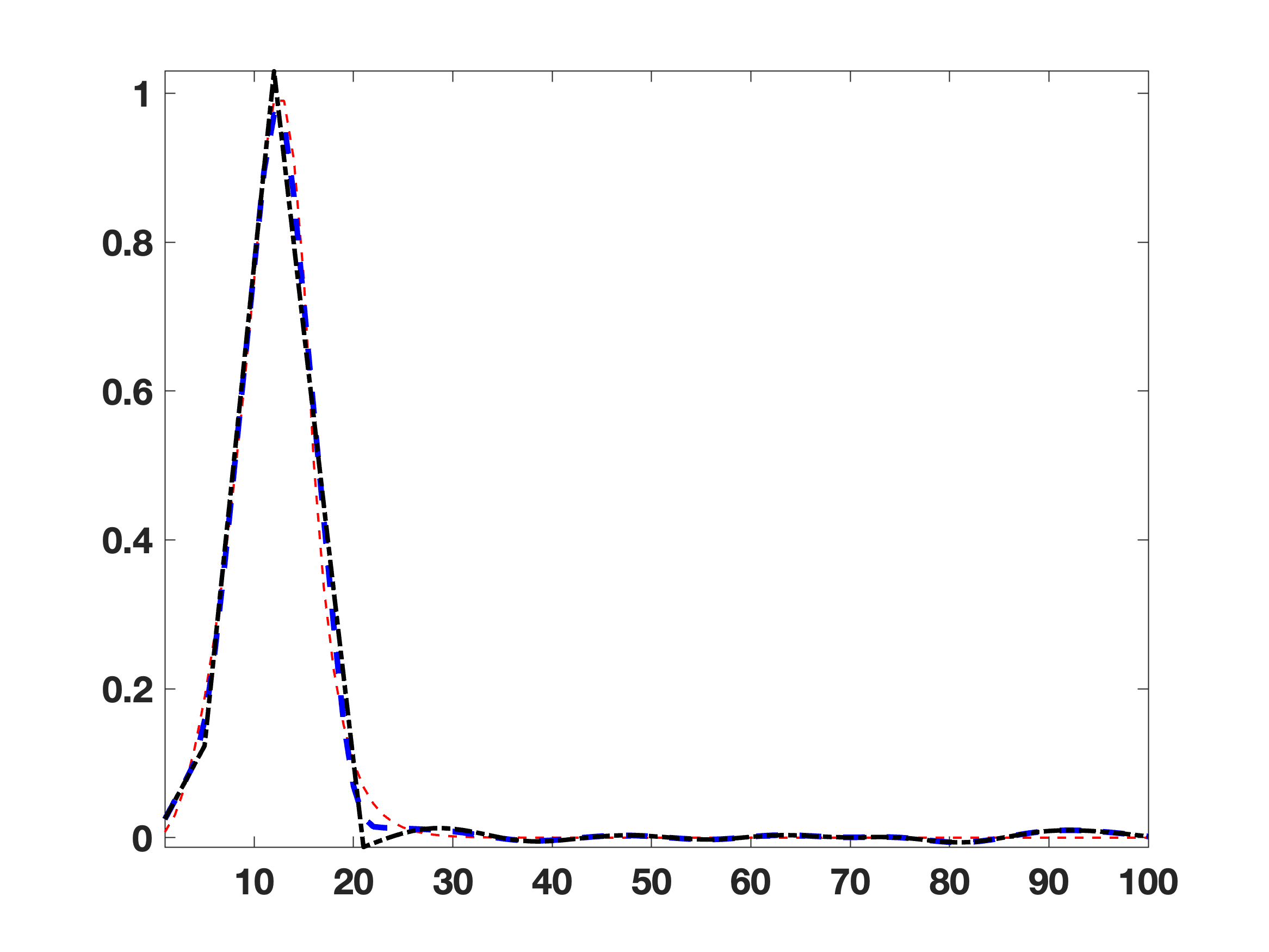}\hspace{-.6cm}
\includegraphics[width=6cm]{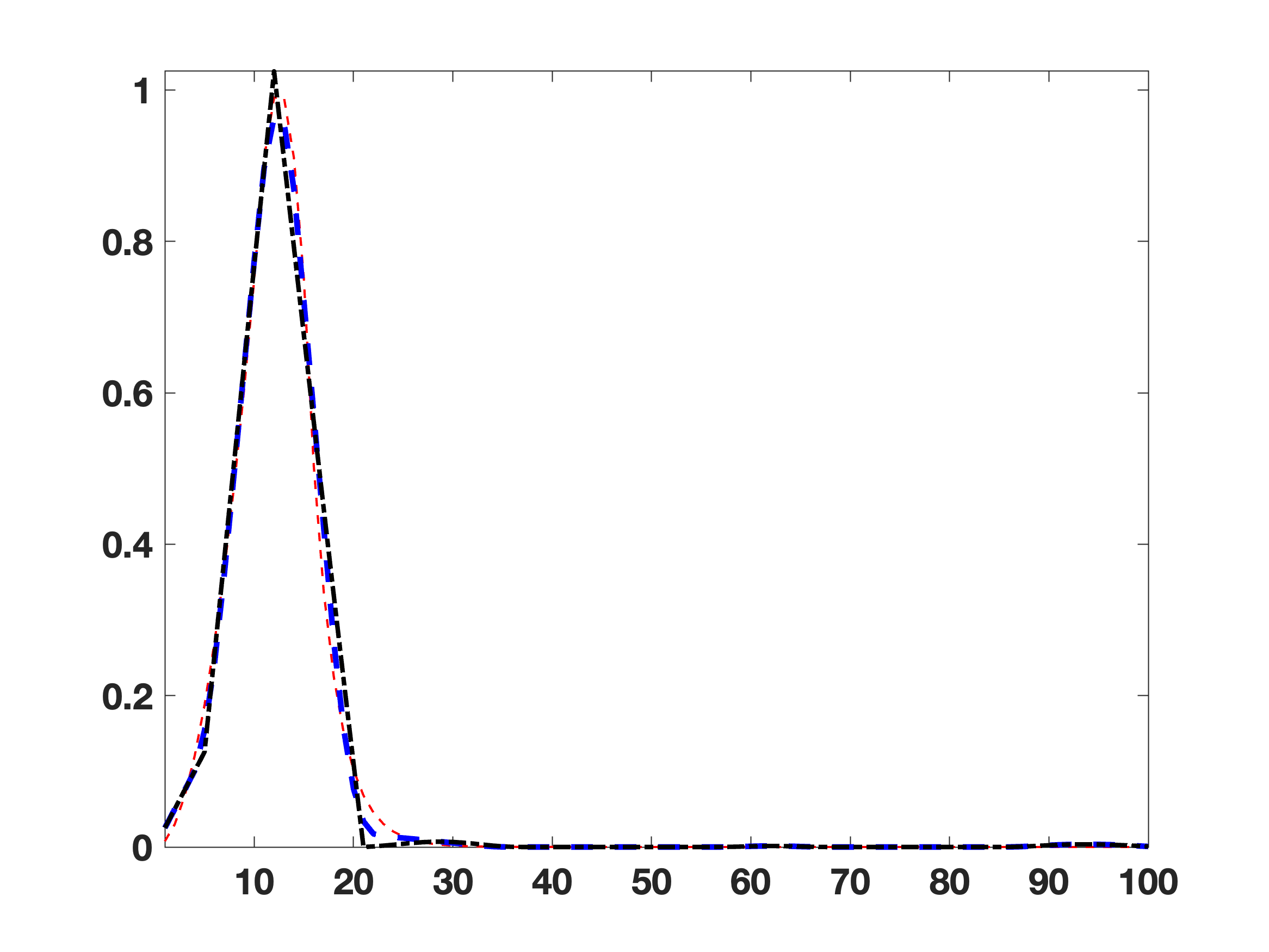}
\end{center}
\caption{Test problem {\tt T1} ($\delta = 0.01$). Solutions $\bu$. (a) $\Omega= \R^N$ (b) $\Omega= \R^N_+$. The black dash-dot line represents  $\bu$ computed by Algorithm UPenMM and  the blue dashed line represents  $\bu$ computed by Algorithm GUPenMM. The red dashed line represents the ground-truth solution.}
\label{fig:T1_sol}
\end{figure}
This feature is even more evident in the case of more complex signals such as those represented by test problems T2 and T3; in these cases, we only represent the reconstructed signals for $\Omega= \R^N_+$ in Figure \ref{fig:T2T3_sol}.
\begin{figure}[htbp]
\begin{center}
(a) \hspace{5cm} (b) \\
\includegraphics[width=6cm]{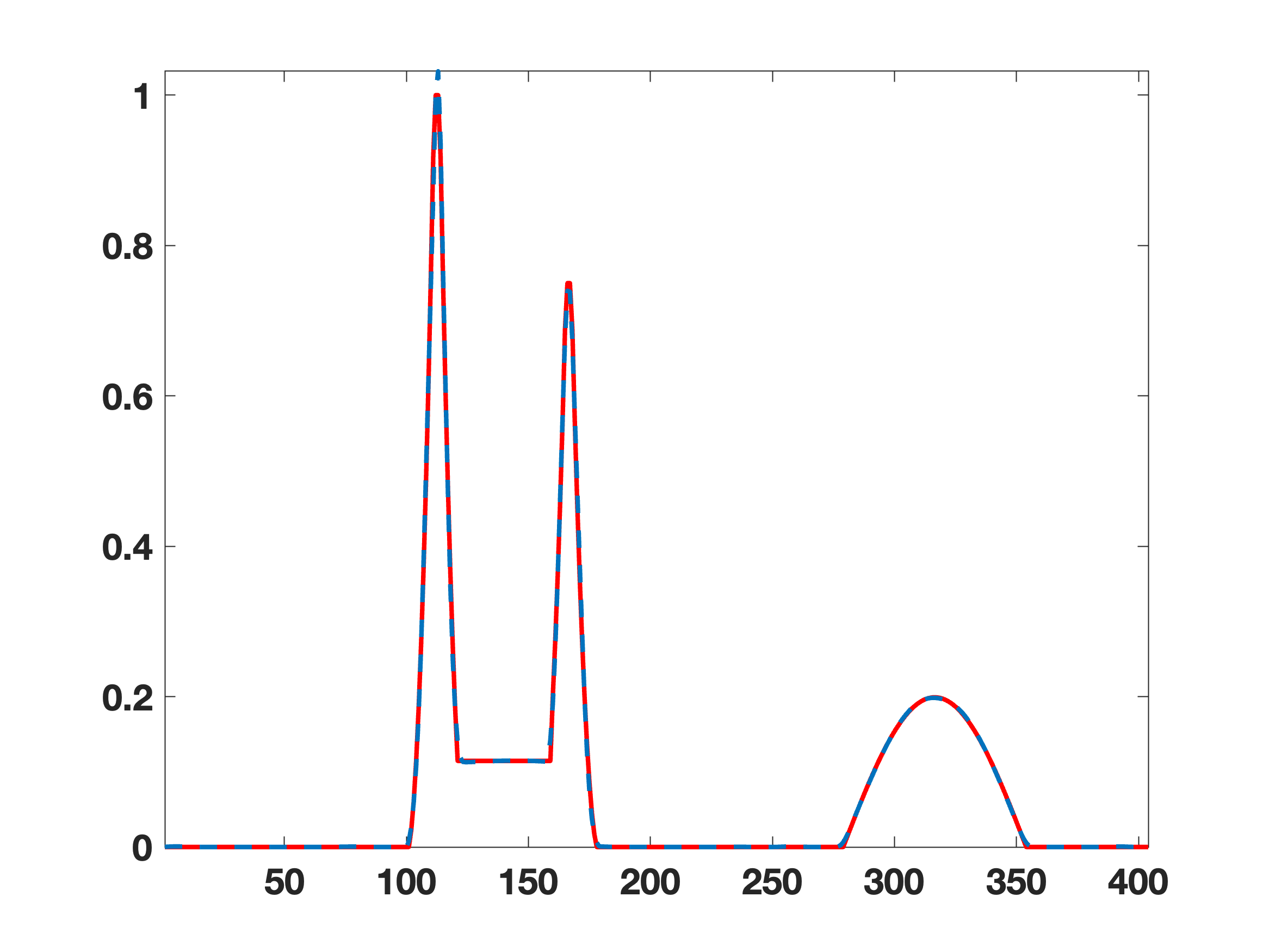}\hspace{-.5cm}
   \includegraphics[width=6cm]{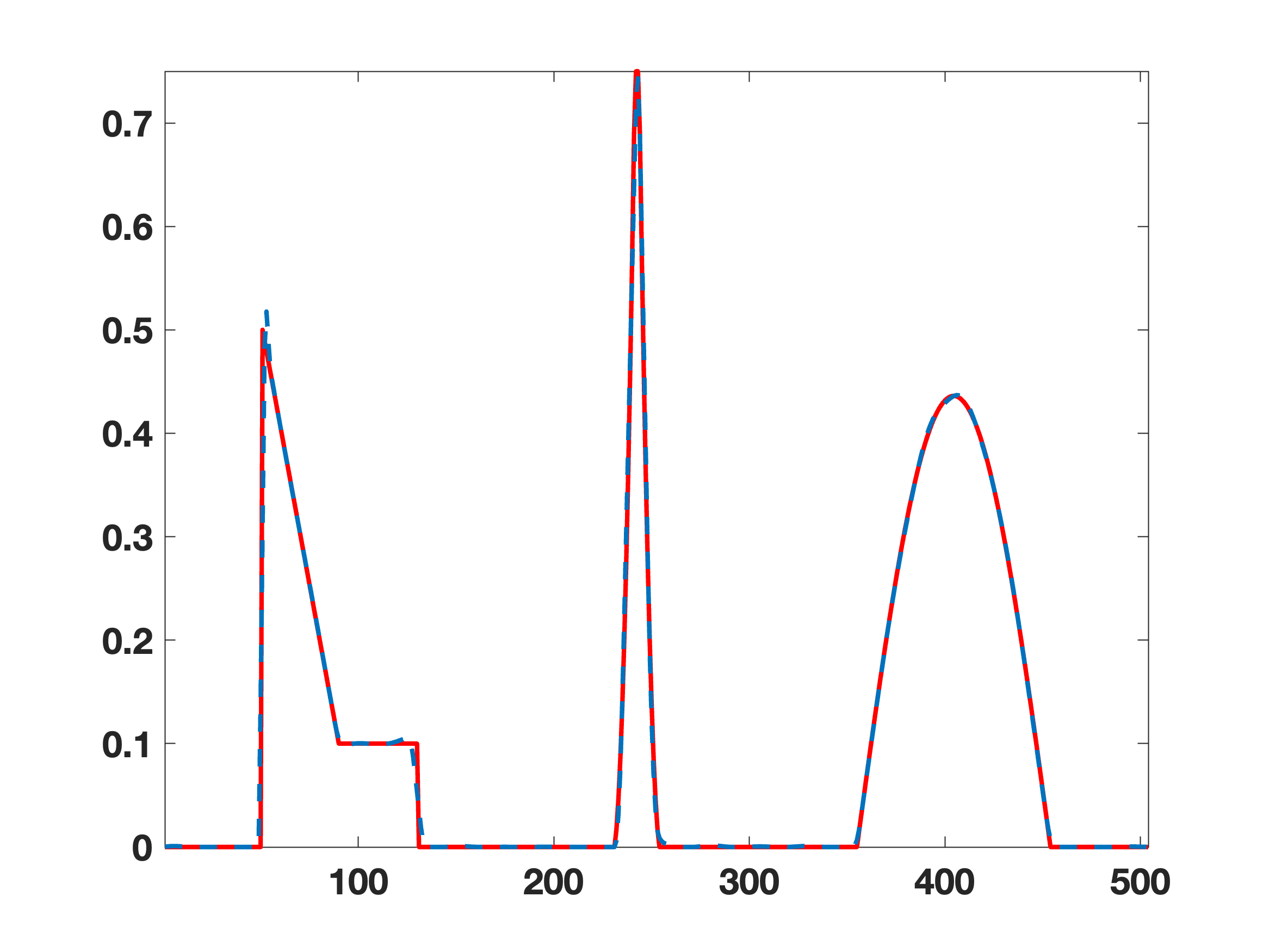}
\end{center}
\caption{Solutions $\bu$ by Algorithm GUPenMM + (blue dashed line) and ground-truth solution (red dashed line). (a) Test problem {\tt T2} ($\delta = 0.01$) (b) Test problem {\tt T3} ($\delta = 0.01$).}
\label{fig:T2T3_sol}
\end{figure}
\begin{table}[h!]
{\small
\begin{center}
\begin{tabular}{|l|c|l|c|l|c|l|}
\hline
 \multirow{ 2}{*}{Method}   & \multicolumn{2}{|c|}{\texttt{T1}} & \multicolumn{2}{|c|}{\texttt{T2}} & \multicolumn{2}{|c|}{\texttt{T3}}\\
 \cline{2-7}
              & Rel Err  &   Iters          & Rel Err  &   Iters           & Rel Err  &   Iters  \\
\hline
 UPenMM   &  1.1313$\cdot 10^{-1}$  &  15     &  9.8773$\cdot 10^{-2}$  &  7      &  1.0532$\cdot 10^{-1}$  & 9       \\
 UPenMM + &  1.1136$\cdot 10^{-1}$  &  13(6) &  8.6412$\cdot 10^{-2}$  &  6 (42) &  1.0188$\cdot 10^{-1}$  & 11(47) \\
 GUPenMM  &  7.1381$\cdot 10^{-2}$  &  10     &  2.5204$\cdot 10^{-2}$  &  7      &  6.2736$\cdot 10^{-2}$  & 18      \\
 GUPenMM +&  6.5993$\cdot 10^{-2}$  &  9(48) &  2.7592$\cdot 10^{-2}$  &  9(66) &  6.1298$\cdot 10^{-2}$  & 19(65) \\
 L2       &  1.1527$\cdot 10^{-1}$  &  /      &  7.9064$\cdot 10^{-2}$  &  /      &  1.0380$\cdot 10^{-1}$  & / \\
 L2  +    &  7.9265$\cdot 10^{-2}$  &  /      &  6.1054$\cdot 10^{-2}$  &  /      &  8.9234$\cdot 10^{-2}$  & / \\
\hline
\end{tabular}
\vspace*{3mm}
\caption{1D test problems ($\delta = 0.01$). Relative error and the corresponding number of iterations with $Tol_{\lambda}=10^{-2}$. The symbol $+$ indicates the constrained case. Between parenthesis, the number of Newton Projection iterations is reported. \label{tab:compar35_minerr}}
\end{center}
}
\end{table}
%
%
The computation times presented in Table~\ref{tab:comTimes} show the performance of the different methods. The computation times for UPenMM and GUPenMM indicate a relatively stable increase across the test scenarios from T1 to T3. This increment is expected as the size of the test problems increases from T1 to T3.
When the non-negativity constraint is present (``+'' versions), computation times for T2 and T3 increase in both UPenMM+ and GUPenMM+.
 Despite this increase, Table~\ref{tab:compar35_minerr} shows that the accuracy generally improves.
  This suggests that the non-negativity constraint, while increasing computation time, enhance the solution's precision.
\begin{table}[h!]
{\small
\begin{center}
\begin{tabular}{|l|ccc|}
\hline
 \multirow{ 2}{*}{Method}   & \multicolumn{3}{|c|}{Computation times}\\
 \cline{2-4}
              & T1  & T2  & T3  \\
\hline
 UPenMM   & 3.01$\cdot 10^{-2}$  & 5.53$\cdot 10^{-2}$ & 1.09$\cdot 10^{-1}$ \\
 UPenMM + & 2.70$\cdot 10^{-2}$  & 3.00$\cdot 10^{0}\;\;$ & 2.75$\cdot 10^{0}\;\;$ \\
 GUPenMM  & 2.75$\cdot 10^{-2}$  & 1.66$\cdot 10^{-1}$ & 3.06$\cdot 10^{-1}$  \\
 GUPenMM +& 2.78$\cdot 10^{-2}$ & 3.00$\cdot 10^{0}\;\;$ &  4.86$\cdot 10^{0}\;\;$\\
\hline
\end{tabular}
\vspace*{3mm}
\caption{1D test problems ($\delta = 0.01$). Computation times in seconds.  The symbol $+$ indicates the constrained case. \label{tab:comTimes}}
\end{center}
}
\end{table}
A final consideration concerns the values of the point-wise regularization parameters. Figure \ref{fig:regpar} illustrates the characteristic of this method: since all the penalty terms are constant, then the  values of the regularization parameters are larger in correspondence with flat areas and become smaller where the solution exhibits rapid changes.
\begin{figure}[htbp]
\begin{center}
\includegraphics[width=4cm]{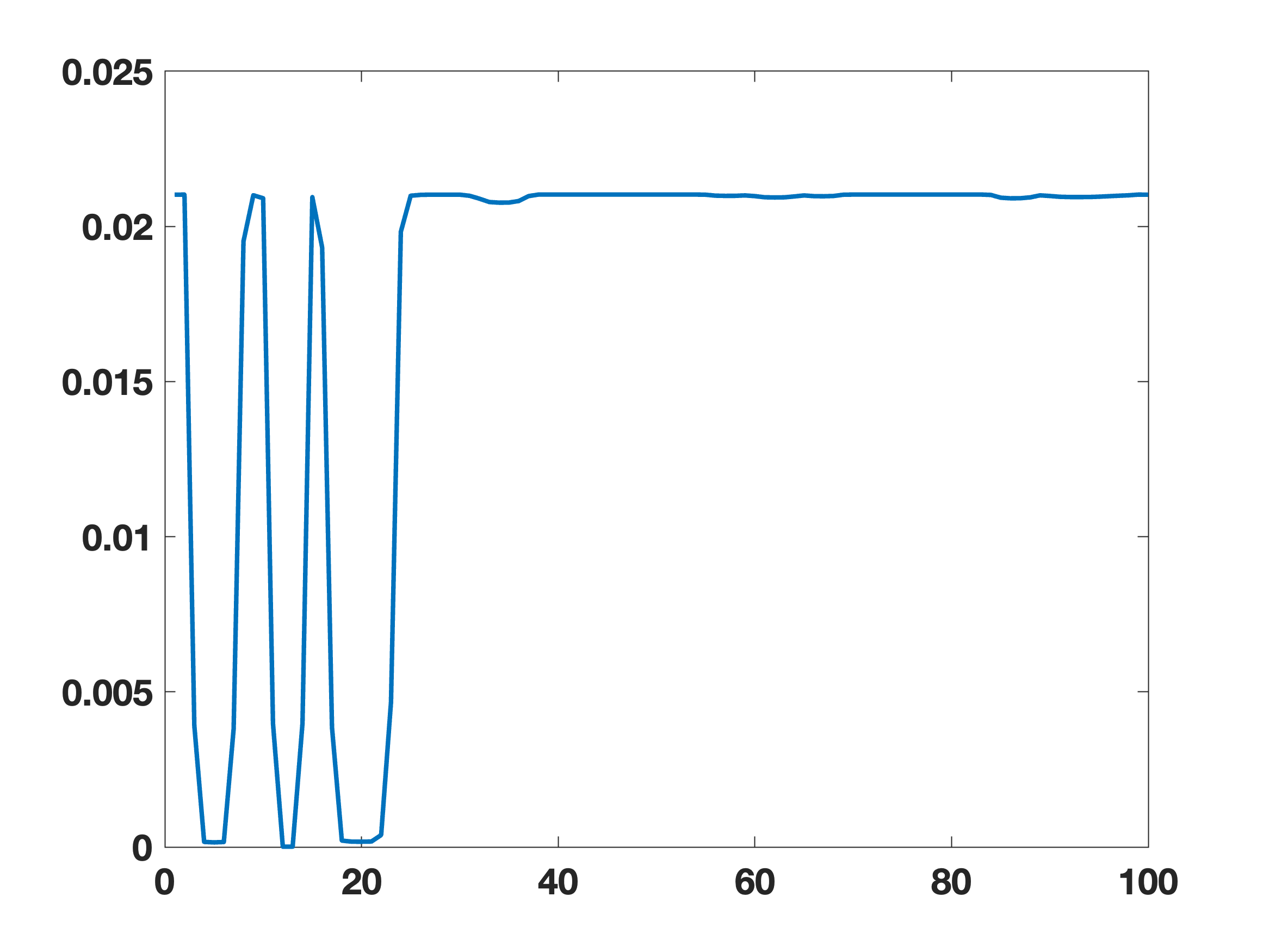}
\includegraphics[width=4cm]{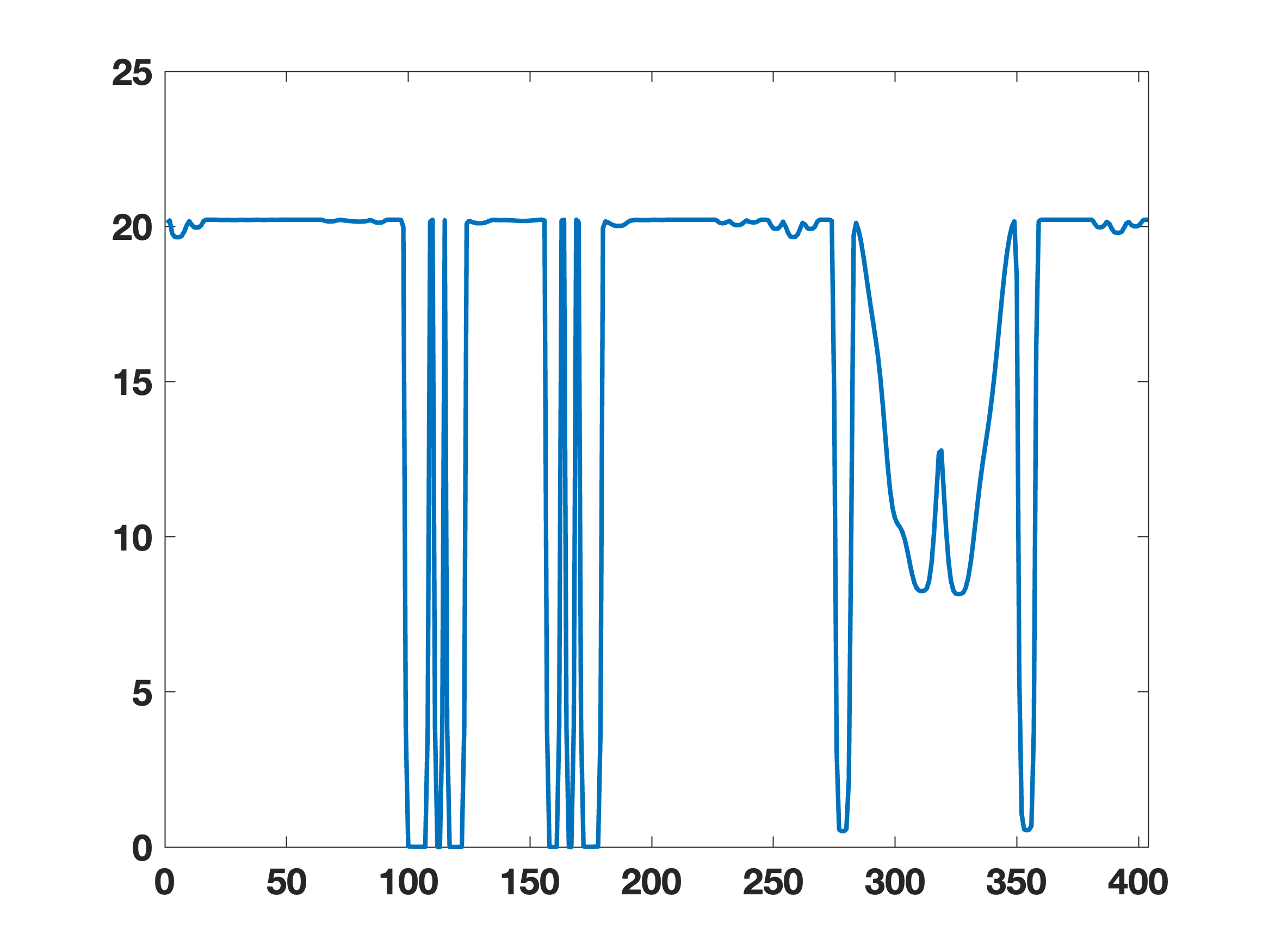}
\includegraphics[width=4cm]{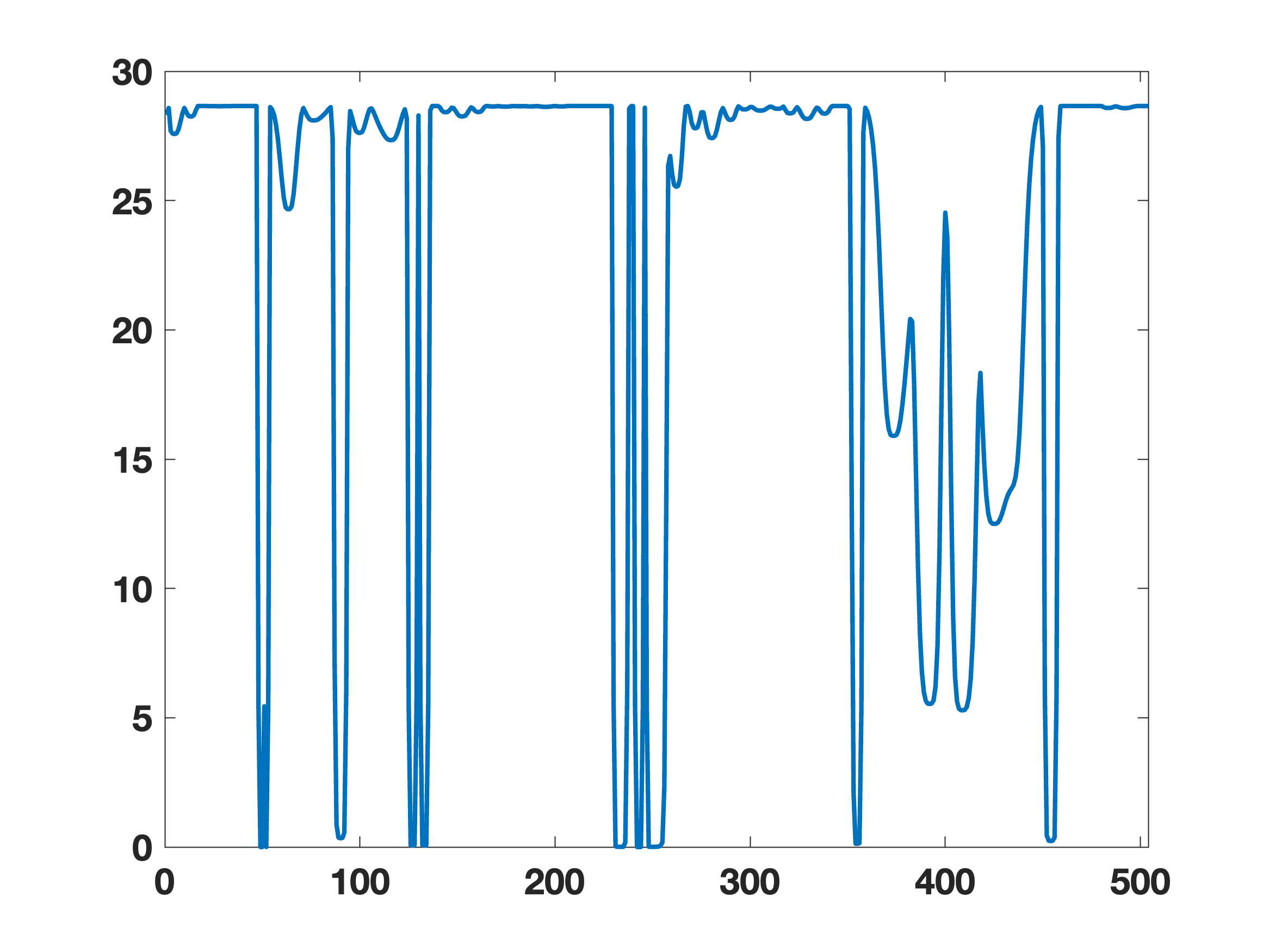}
\end{center}
\caption{1D test problems ($\delta = 0.01$). Regularization parameters computed by G\name{}. Left {\tt T1}, center {\tt T2}, right {\tt T3}.}
 \label{fig:regpar}
\end{figure}
The last two rows of Table \ref{tab:compar35_minerr} compare the relative errors when utilizing the Tikhonov method with optimal regularization parameter.  Figure \ref{fig:tikh} illustrates that a single parameter, even if optimally weighted, is unable to reconstruct equally well all the details of complex signals presenting many different features (e.g. {\tt T2}, {\tt T3}).
\begin{figure}[htbp]
\begin{center}
\includegraphics[width=4cm]{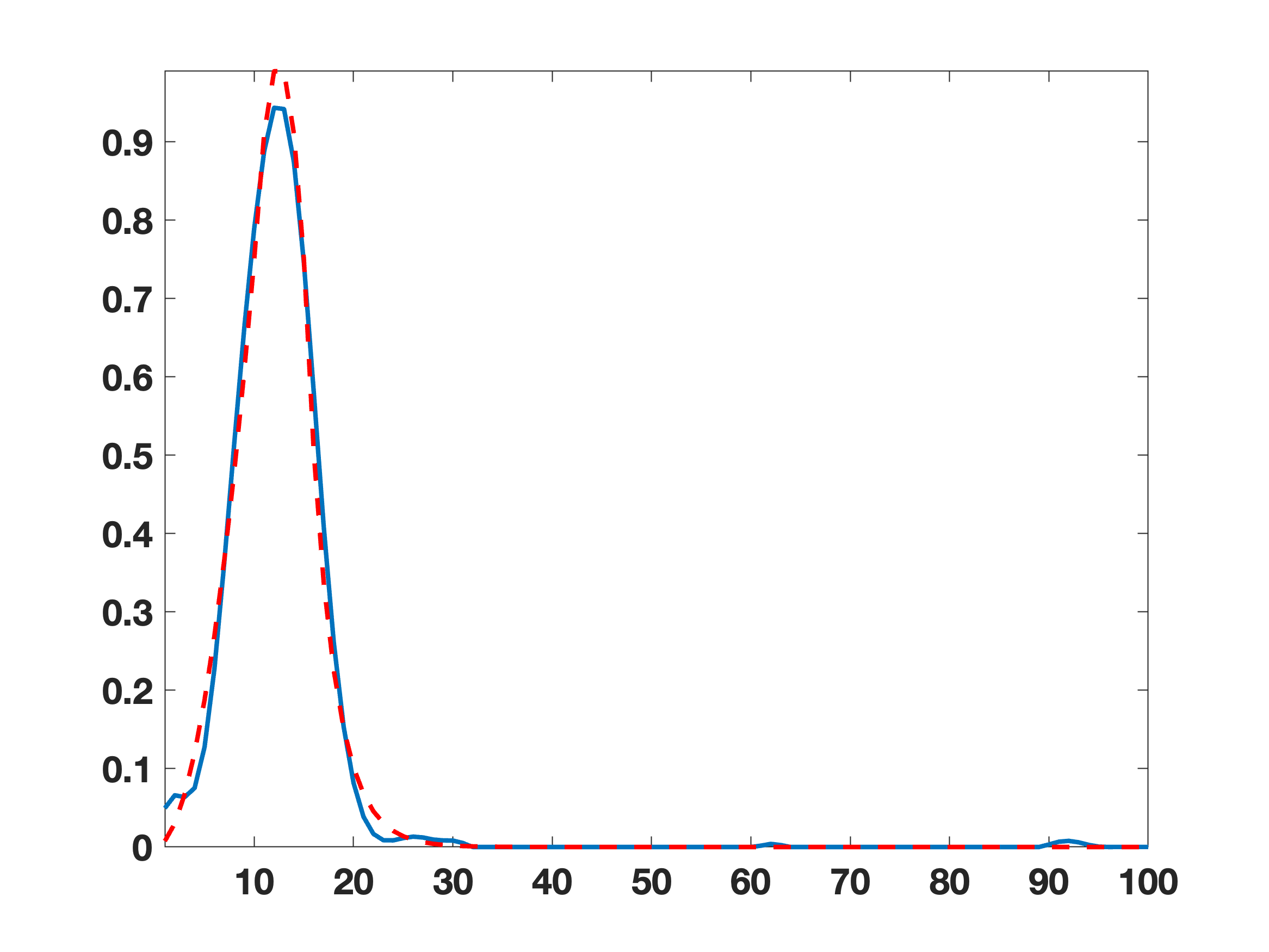}
\includegraphics[width=4cm]{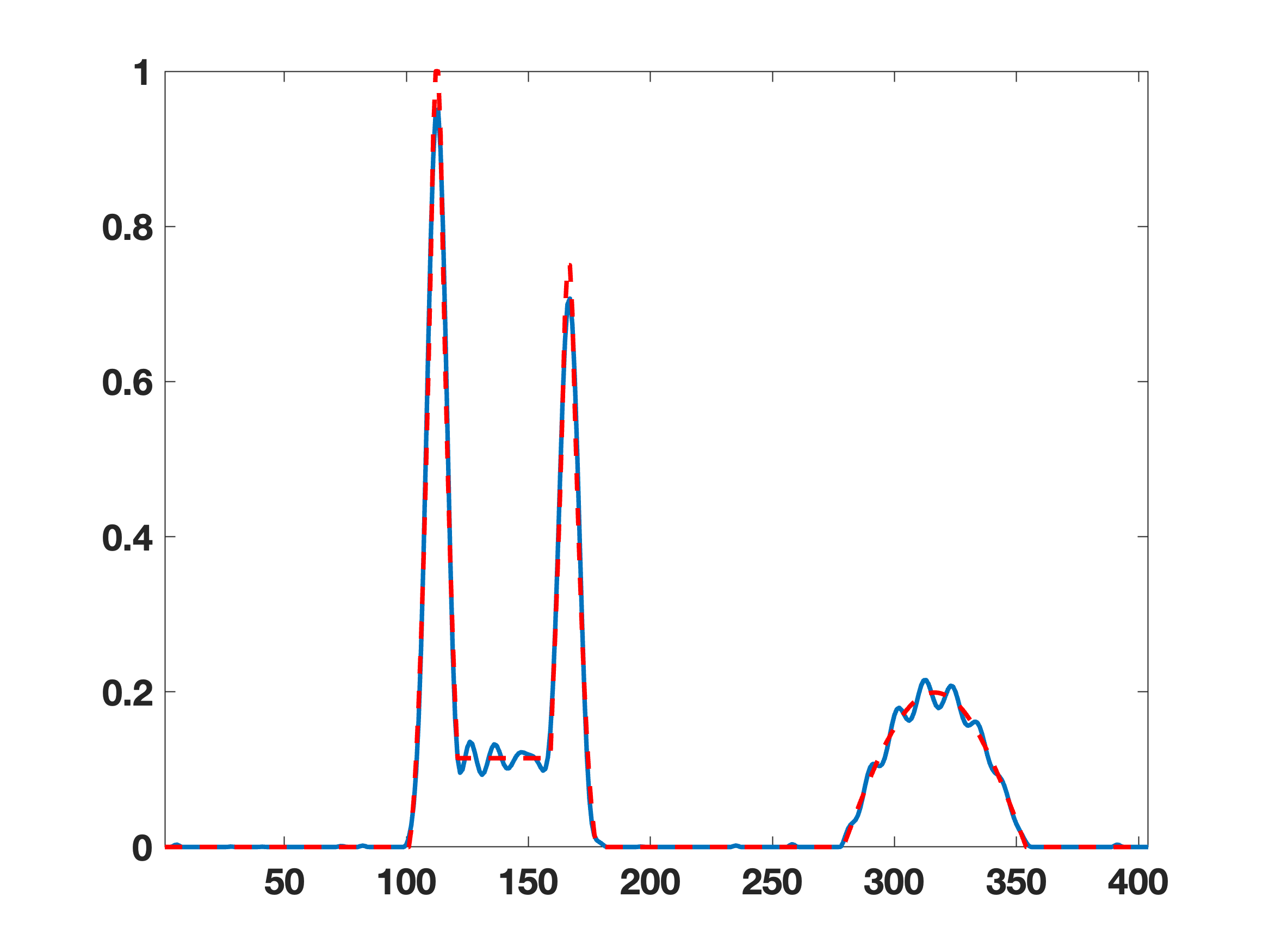}
\includegraphics[width=4cm]{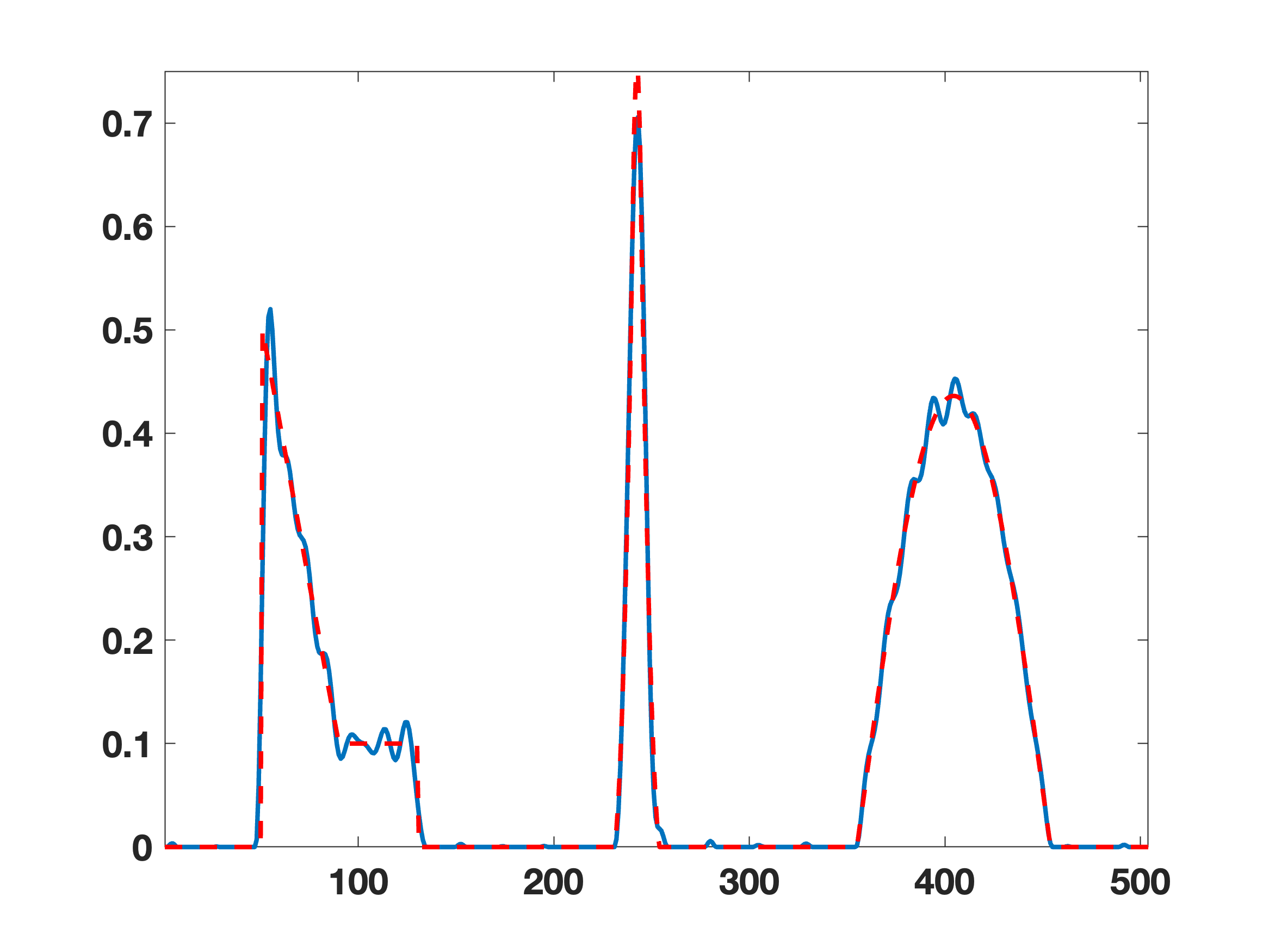}
\end{center}
\caption{1D test problems ($\delta = 0.01$). Method  $L2+$ with optimal regularization parameter.
Left {\tt T1}, $\lambda=2.7838 \times 10^{-5}$. Center {\tt T2}, $\lambda=3.495495 \times 10^{-3}$.  Right {\tt T3}, $\lambda=3.242424 \times 10^{-3}$.}
\label{fig:tikh}
\end{figure}
\section{High-noise results}\label{sec:HN}
In this section, we present the results obtained with high-noise ($\delta=0.1$), restricted to the case $\Omega=\R^N_+$, because the absence of the non-negativity constraint produces worse results. 
Figure \ref{fig:HN_tests} shows the blurred noisy data $\bb$ (blue line), and the blurred data $\by$.
%
\begin{figure}[htbp]
\begin{center}
\includegraphics[width=4cm]{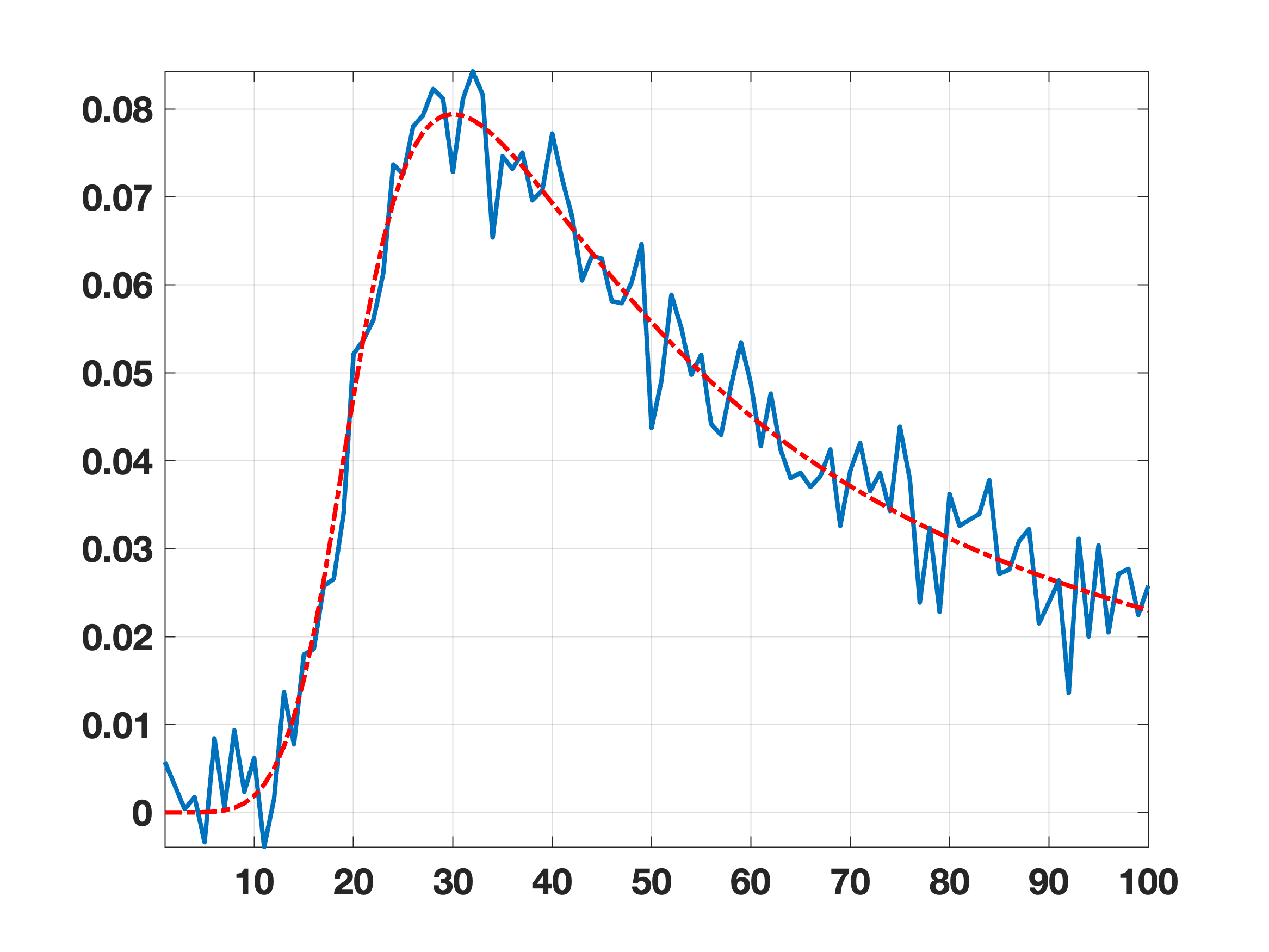}
\includegraphics[width=4cm]{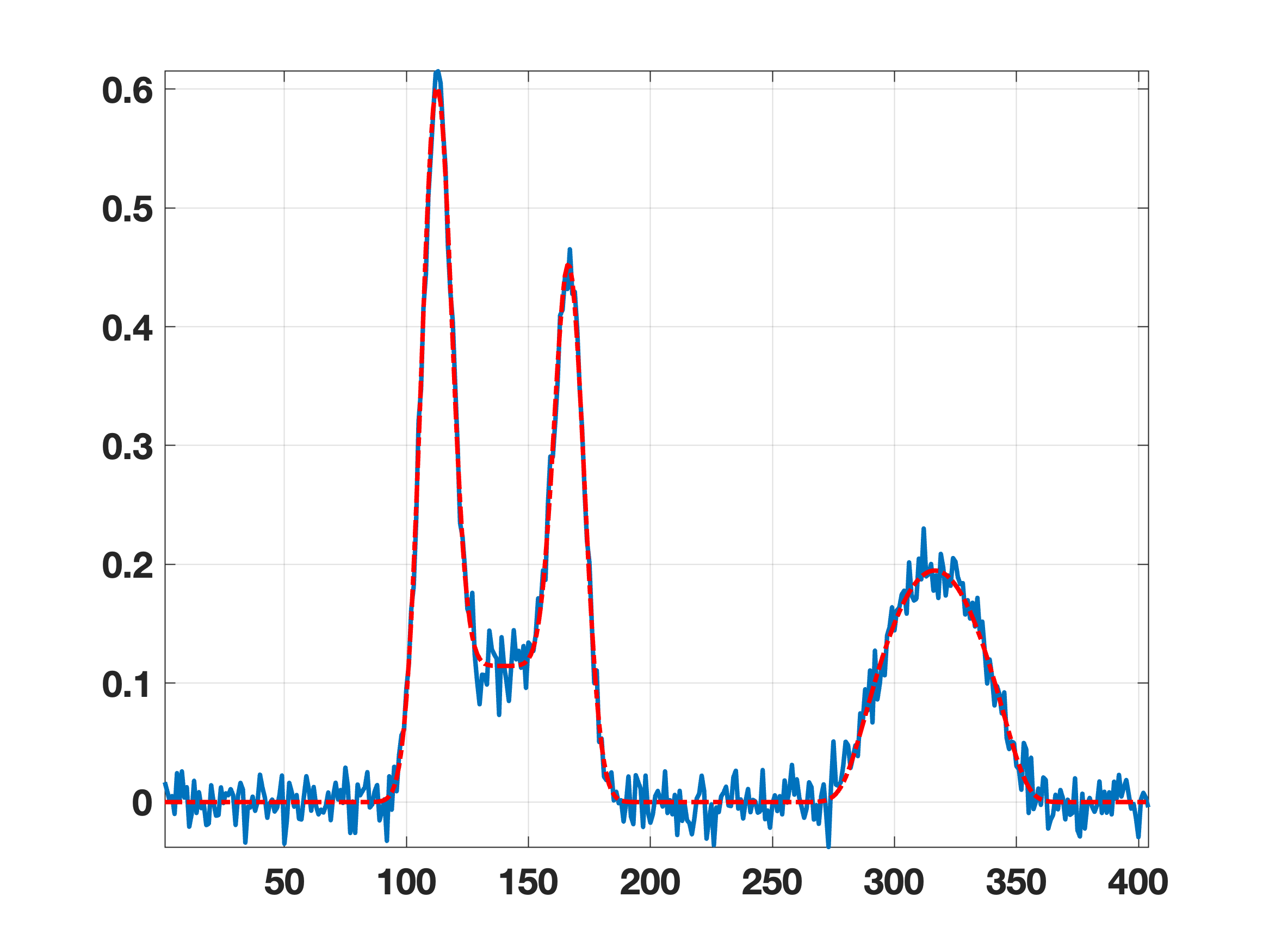}
\includegraphics[width=4cm]{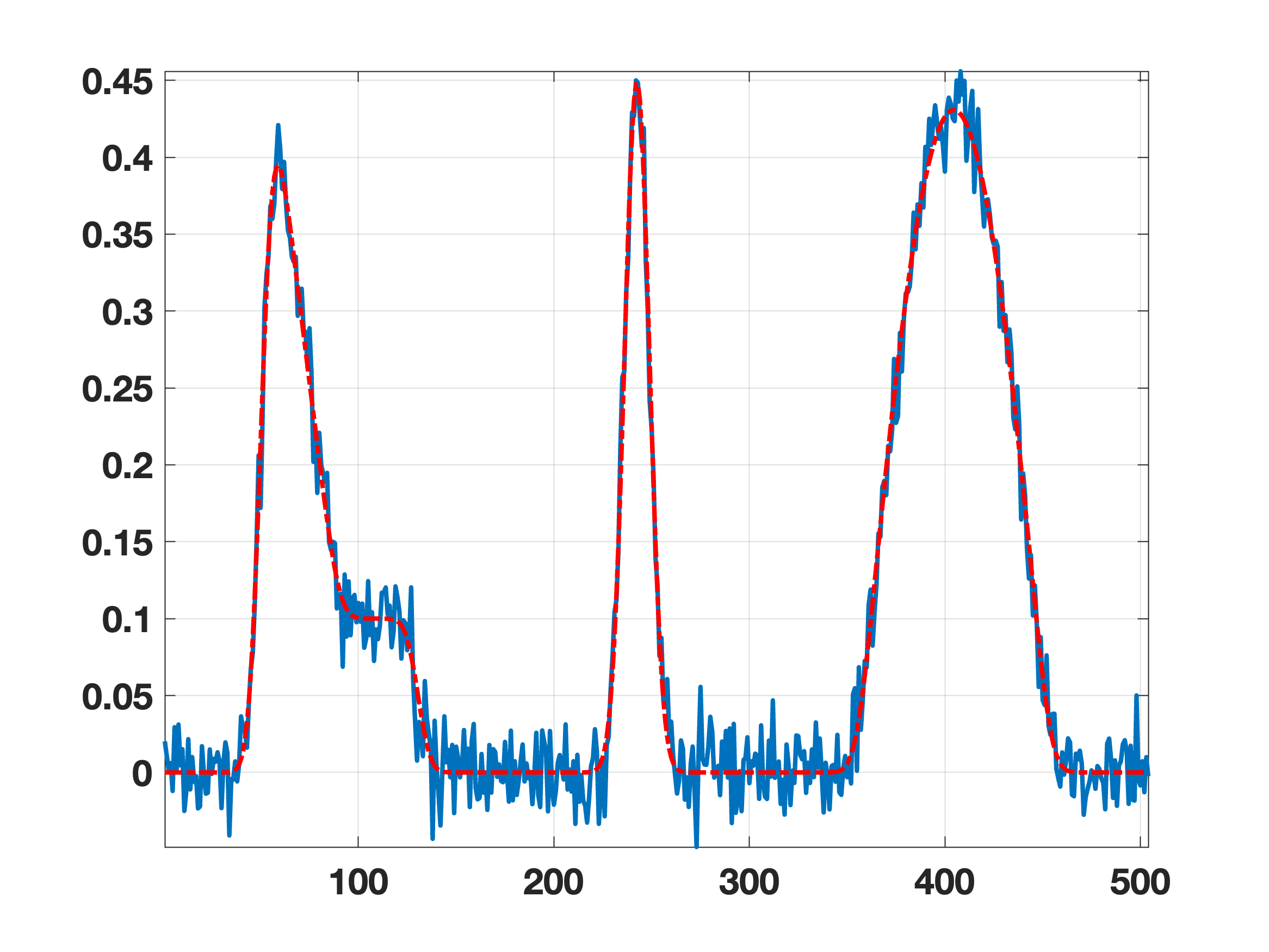}
\end{center}
\caption{1D test problems ($\delta = 0.1$). Blue line: blurred noisy data $\bb$, dashed red line: blurred data $\by$. Left {\tt T1}, center {\tt T2}, right {\tt T3}.}
 \label{fig:HN_tests}
\end{figure}

%
%
\begin{figure}[htbp]
\begin{center}
(a) \hspace{3cm} (b) \hspace{3cm} (c)\\
\includegraphics[width=4cm]{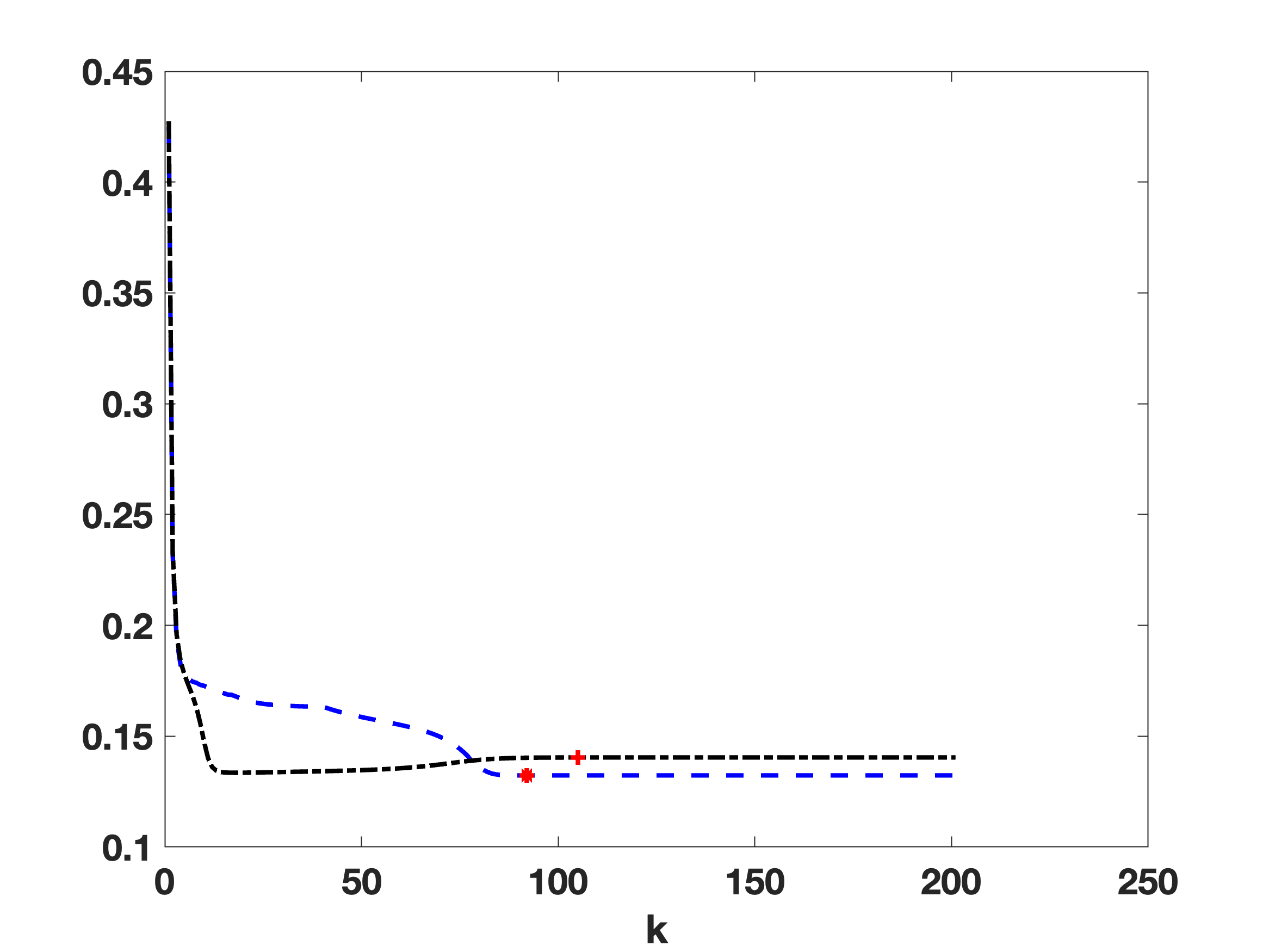}
\includegraphics[width=4cm]{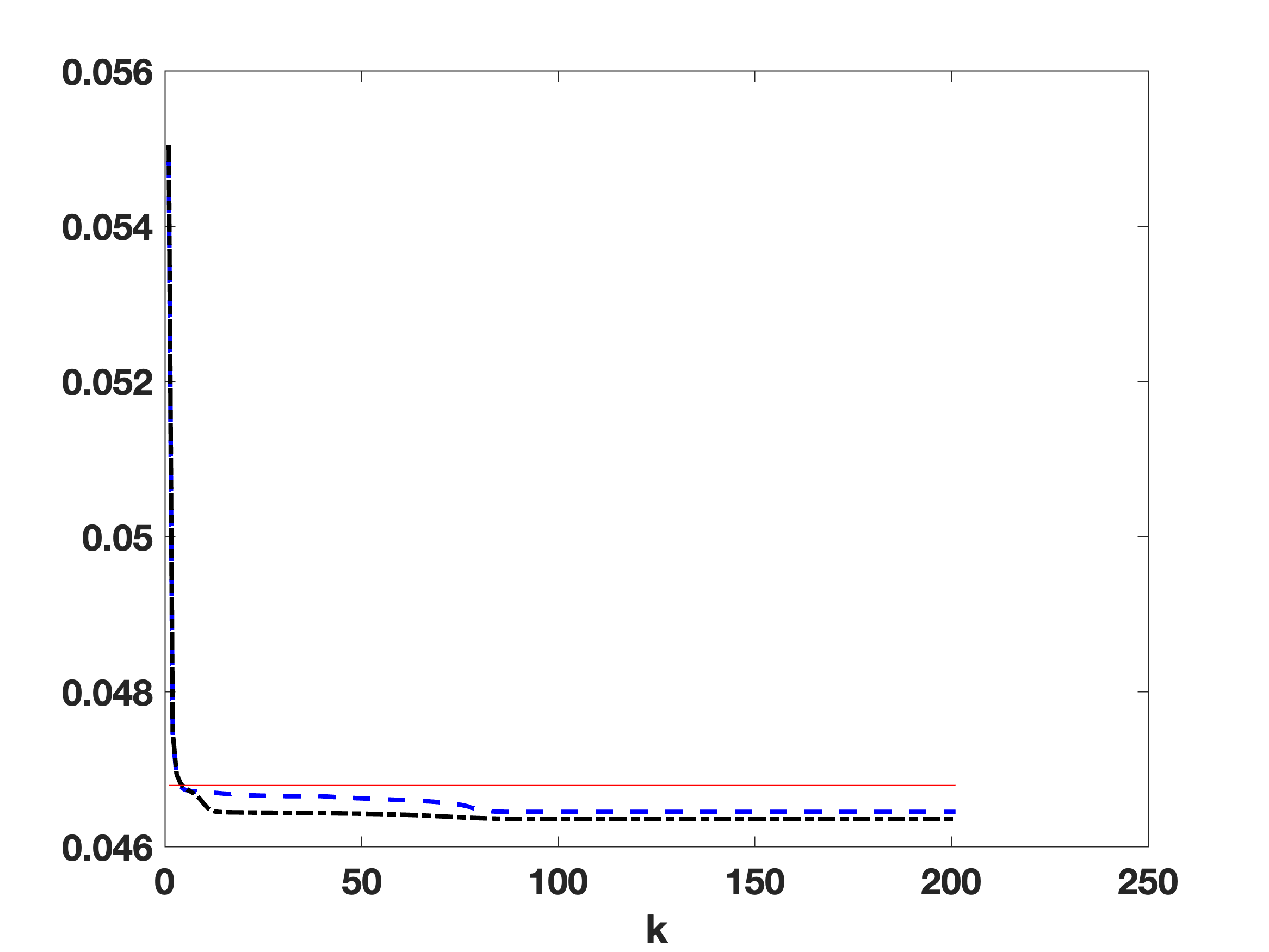	}
\includegraphics[width=4cm]{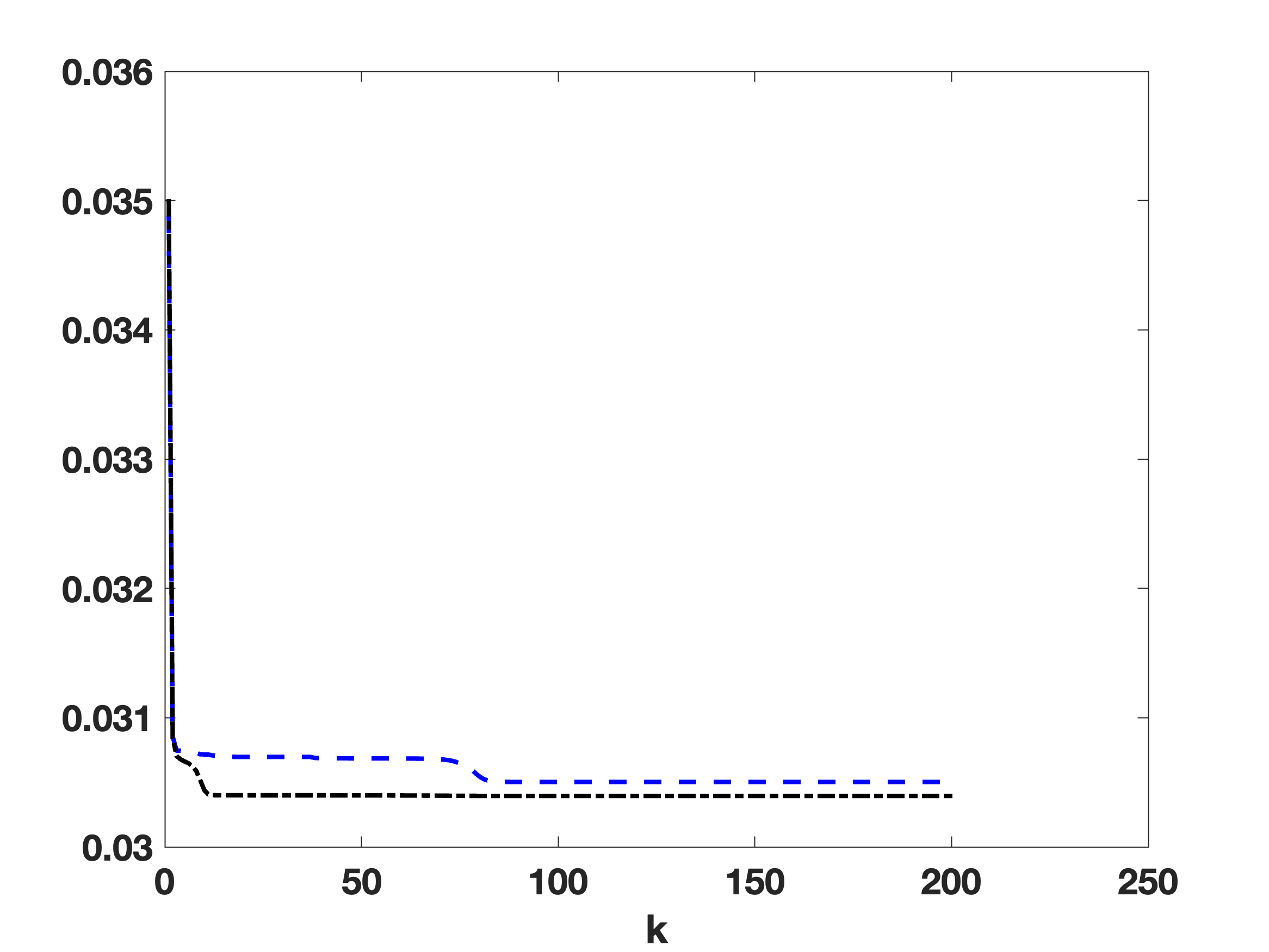}
\end{center}
\caption{T1 test problem ($\delta = 0.1$). (a) Relative error curves: Algorithm UPenMM (black dash-dot line) and Algorithm GUPenMM (blue dashed line). Red dots indicate values at the exit criterion \eqref{eq:SC}. (b) Residual norm behaviour: Algorithm UPenMM (black dash-dot line), Algorithm GUPenMM (blue dashed line), and noise norm (red line). (c) Surrogate function values: $Q(\widehat{\bl}^{(k+1)},\bl^{(k)})$ for Algorithm UPenMM (black dash-dot line) and $Q(\widetilde{\bl}^{(k+1)},\bl^{(k)})$ for Algorithm GUPenMM (blue dashed line).}
 \label{fig:T1_HN_curves}
\end{figure}
In Figure \ref{fig:T1_HN_curves}(a), we represent the relative error curves of Algorithm UPenMM (black dash-dot line) and Algorithm GUPenMM (blue dashed line) for T1. The red dots represent the relative error values corresponding to the {\tt stopping criterion} \eqref{eq:SC} with $Tol_{\lambda}=10^{-5}$. Compared to the low-noise case, a smaller value of the tolerance  is necessary to stop the algorithms close to their limit. While the error of Algorithm UPenMM decreases more rapidly in the initial iterations, Algorithm GUPenMM achieves a slightly smaller error in the limit. 
The good agreement of the residual curves with the noise norm is observed in Figure \ref{fig:T1_HN_curves}(b) where the black dash-dot line represents the residual norm values in Algorithm UPenMM,  the blue dashed line indicates the residual norm values in Algorithm GUPenMM and the red line represents the noise norm. 
The plot of the values of the surrogate function is shown in Figure \ref{fig:T1_HN_curves}(c).
\begin{figure}[htbp]
\begin{center}
(a) \hspace{5cm} (b) \\
\includegraphics[width=6cm]{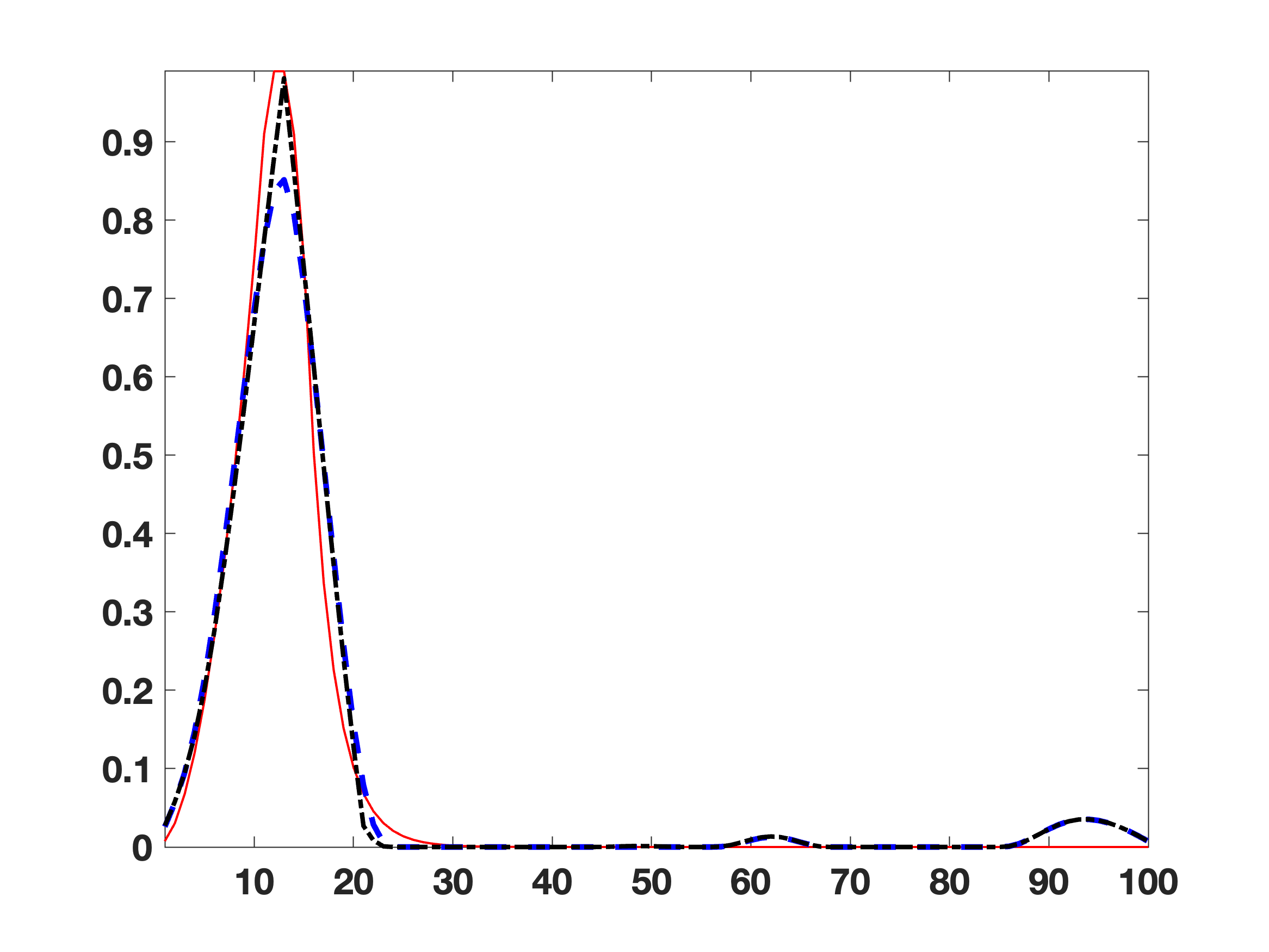}\hspace{-.6cm}
   \includegraphics[width=6cm]{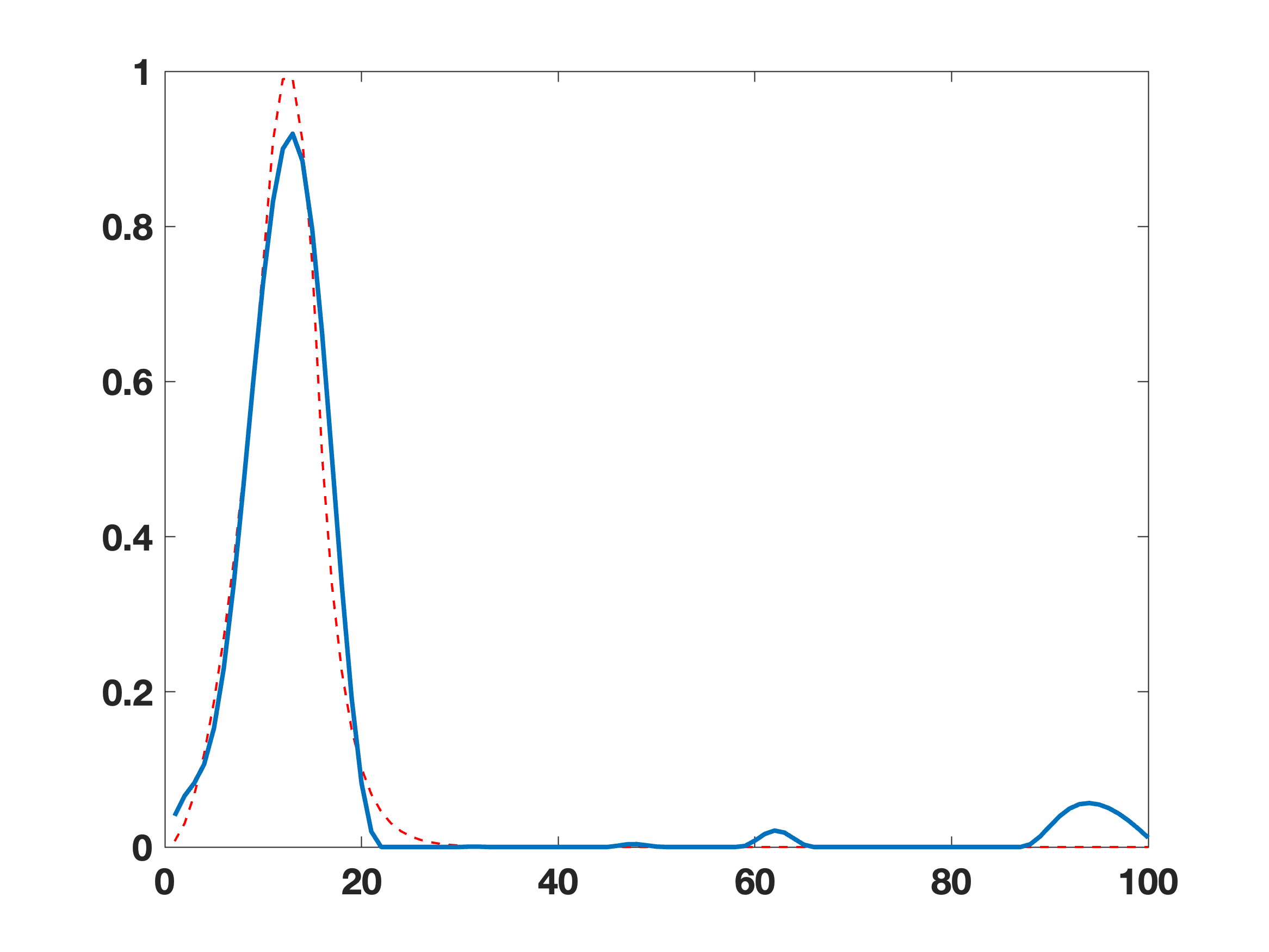}
\end{center}
\caption{Test problem {\tt T1} ($\delta = 0.1$). Solution $\bu$ with  $\Omega= \R^N_+$. (a) $\bu$ computed by Algorithm UPenMM (black dash-dot line), $\bu$ computed by Algorithm GUPenMM (blue dashed line) and ground-truth solution (red dashed line). (b) Solution computed by Tikhonov $L2 +$ (blue line)  with optimal parameter  $\lambda=1.5167 \ 10^{-3}$, and ground-truth solution (red dashed line).}
\label{fig:T1_sol_HN}
\end{figure}
Figure \ref{fig:T1_sol_HN}(a) shows the good reproduction of the high peak obtained from both algorithms,
with spurious oscillations in the flat region that, however, they are less pronounced compared to the Tikhonov $L2+$ solution with optimal regularization parameter (Figure \ref{fig:T1_sol_HN}(b)).
\begin{table}[h!]
{\small
\begin{center}
\begin{tabular}{|l|c|l|c|l|c|l|}
\hline
 \multirow{ 2}{*}{Method}   & \multicolumn{2}{|c|}{\texttt{T1}} & \multicolumn{2}{|c|}{\texttt{T2}} & \multicolumn{2}{|c|}{\texttt{T3}}\\
 \cline{2-7}
              & Rel Err  &   Iters          & Rel Err  &   Iters           & Rel Err  &   Iters  \\
\hline
 UPenMM + & 1.401$\cdot 10^{-1}$ & 105(169) &  1.081$\cdot 10^{-1}$ & 14(72) &  1.093$\cdot 10^{-1}$ & 15(89) \\
 GUPenMM + &  1.323$\cdot 10^{-1}$ & 92(171) &  8.068$\cdot 10^{-2}$ & 84(241) & 1.072$\cdot 10^{-1}$ & 361(825) \\
  L2  +    & 1.427$\cdot 10^{-1}$ & /  & 1.019$\cdot 10^{-1}$ & / &  1.270$\cdot 10^{-1}$ & / \\
\hline
\end{tabular}
\vspace*{3mm}
\caption{1D test problems ($\delta = 0.1$). Relative error and the corresponding number of iterations with $Tol_{\lambda}=10^{-5}$. The symbol $+$ indicates the constrained case. Between parenthesis, the number of Newton Projection iterations is reported. \label{tab:HN_errs}}
\end{center}
}
\end{table}
A similar behaviour is observed in tests T2 and T3, as depicted in Figures \ref{fig:T2_sol_HN} and \ref{fig:T3_sol_HN}. Table \ref{tab:HN_errs} shows that the smallest relative error is obtained by Algorithm GUPenMM (row GUPenMM $+$) in each case and that Tikhonov $L2 +$ produces the worst results. 
\begin{figure}[h!]
\begin{center}
(a) \hspace{5cm} (b) \\
\includegraphics[width=6cm]{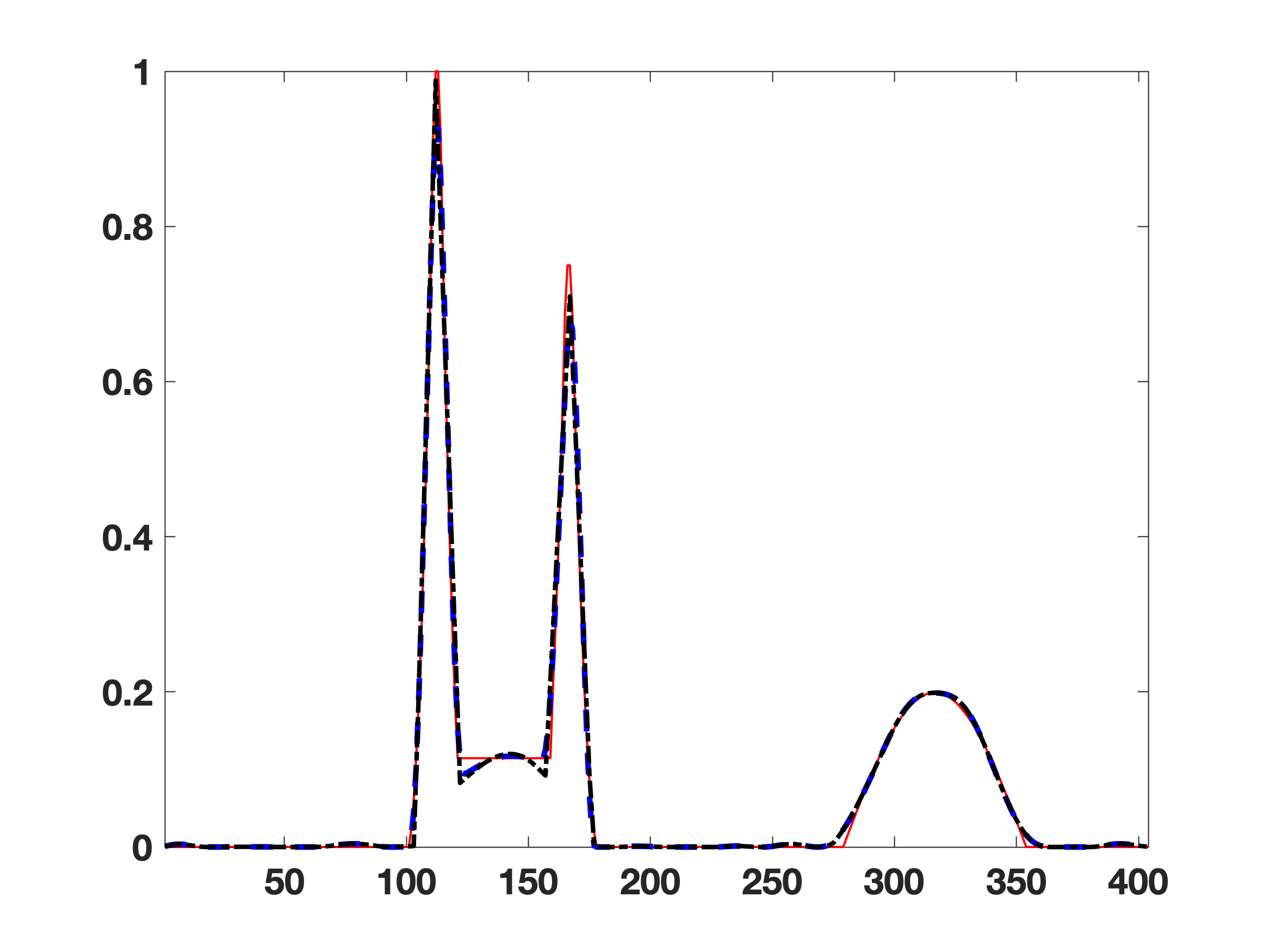}\hspace{-.6cm}
   \includegraphics[width=6cm]{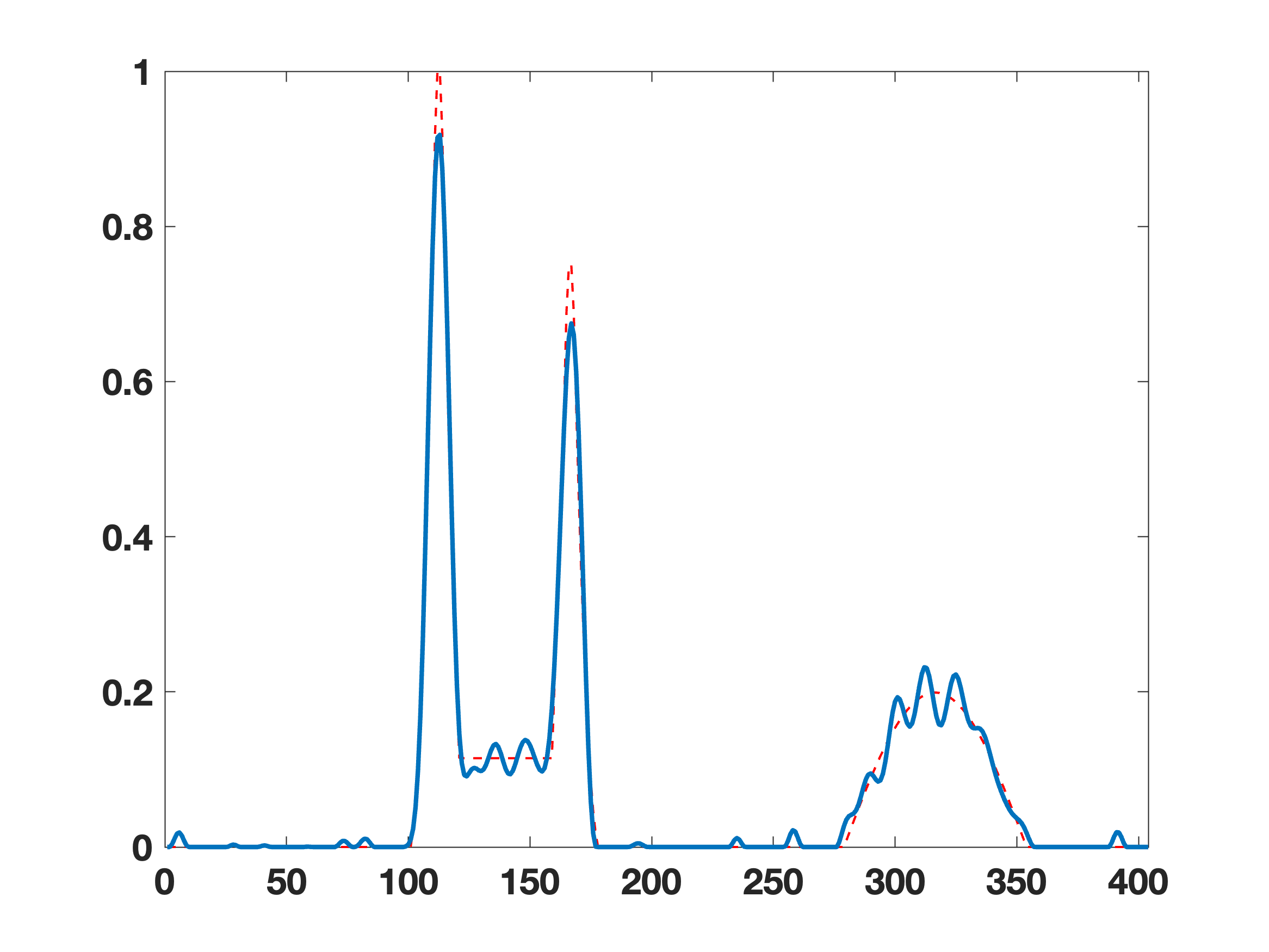}
\end{center}
\caption{Test problem {\tt T2} ($\delta = 0.1$). Solution $\bu$ with  $\Omega= \R^N_+$. (a) $\bu$ computed by Algorithm UPenMM (black dash-dot line), $\bu$ computed by Algorithm GUPenMM (blue dashed line) and ground-truth solution (red dashed line). (b) Solution computed by Tikhonov $L2 +$   (blue line) with  optimal parameter $\lambda=7.2326 \ 10^{-2}$, and ground-truth solution (red dashed line).}
\label{fig:T2_sol_HN}
\end{figure}
\begin{figure}[h!]
\begin{center}
(a) \hspace{5cm} (b) \\
\includegraphics[width=6cm]{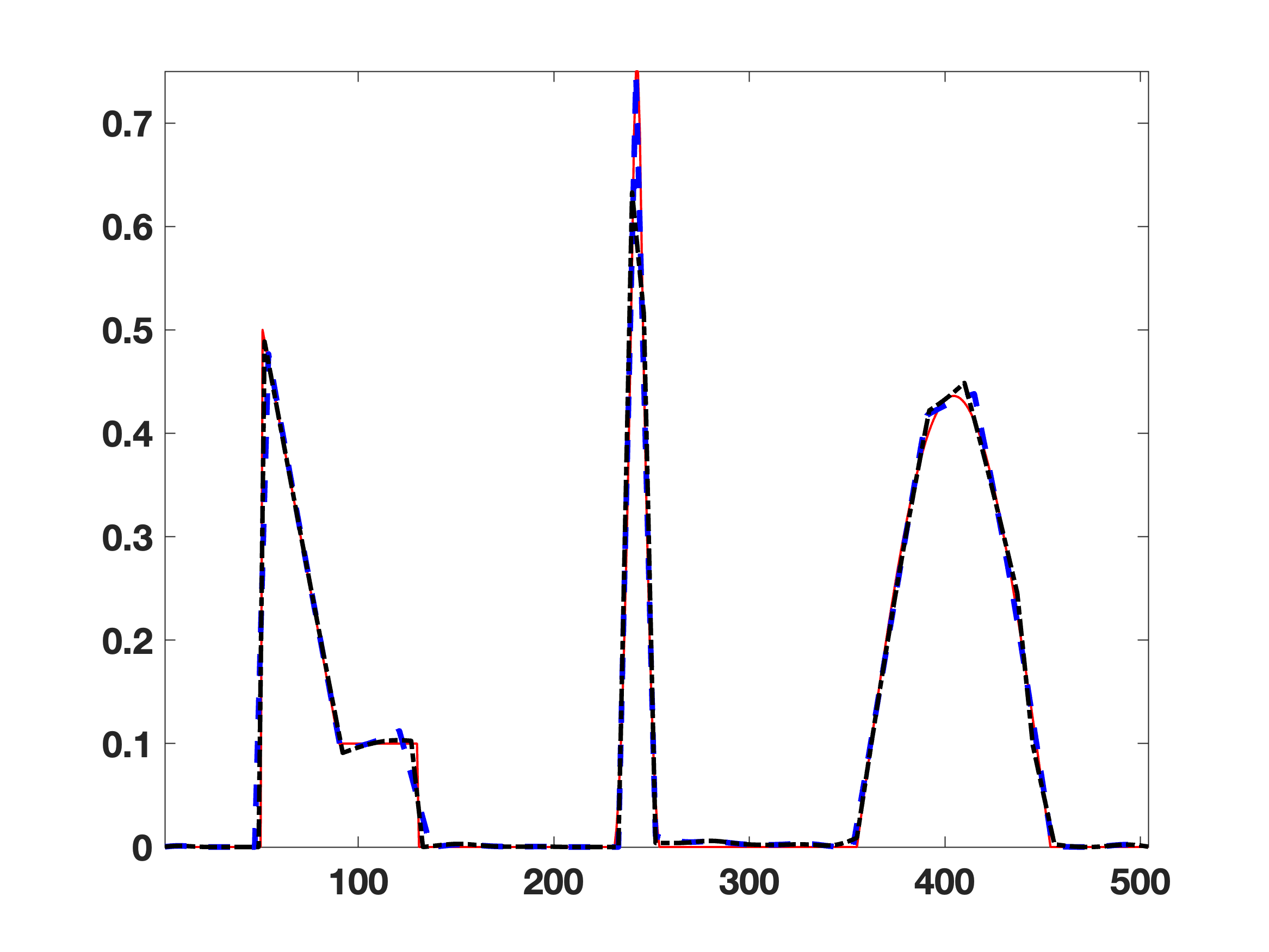}\hspace{-.6cm}
   \includegraphics[width=6cm]{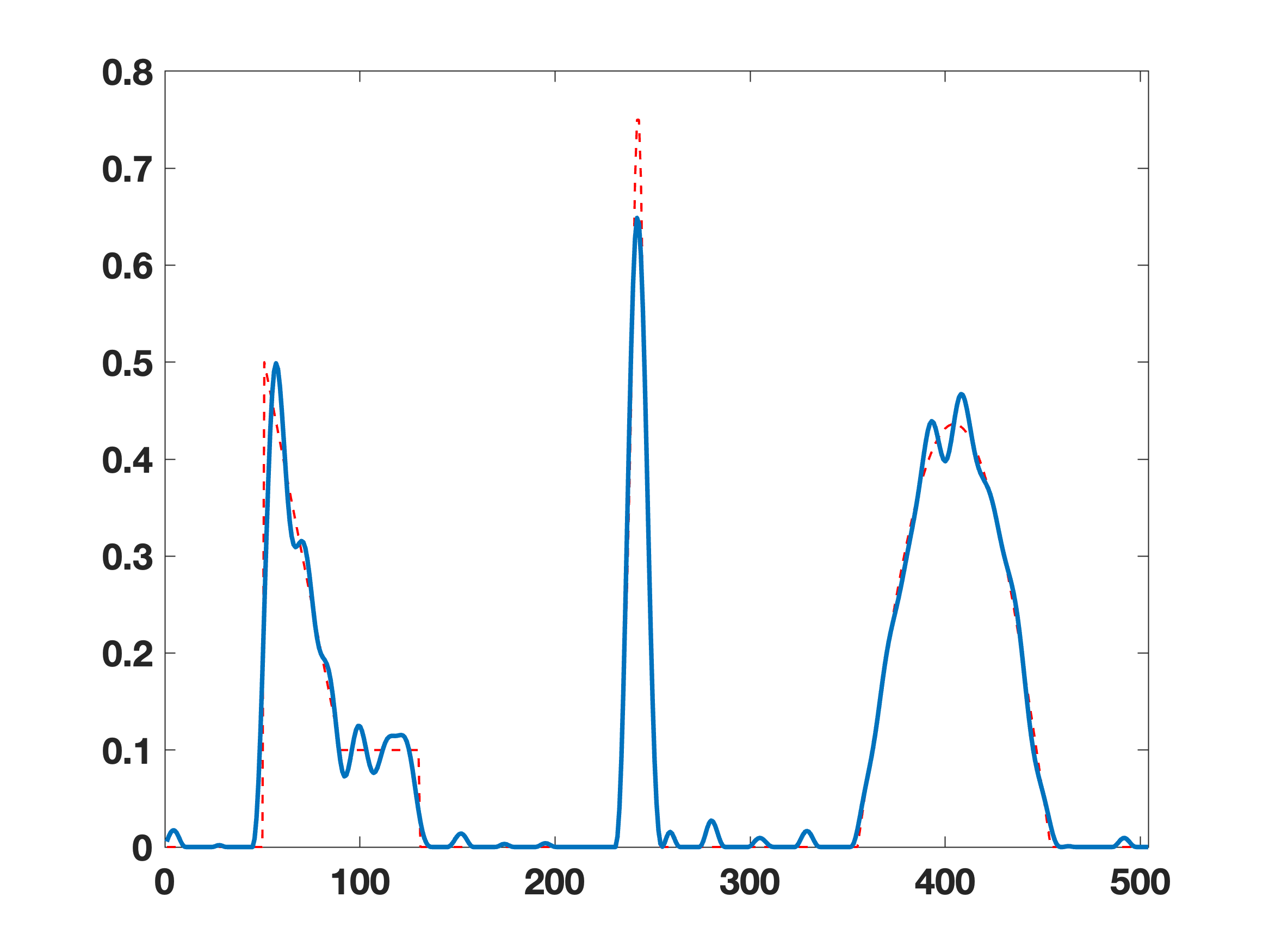}
\end{center}
\caption{Test problem {\tt T3} ($\delta = 0.1$). Solution $\bu$ with  $\Omega= \R^N_+$. (a) $\bu$ computed by Algorithm UPenMM (black dash-dot line), $\bu$ computed by Algorithm GUPenMM (blue dashed line) and ground-truth solution (red dashed line). (b) Solution computed by Tikhonov $L2 +$ (blue line) with optimal parameter $\lambda=6.0103 \ 10^{-1}$, and ground-truth solution (red dashed line).}
\label{fig:T3_sol_HN}
\end{figure}
Figure \ref{fig:HNregpar} represents the point-wise regularization parameter, confirming the characteristics already observed in the low-noise case.
\begin{figure}[h!]
\begin{center}
\includegraphics[width=4cm]{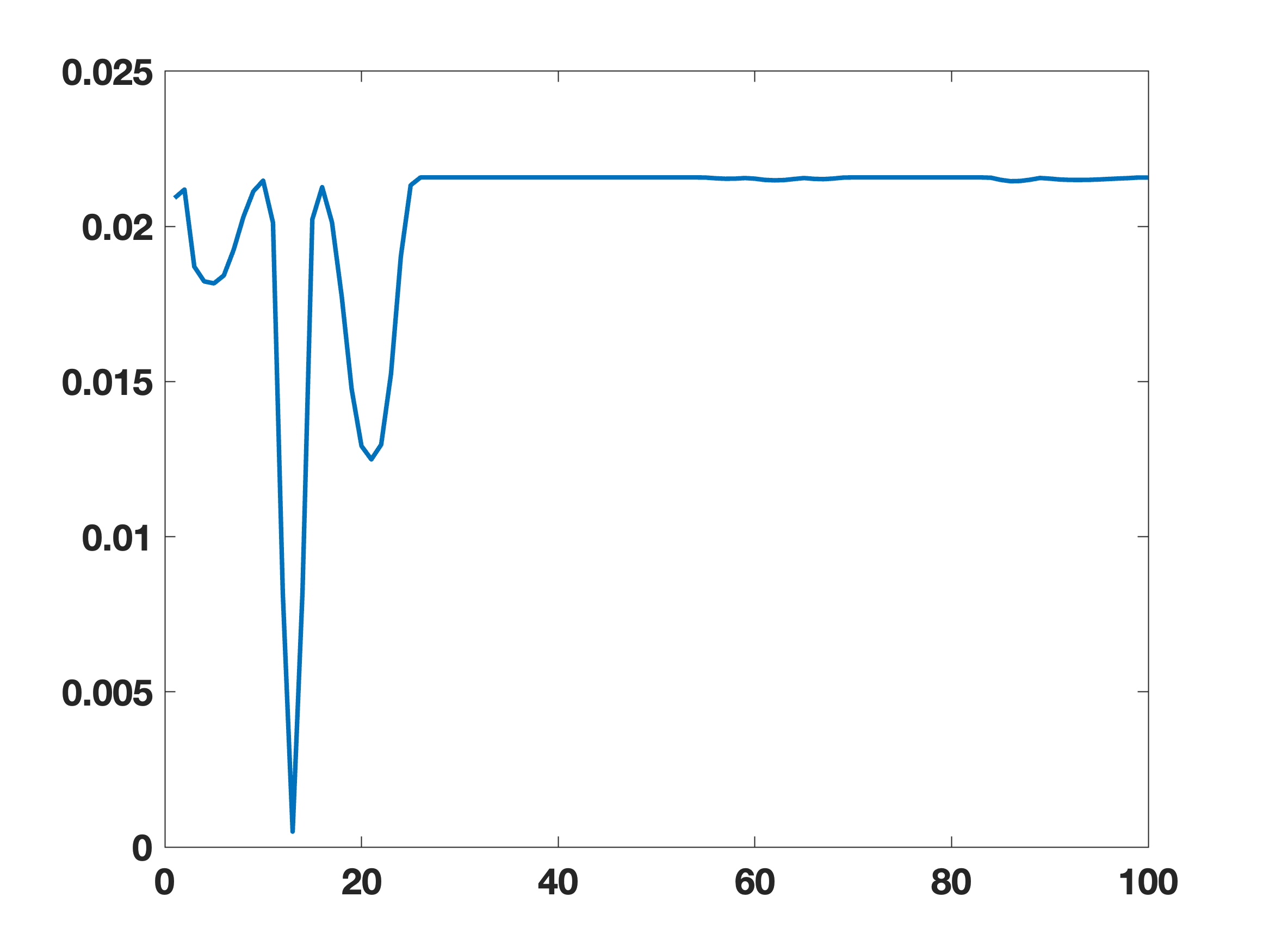}
\includegraphics[width=4cm]{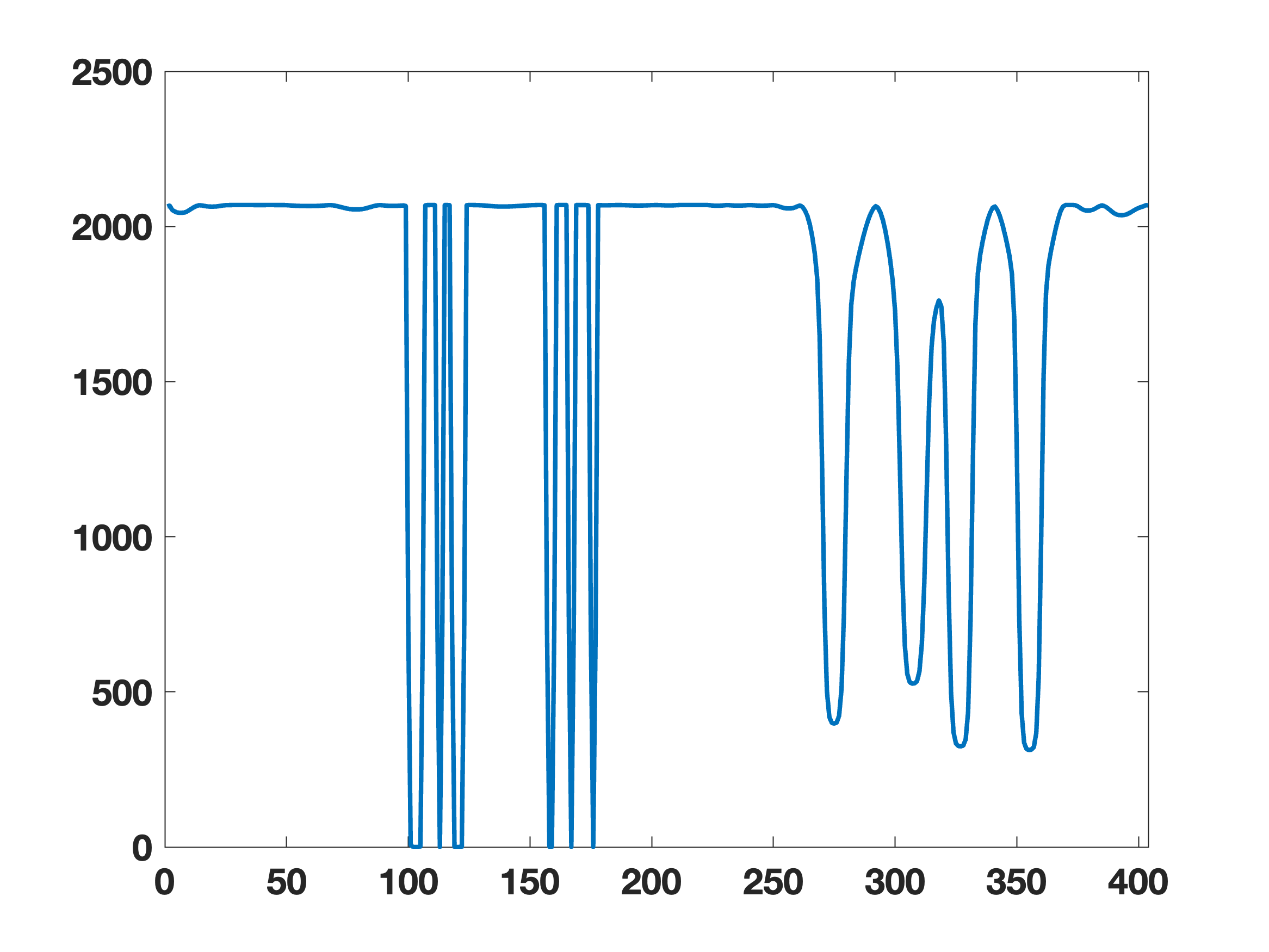	}
\includegraphics[width=4cm]{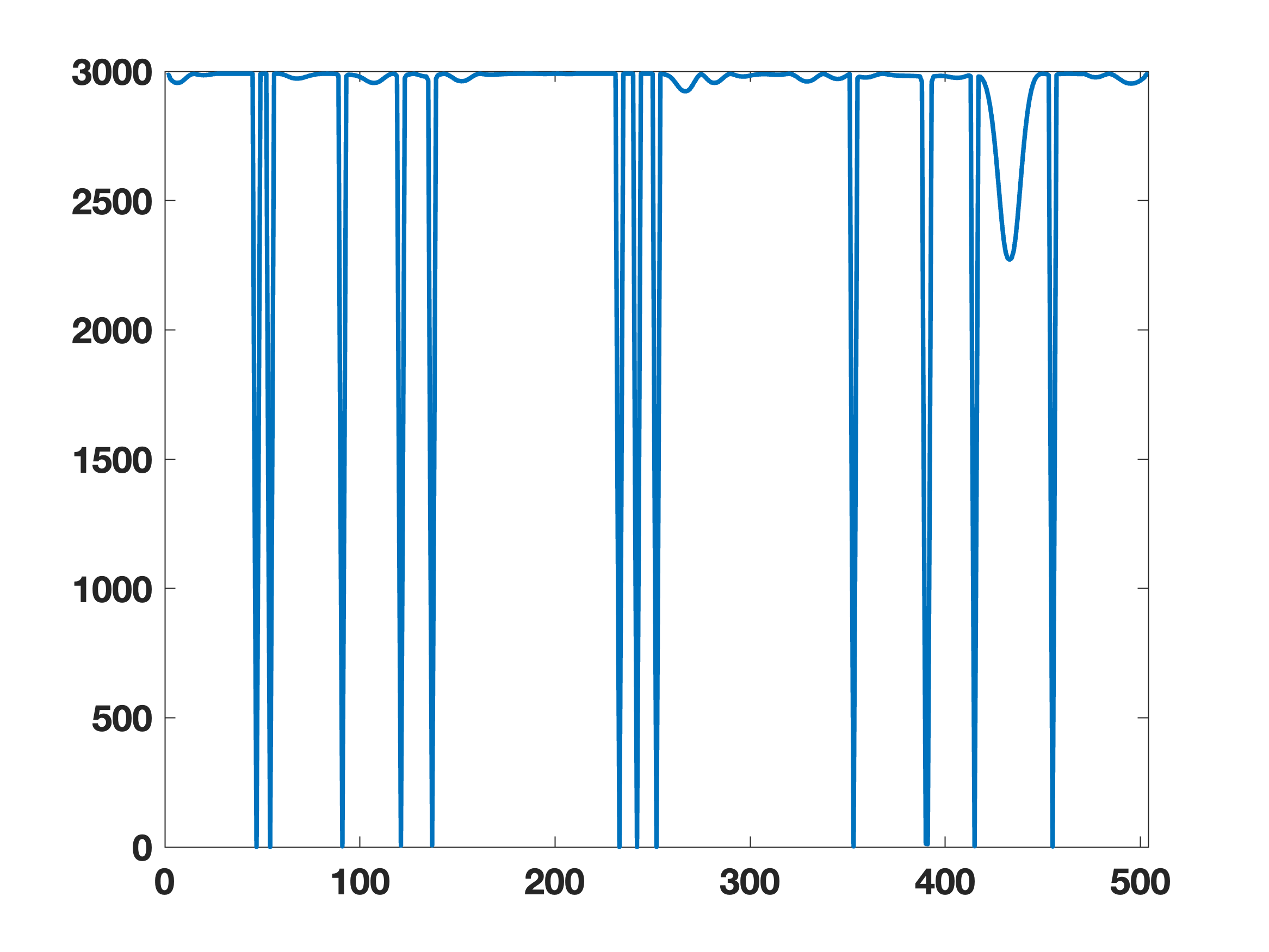	}
\end{center}
\caption{1D test problems ($\delta = 0.1$). Regularization parameters computed by G\name{}. Left {\tt T1}, center {\tt T2}, right {\tt T3}.}
 \label{fig:HNregpar}
\end{figure}
Finally, concerning the computation times reported in Table \ref{tab:HNcomTimes}, we observe higher values compared to the low-noise case (Table \ref{tab:comTimes}) due to the use of a smaller tolerance  ($Tol_{\lambda}=10^{-5}$ \emph{vs} $Tol_{\lambda}=10^{-2}$ with low noise).  Additionally, GUPenMM $+$ is more expensive compared to UPenMM $+$.
\begin{table}[h!]
{\small
\begin{center}
\begin{tabular}{|l|ccc|}
\hline
 \multirow{ 2}{*}{Method}   & \multicolumn{3}{|c|}{Computation times}\\
 \cline{2-4}
              & T1  & T2  & T3  \\
\hline
 UPenMM + & 5.61$\cdot 10^{-2}$  & 3.02$\cdot 10^{0}\;\;$ & 5.58$\cdot 10^{0}\;\;$ \\
 GUPenMM +& 7.57$\cdot 10^{-2}$ & 9.36$\cdot 10^{0}\;\;$ &  4.50$\cdot 10^{1}\;\;$\\
\hline
\end{tabular}
\vspace*{3mm}
\caption{1D test problems ($\delta = 0.1$). Computation times in seconds.  The symbol $+$ indicates the constrained case. 
\label{tab:HNcomTimes}}
\end{center}}
\end{table}